\documentclass[11pt, reqno]{amsart}
\usepackage{multirow}
\usepackage{tikz}
\usepackage{caption}
\usetikzlibrary{
 decorations.pathmorphing,%
 decorations.markings,%
 decorations.pathreplacing,%
 decorations.text,%
 calc,%
 patterns,%
 shapes.geometric,%
 arrows,
 decorations.shapes,
 positioning,
 plotmarks,
 shadings,
 math,
 intersections,
 external
}

\definecolor{jeffColor}{RGB}{102, 0, 204}
\definecolor{yaizaColor}{RGB}{0, 153, 153}

\definecolor{periodColor}{RGB}{255, 167, 105}
\definecolor{dark-green}{RGB}{135, 194, 130}
\tikzset{>=latex} 
\tikzset{font=\small}
\tikzset{mark size=1.5pt, mark options=thin}
\tikzset{pin distance=4pt,
 every pin edge/.style={<-, thin, shorten <= -2pt}}
\usepackage{array,booktabs,tabularx}
\usepackage{mathrsfs}  
\usepackage{graphicx}
\usepackage{amssymb}
\usepackage{enumerate}
\usepackage{color}
\usepackage{esint}
\usepackage{mathtools}
\usepackage{dsfont}
\usepackage{ esint }
\usepackage{placeins}
\usepackage[final]{ed}
\usepackage[final]{showkeys}

\usepackage{comment}
\usepackage{enumitem}
\usepackage{soul}
\usepackage{needspace}
\usepackage[normalem]{ulem}

\usepackage{multicol}
\usepackage{tabto}
\usepackage{hyperref}

\newtheorem{lemma}{Lemma}
\newtheorem{theorem}[lemma]{Theorem}
\newtheorem{example}[lemma]{Example}

\newtheorem{corollary}[lemma]{Corollary}
\newtheorem{proposition}[lemma]{Proposition}

\newtheorem{conjecture}[lemma]{Conjecture}

\theoremstyle{definition}
\newtheorem{definition}[lemma]{Definition}
\newtheorem{remark}[lemma]{Remark}

\usepackage{dsfont}

\newcommand{\1}{\mathds{1}}

\newcommand{\R}{{\mathbb R}}

\newcommand{\sub}[1]{_{_{#1}}}

\newcommand{\h}{\hbar}

\newcommand{\mc}[1]{\mathcal{#1}}
\newcommand{\re}{\mathbb{R}}

\newcommand{\lp}{\mu}
\newcommand{\Op}[1]{\operatorname{Op}^{\!W}_{#1}}

\newcommand{\s}{SN}

\newcommand{\comp}{\operatorname{comp}}

\newcommand{\step}{\mathfrak{s}_n}

\newcommand{\w}{\omega}

\newcommand{\Sp}{S_{\operatorname{phg},\delta}}
\newcommand{\Spz}{S_{\operatorname{phg},0}}

\def\XXint#1#2#3{{\setbox0=\hbox{$#1{#2#3}{\int}$} \vcenter{\hbox{$#2#3$}}\kern-.5\wd0}}

\DeclareMathOperator{\vol}{vol}

\DeclareMathOperator{\diam}{diam}

\DeclareMathOperator{\supp}{supp}

\newcommand{\e}{\varepsilon}

\numberwithin{equation}{section}
\numberwithin{lemma}{section}
\renewcommand{\mathfrak}[1]{\mathbf{#1}}
\newcommand{\sem}[1]{\mathfrak{#1}}

\newcommand{\fcomp}{\operatorname{mcomp}}

\newcommand{\ad}{\operatorname{ad}}
\newcommand{\WF}{\operatorname{WF}}
\newcommand{\Diff}{\operatorname{Diff}}
\newcommand{\semOp}[1]{\sem{H}(#1)}
\newcommand{\normOp}[1]{H(#1)}

\newcommand {\bx}{\mathbf x}

\newcommand {\be}{\mathbf e}

\newcommand {\by}{\mathbf y}

\newcommand {\bth}{\boldsymbol\theta}

\newcommand {\bxi}{\boldsymbol\xi}

\newcommand{\bee}{\begin{equation}}
\newcommand{\ene}{\end{equation}}
\newcommand{\bet}{\begin{theorem}}
\newcommand{\ent}{\end{theorem}}
\newcommand{\QM}{{}^{^{\mc {M}}}\!V}

\newcommand{\QG}{{}^{^{\mc {G}}}\!V}
\newcommand{\vG}{\raisebox{-.2em}{${}^{^{\mc {G}}}$}\!v}
\newcommand{\HM}{{}^{^{\mc {M}}}\!\!H}
\newcommand{\HG}{{}^{^{\mc {G}}}\!\!H}
\newcommand{\N}{\mathcal{N}}

\newcommand{\cutMan}{\mathbb{X}}%
\newcommand{\cutEnergy}{\Phi}%
\newcommand{\cutEnergyFourier}{\nu}
\newcommand{\cutFreq}{\Xi}%
\newcommand{\cutPhase}{\mathbb{P}}
\newcommand{\cutThetaFreq}{\Theta}%
\newcommand{\auxCut}{f}%
\newcommand{\coA}{f}
\newcommand{\coB}{g}

\usepackage{scalerel,stackengine}
\stackMath
\renewcommand\widehat[1]{%
\savestack{\tmpbox}{\stretchto{%
  \scaleto{%
    \scalerel*[\widthof{\ensuremath{#1}}]{\kern-.6pt\bigwedge\kern-.6pt}%
    {\rule[-\textheight/2]{1ex}{\textheight}}
  }{\textheight}%
}{0.5ex}}%
\stackon[1pt]{#1}{\tmpbox}%
}

\newcommand{\scale}{\nu}

\newcommand{\USB}{C_b^\infty}

\newcommand{\referee}[1]{\blue{#1}}

\renewcommand{\referee}[1]{#1}
\renewcommand{\st}[1]{}

\title[Spectral Asymptotics in one dimension]{Classical Wave methods and modern gauge transforms:\\ Spectral Asymptotics in the one dimensional case}

\author[J. Galkowski]{Jeffrey Galkowski}
\address{}
\email{j.galkowski@ucl.ac.uk}
\address{Department of Mathematics, University College London, UK}
\author[L. Parnovski]{Leonid Parnovski}
\email{l.parnovski@ucl.ac.uk}
\address{Department of Mathematics, University College London, UK}
\author[R. Shterenberg]{Roman Shterenberg}
\email{shterenb@math.uab.edu}
\address{Department of Mathematics, University of Alabama at Birmingham, USA}

\begin{document}
\maketitle


\begin{abstract}
In this article, we consider the asymptotic behaviour of the spectral function of Schr\"odinger operators on the real line. Let $H: L^2(\mathbb{R})\to L^2(\mathbb{R})$ have the form
$$
H:=-\frac{d^2}{dx^2}+Q,
$$
where $Q$ is a formally self-adjoint first order differential operator with smooth coefficients, bounded with all derivatives. We show that the kernel of the spectral projector, $\mathds{1}_{(-\infty,\rho^2]}(H)$, has a complete asymptotic expansion in powers of $\rho$. This settles the 1-dimensional case of a conjecture made by the last two authors. 
\end{abstract}

\section{Introduction}

Consider a Schr\"odinger operator $H$ acting on $L^2(\mathbb{R})$ and given by 
\begin{equation}
\label{e:basicHA}
H:=D^2 +\referee{V},\qquad D:=-i\partial_x.
\end{equation}
We assume that \referee{the potential $V=V(x)$ is real valued, infinitely smooth and satisfies }
\begin{equation}
\label{LP:USB}
\|\partial_{x}^\alpha V\|_{L^\infty}<\infty,\qquad \alpha \in \mathbb{N}.
\end{equation}
We call any potential  $V$ satisfying condition \eqref{LP:USB} a \emph{uniformly smoothly bounded (USB)} potential and denote by $\USB(\mathbb{R})$ the class of such potentials. Let $E(H)(\rho)=\mathds{1}_{(-\infty,\rho^2]}(H)$ be the spectral projector for $H$ and $\referee{E}(H)(\rho;x,y)$ be its integral kernel (also called {\it the spectral function} of $H$). 
In this article, we study the behaviour of $\referee{E}(H)(\rho;\cdot,\cdot)$ when $\rho$ is large. One of our results is:

\begin{theorem}
\label{t:USBAsymptotics1}
Under the above assumptions, there are $\referee{f}_k\in \USB(\mathbb{R})$, $k=0,1,\dots$ such that for all $N\in \mathbb{N}$, there is $C_N>0$ such that for all $x\in \mathbb{R}$ \referee{and $\rho\geq 1$} we have
\begin{equation}
\label{LP:asymptotic}
\Big|\referee{E}(H)(\rho;x,x)- \sum_{k=0}^{N-1} \referee{f}_k(x) \rho^{1-2k}\Big|\leq C_N \rho^{1-2N}. 
\end{equation}
Here, $\referee{f}_0\equiv\frac{1}{\pi}$, and $\referee{f}_k(x)$, $k\geq 1$ can be written explicitly in terms of the derivatives of \referee{$V$} at $x$. 
\end{theorem}

We will \referee{use the} notation $\referee{E}(H)(\rho;x,x)\sim\sum_{k=0}^{\infty} \referee{f}_k(x) \rho^{1-2k}$ to indicate that the estimates \eqref{LP:asymptotic} hold. 
To compute the explicit formulae for $\referee{f}_k$, $k\geq 1$, one can take the Laplace transform of~\eqref{LP:asymptotic} as in~\cite{KoPu:03} and use the results of~\cite{Hi:02,HiPo:03,HiPo:03b} (see also~\cite[Lemma 3.63, Theorem 3.64]{DyZw:19}). 
We also obtain a complete asymptotic expansion of $\referee{E}(H)(\rho;\cdot,\cdot)$ (and its derivatives) off the diagonal, see Section~\ref{s:formulate} for a precise formulation of these results. 

Note that the spectrum of operators of the form~\eqref{e:basicHA} can have any spectral type for large energies: absolutely continuous, singular continuous (see e.g.~\cite{Si:95}), or dense pure point (see e.g.~\cite{CaLa:90}). Moreover, examples exist for which the spectrum has Lebesgue measure zero and even arbitrarily small but positive Hausdorff dimension (see e.g.~\cite{DaFiGo:21}). Despite the potentially wild behavior of the spectrum, our results show that, at high energy, the spectrum wants to be absolutely continuous\referee{; see for example Corollaries~\ref{c:1}. \ref{c:2}, \ref{LP:cor1}, and~\ref{LP:corLya}.}

Similarly to $\USB(\mathbb{R})$, we define $\USB(\mathbb{R}^d)$ for any $d\ge 1$ as the class of functions $V:\mathbb{R}^d\to\mathbb{R}$ that are bounded together with all their partial derivatives \referee{(see also Definition~\ref{d:USBd})}. We then consider
a Schr\"odinger operator $H$ acting on $L^2(\mathbb{R}^d)$:
\begin{equation}
\label{LP:H}
H:=-\Delta +V,\qquad \referee{V\in \USB(\mathbb{R}^d)}.
\end{equation}
 In~\cite{PaSh:16} (two of) the authors of this article formulated the following conjecture.
\referee{\begin{conjecture}
\label{conj}
The spectral function of any operator \eqref{LP:H} admits a complete asymptotic expansion in powers of $\rho$ for large energy:
\begin{equation}
\label{LP:asymptotic1}
\referee{E}(H)(\rho\,;\,\bx,\bx)\sim\sum_{k=0}^{\infty} \referee{f}_k(\bx) \rho^{d-2k}, \ \ \bx\in{\mathbb {R}}^d.
\end{equation} 
\end{conjecture}
\begin{remark}
Notice that one consequence of~\eqref{LP:asymptotic1} is super-polynomial decay of spectral gaps. Therefore, no such asymptotic expansion can hold for potentials which are bounded below but grow as a power of $x$ towards infinity. \end{remark}
}
The intuition behind this conjecture is as follows: it is well known that geodesic loops (geodesics for the metric defining the Laplacian that start and finish at $x$) are usually responsible for preventing asymptotic expansions of this type, and the usual `rule of thumb' is that the fewer periodic geodesics exist, the more asymptotic terms in \eqref{LP:asymptotic1} (or its integrated versions) one can obtain. This leads to a natural guess that if there are NO looping geodesics, a complete asymptotic expansion of the form~\eqref{LP:asymptotic1} should exist. One should, of course, be careful with this type of reasoning since in general it is possible to have singularities in the spectral function that arise from loops of infinite length; i.e. where singularities in the wave propagator return from infinity. However, when the dynamics arise from $\mathbb{R}^d$, or, more generally, from an asymptotically flat metric, this type of return from infinity is not expected.

It is not difficult to see that this conjecture is equivalent to the following statement: suppose, $V_1$ and $V_2$ are two $\USB$ potentials that coincide in a neighbourhood of $\bx$ (or even simply have the same values of all the derivatives at $\bx$) and $H_j=-\Delta+V_j$. Then  \referee{$E(H_1)(\rho;\bx,\bx)-E(H_2)(\rho;\bx,\bx)=O(\rho^{-\infty})$} as $\rho\to\infty$. 

Before \cite{PaSh:16}, \referee{\st{this}} \referee{Conjecture~\ref{conj}} had been proved for smooth potentials with compact support \cite{PoSh:83,Va:83,Va:84,Va:85} using the standard wave equation methods \referee{\st{only instead of tauberian theorems the straightforward inverse Fourier transform of the solutions was constructed}} (see also~\cite{Iv:18b} for related problems in the semiclassical setting). In \cite{PaSh:16}, Conjecture~\ref{conj} was proved in the following three cases: 
\begin{itemize}
\item[(a)] $V$ smooth periodic,
\item[(b)]  $V$ quasi-periodic (a finite linear combination of complex exponentials) with one additional (generic) assumption,
\item[(c)] $V$ smooth almost-periodic with several additional assumptions ensuring that the Fourier coefficients of $V$ decay fast enough. 
\end{itemize}
See also~\cite{ShSh:85, Sa:88} for the 1-dimensional case and~\cite{Iv:18a} for related problems in the semiclassical setting. 

The method used in~\cite{PaSh:16} is often called the method of gauge transform. This method \referee{has appeared in many contexts and is also known by a variety of names; e.g. conjugation to quantum Birkhoff normal form or, in the theory of quasi-periodic operators, KAM. This method was used} in~\cite{Ro:78} to study the discrete spectra of one-dimensional pseudodifferential operators (see also~\cite{Ag:84,HeRo:82}). It was then adapted to periodic operators in \cite{So:05, So:06} and further developed in \cite{PaSo:10,PaSh:12}. \referee{Some examples of the use of this method occur in~\cite{ChVu:08,Sj:00,We:77}, but there are many others}. Since our article also relies on a version of the gauge transform method, we describe this method below in detail. 

To the authors' knowledge, the only other case in which~\eqref{LP:asymptotic1} is known is in dimension one with a certain generalization of almost periodic potentials where complex exponentials are multiplied by functions that are well behaved at infinity instead of constants~\cite{Ga:20}. In that case, the first author was able to apply the gauge transform method together with wave methods and some modern microlocal tools to prove the conjecture. It seems that new ideas would be required to extend these methods to higher dimensions.

The wave method and gauge transform method are intrinsically quite different from each other and it has proved difficult to combine them together. In fact, even obtaining \eqref{LP:asymptotic1} for a sum of a periodic potential and a potential with compact support is still an open question in dimensions larger than one.

This article is the first in a series of papers that aim to address this issue. Here, we prove \referee{Conjecture~\ref{conj}} in its complete generality (i.e. making no assumptions other than that $V$ is a $\USB$ potential) in the one-dimensional case. In (a) subsequent article(s) we plan to consider the case of several dimensions, where, unfortunately, it seems that we will have to impose more restrictions on the potential. 

\subsection{New methods}

First of all, we need some notation. Consider a pseudo-differential operator $V$ acting on $L^2(\R^d)$ with symbol $v=v(\bx,\bxi)$ \referee{satisfying
$$
|\partial_{\bx}^\alpha \partial_{\bxi}^\beta v(\bx,\bxi)|\leq C_{\alpha\beta}(1+|\bxi|^2)^{-|\beta|/2},\qquad (\bx,\bxi)\in T^*\mathbb{R}^d;
$$} 
 all our symbols will be considered in the Weyl quantisation, \referee{i.e.
$$
[Vu](\bx):=\frac{1}{(2\pi )^d}\int e^{i\langle \mathbf{x}-\mathbf{y},\bxi\rangle}v\Big(\frac{\mathbf{x}+\mathbf{y}}{2},\bxi\Big)u(\mathbf{y})d\mathbf{y}d\bxi.
$$} 
We denote by $\hat v=\hat v(\bth,\bxi)$ the Fourier transform of $v$ in the $\bx$ variable considered in the sense of  \referee{tempered} distributions; in analoguey with the periodic case, the variable $\bth$ will be called {\it a frequency}. If $v$ is periodic in $\bx$ with $\Gamma$ being its lattice of periods, then $\hat v$ is a linear combination of delta-functions located at the points of the dual lattice $\Gamma'$. We say that an operator $A$ with symbol $a$ is \referee{a Fourier multiplier} if $a$ does not depend on $\bx$, i.e. $\hat a$ is a multiple of the delta-function at $\bth=0$ (with coefficient depending on $\bxi$). An equivalent description of a Fourier multiplier is this: if we put 
\begin{equation}
\be_{\bxi}(\bx):=e^{i\langle\bxi,\bx\rangle},\qquad\text{then}\qquad 
\label{e:diagonalOp}
A\be_{\bxi}=a(\bxi)\be_{\bxi}.
\end{equation}
For simplicity, in this discussion we assume that $H=-\Delta+V$ is a Schr\"odinger operator. We \referee{take a large} $\rho$ and try to compute $\referee{E}(H)(\rho;\bx,\bx)$. We note that for any \referee{Fourier multiplier}, $A$, it is a relatively simple task to compute $\referee{E}(A)(\rho;\bx,\bx)$. Indeed, since $A$ becomes a multiplication operator after conjugation by the Fourier transform, the spectral function can be computed using the formula for the spectral projector of a multiplication operator \referee{and is given by
\begin{equation}
\label{e:fourierMultSpectral}
\referee{E}(A)(\rho;\bx,\by)=\frac{1}{(2\pi)^d}\int_{G(\rho)}e^{i\langle \bx-\by,\bxi\rangle}d\bxi,\qquad G(\rho):=\{ a(\bxi)\leq \rho^2\}.
\end{equation}} 
\referee{Sometimes we will call Fourier multipliers operators with constant coefficients or diagonal operators because they act diagonally in the Besicovitch space $B_2(\mathbb{R}^d)$.}

Now we will discuss the methods used to establish our results.  In the beginning of our paper, we will treat the case of arbitrary dimension and put $d=1$ only when it becomes necessary. Without loss of generality, we \referee{temporarily} put $\bx=0$ and \referee{\st{, for brevity,  define $\N(\rho;H):=$} call} $\referee{E}(H)(\rho;0,0)$ the \emph{local density of states} at 0. \referee{We usually denote by $N$ the exponent in the remainder in the asymptotic formula  \eqref{LP:asymptotic1}  (which means we can ignore terms $o(\rho^{-N})$). }

\subsubsection{Mass Transport}
\label{s:massTrans} The first step of our approach consists of replacing the operator $H$ with a different operator, $\HM=-\Delta+\QM$; the superscript stands for the {\it mass transport} -- a terminology we explain in a moment.   
 This operator is still a differential Schr\"odinger operator with a $\USB$ potential $\QM$ that `almost agrees' with $V$ on a large box, i.e. we have 
\bee
\label{LP:1}
|V(\bx)-\QM(\bx)|= O(\rho^{-N'}),\qquad\text{for any } \bx\in B(0,\rho^{N'}).
\ene
Here, $N'$ is a large number depending on $N$ and $B(0,R)$ is a ball in $\R^d$ with centre at $0$ and radius $R$.

\referee{The usefulness of this notion of mass transport follows from our next two claims. We} first claim that \referee{for any $N>0$ there is $N'>0$ such that} whenever  
\eqref{LP:1} is satisfied we have
\bee
\label{LP:j} 
\referee{E(H)(\rho;0,0)-E(\HM)(\rho;0,0)}=O(\rho^{-N}).
\ene
\referee{Second, we claim that for any $V\in \USB(\mathbb{R}^d)$, one can use the flexibility of choosing $\QM$ satisfying~\eqref{LP:1} to simplify the problem of computing the spectral function.} 

\begin{remark}
\referee{We expect that one could take $N'=2(N+d)$, but we do not attempt to follow the dependence of $N'$ on $N$ carefully.}
\end{remark}

\referee{Our first claim,~\eqref{LP:j},} may be surprising at first glance. We have made a potentially large change to the operator that does not arise from a unitary transformation and yet the density of states is affected only very mildly. To understand why this large change does not have a large effect on the spectral function at 0, we use the fact that solutions of the wave equation corresponding to $H$ and $\HM$ with the same initial conditions having support in a fixed neighbourhood of the origin agree
up to  $O(\rho^{-N'})$ for a very long time ($t\le \rho^{N'}$). Using the wave method, we are then able to convert this wave estimate into one on spectral functions. This is the only essential place in our approach where we use the wave equation method; we discuss this \referee{\st{mass transport}} method in more detail (and prove it) in Section~\ref{s:comparison}. 

\referee{\begin{remark}
The reason we refer to this process as mass transport is because, when the Fourier transform of $V$ is a measure, the estimate~\eqref{LP:1} holds whenever the natural \emph{mass transport distance}, the $1$-Wasserstein distance, (see e.g~\cite[Chapter 6]{Vi:09}) between $\hat{V}$ and $\widehat{\QM}$ is $O(\rho^{-2N'})$. See Figure~\ref{f:massTransport} for a schematic of this mass transport on the Fourier transform side. If $V$ is almost periodic, then working with $1$-Wasserstein distances of the Fourier transform of $V$ is more convenient than working directly with the values of $V$. Indeed, if $V$ is almost periodic, then our result shows that under certain mild extra assumptions a small change in its frequencies results in a small change of the spectral function. In fact, we arrived at the statement of Theorem~\ref{t:simpleCompare} by guessing that a small mass transport of this type should lead to a small change in the local density of states. \end{remark}}

\referee{To explain the second claim,} we ask the natural question: what is the best way to modify our potential $V$ outside of the box $B(0,\rho^{N'})$ so that we can compute the spectral function $\referee{E(\HM)(\rho;0,0)}$ of the resulting operator (up to a small error)? It may seem natural to choose $\QM$ with compact support, but we do not know of any `standard' microlocal methods that can handle a potential which is compactly supported, but with support depending badly on $\rho$. Instead, perhaps slightly surprisingly, we choose $\QM$ to be periodic (with period $\rho^{N'}$) 
and try to compute $\referee{E(\HM)(\rho;0,0)}$ using the periodic method of gauge transform (GT). \referee{The advantage of a periodic potential is that the support of its Fourier transform is discrete (at scale $\rho^{-N'}$).} The significant new difficulty, as compared to the `standard' setting of using the GT, is that now the frequencies (elements of the lattice dual to the lattice of periods) can become very small (of size $\rho^{-N'}$). In order to explain how we overcome this difficulty, we first describe the `standard' GT, referring, in the first instance, to~\cite{LMPPS} where this method is described in an abstract setting.

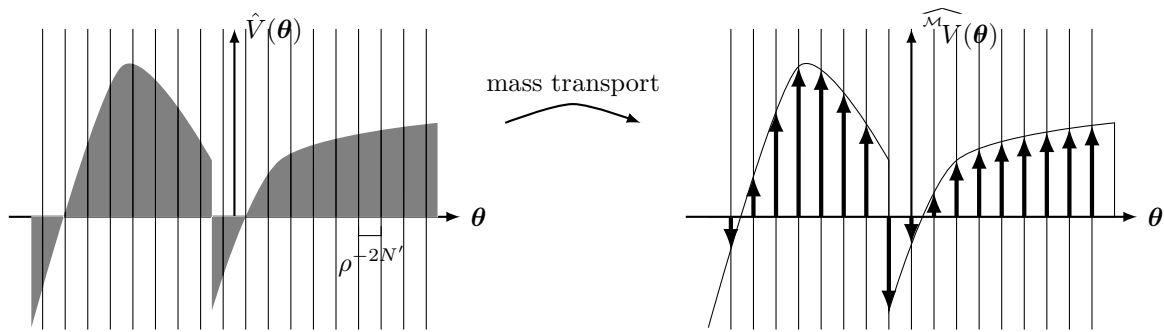
\begin{figure}
\begin{tikzpicture}
\def \w{3};
\def \h{2.5};
\def\min{.6};
\def \step{.15};
\draw [->,thick] (0,0)--(\w,0)node[right]{$\bth$};
\draw [thick] (0,0)--(-\w,0);
\draw[->,thick](0,0)--(0,\h)node[right]{$\hat{V}(\bth)$};
\draw [white, name path=hatV]plot [smooth]coordinates{(-\w*.9,-\h*.59)(-\w*.5,\h*.8)(-\w*.1,\h*.3)}--plot [smooth]coordinates{(-\w*.1,-\h*.5)(\w*.2,\h*.3)(\w*.9,\h*.5)}--(\w*.9,0)--(-\w*.9,0);

\fill[fill=gray,opacity=.6] plot [smooth]coordinates{(-\w*.9,-\h*.59)(-\w*.5,\h*.8)(-\w*.1,\h*.3)}--plot [smooth]coordinates{(-\w*.1,-\h*.5)(\w*.2,\h*.3)(\w*.9,\h*.5)}--(\w*.9,0)--(-\w*.9,0)--cycle;
\foreach \s in{-.85, -.75,...,.9}{
\draw[opacity=0, name path=line] (\s*\w,-\min*\h)--(\s*\w,\h);
\draw[name intersections= {of =line and hatV,total=\t}](\s*\w,0)--(intersection-1);
}
\draw  (.55*\w,-.05*\h)--(.55*\w,-.1*\h)--(.61*\w,-.1*\h) node[below]{$\rho^{-2N'}$}--(.65*\w,-.1*\h)--(.65*\w,-.05*\h);

\draw[->,thick] plot [smooth]coordinates {(1.2*\w,.5*\h) (1.5*\w,.6*\h)(1.8*\w,.5*\h)};
\draw[thin] node at(1.5*\w,.7*\h){mass transport};

\begin{scope}[xshift=3*\w cm]
\draw [->,thick] (0,0)--(\w,0)node[right]{$\bth$};
\draw [thick] (0,0)--(-\w,0);
\draw[->,thick](0,0)--(0,\h)node[right]{$\widehat{\QM}(\bth)$};
\draw [opacity=0, name path=hatV2]plot [smooth]coordinates{(-\w*.9,-\h*.59)(-\w*.5,\h*.8)(-\w*.1,\h*.3)}--plot [smooth]coordinates{(-\w*.1,-\h*.5)(\w*.2,\h*.3)(\w*.9,\h*.5)}--(\w*.9,0)--(-\w*.9,0);

\foreach \s in{-.8, -.7,...,.90}{
\draw[opacity=0, name path=line] (\s*\w,-\min*\h)--(\s*\w,\h);
\draw[name intersections= {of =line and hatV2,total=\t}, ultra thick,->](\s*\w,0)--(intersection-1);
}
\end{scope}
\end{tikzpicture}
\caption{\label{f:massTransport} The figure shows the process of mass transport for a potential on the Fourier transform side when $\hat{V}$ is a measure. When we replace $V$ with a periodic potential $\QM$ which agrees on $B(0,\rho^{N'})$, we replace its Fourier transform by a sum of delta functions with a lattice at scale $\rho^{-2N'}$. Roughly speaking, we transport the total mass of the potential near each lattice point to a delta function at that lattice point. }
\end{figure}

\subsubsection{The Standard Method of Gauge Transform}

\referee{Although the bulk of this article is written in dimension one, it is important to understand the context into which the methods fit. To this end, we review the method of gauge transform, as it applies to spectral asymptotics, in all dimensions. In the next subsection, we then focus specifically on dimension one and the new gauge transform methods developed in this article.}

\referee{In this subsection, we assume $V\in \USB(\mathbb{R}^d)$ or, more generally, $V$ is a pseudodifferential operator with symbol $v(\bx,\bxi)$ bounded with all derivatives. We denote by $\hat{v}(\bth,\bxi)$ the Fourier transform of $v(\bx,\bxi)$ in the $\bx$ variables considered as a tempered distribution.}
Given an operator $H(V):=-\Delta+V$, our ultimate goal (Task A) is to find a unitary operator $U$ such that, after conjugating by $U$, $H$ becomes simpler: 
\bee
\label{LP:aim}
U^{-1}HU=\referee{-\Delta}+\mathbf{a}(D)+\check{V},\qquad D:=-i\partial_x.
\ene
Here, $\mathbf{a}(D)$ is a \referee{Fourier multiplier} and \referee{$\|\check{V}\|_{\infty}=O(\rho^{-N})$ for any $N$ so} it does not contribute to the asymptotic expansion \referee{of the spectral function}. \referee{Therefore, we may compute the spectral function of the conjugated operator using~\eqref{e:fourierMultSpectral}.} At first, we notice that if we can construct $U$ to achieve the simpler task (Task $\text{A}'$) of $\check{V}$ being smaller than $V$ (for example, of smaller order), then we can iterate this process to make the non-diagonal part smaller and smaller, eventually making $\check{V}$ small enough to be negligible and completing Task A.

We look for $U$ of the form $U=e^{i\Psi}$ with $\Psi$ a self-adjoint pseudo-differential operator \referee{with symbol $\psi$}. Then, \referee{at least formally, we have} \referee{\st{the Liouville formula gives}} 
\bee
\label{LP:GT}
U^{-1}HU=H+i[H,\Psi]-\frac12[[H,\Psi],\Psi]+\dots=-\Delta+V+\referee{i[-\Delta,\Psi]+i[V,\Psi]}-\frac12[[H,\Psi],\Psi]+\dots,
\ene
where $[\cdot,\cdot]$ is a commutator and $\dots$ denotes terms involving higher order commutators with $\Psi$. Now we try to accomplish Task $\text{A}'$ by finding $\Psi$ that solves the equation

\bee
\label{LP:equation}
V+i[\referee{-\Delta},\Psi]=0.
\ene
If we can do this with $\Psi$ from a reasonable class of pseudodifferential operators \referee{of order less than zero}, this would finish Task $\text{A}'$. \referee{Since the symbol, $b$, of the pseudodifferential operator $B=[-\Delta,\Psi]$ satisfies
$$
 \hat{b}(\bth,\bxi)= 2\langle \bxi,\bth\rangle \hat\psi(\bth,\bxi),
$$
we see that a solution of this equation is given, ignoring possible small divisor problems,} by \referee{the pseudodifferential operator with symbol $\psi$ satisfying}
\bee
\label{LP:solution}
\hat\psi(\bth,\bxi)=\frac{i\hat v(\bth,\bxi)}{2\langle \bxi,\bth\rangle}.
\ene

\referee{
\begin{remark}
In this text we often use the convention that lower case letters denote the symbol of the operator denoted by the corresponding upper case letter, e.g. $v$ is the symbol of $V$. However, when $V$ is a function, we do not distinguish between the function $V$ and the operator of multiplication by $V$.
\end{remark}}

\referee{
\begin{remark}
Although~\eqref{LP:solution} is simple, it is not very convenient for obtaining $L^\infty$ type estimates. We will later replace it by~\eqref{e:niceFormula} which is more suited to this purpose.
\end{remark}}

We emphasize once again that we work in the Weyl quantisation because in other quantisations the form of the denominator is different (but may be more familiar to some readers). Now it is clear what the main obstacle to solving \eqref{LP:equation} is: the denominator of \eqref{LP:solution} may be very small (or indeed zero). A pair $(\bxi,\bth)$ for which the inner product $\langle\bxi,\bth\rangle$ is small will be called {\it resonant} and otherwise will be called {\it non-resonant}. If 
$(\bxi,\bth)$ is resonant, we will sometimes say that $\bxi$ is resonant with respect to $\bth$ and vice versa. 

Given this information, we can now modify our procedure. We split our perturbation $V$ into two parts:
\bee
V=V^{(r)}+V^{(n)}
\ene
(superscript r stands for `resonant' and n for `non-resonant') so that the support 
of $\hat v^{(n)}$ consists only of non-resonant pairs $(\bxi,\bth)$. Then, instead of \eqref{LP:equation}, we solve the equation 
\bee
\label{LP:equationn}
V^{(n)}+i[H,\Psi]=0.
\ene
Next, using \eqref{LP:GT}, we express the operator $U^{-1}HU$ in the form 
$-\Delta+V^{(r)}+V_1$, where $V_1$ is smaller than $V$. Finally, we repeat the procedure as many times as necessary.

\begin{remark}\label{r:serialParallel} There are two slightly different ways to iterate this procedure. One consists in writing our transform $U$ in the form $U=e^{i(\Phi_1+\Phi_2+\dots)}$. This method is called {\it a parallel} GT in \cite{LMPPS}. The second method (called {\it a serial} GT) looks for $U$ in the form $U=e^{i\Psi_1}e^{i\Psi_2}\dots$. These methods are often equivalent, but it may be more convenient to use either one of them in specific situations. While in the papers \cite{PaSo:10,PaSh:12,PaSh:16} a parallel GT was used, we will use a mixture of both  serial and parallel GTs in this paper. 
\end{remark} 

After repeating the above procedure as many times as necessary, we will arrive at the following  form of the conjugated operator: 
\bee
\label{LP:aim2}
U^{-1}HU=\referee{-\Delta+\mathbf{a}}(D)+V_n^{(r)}+\check{V}_n,
\ene
where $V_n^{(r)}$ is resonant and $\check{V}_n$ is so small that we can ignore it when computing the asymptotic expansion of the spectral function. This is usually where the GT method stops. We are left with having to analyse the operator 
$$\HG:=-\Delta\referee{+\mathbf{a}(D)}+\QG,\qquad \QG:=V_n^{(r)}.$$ 
In particular, we need to compute the spectral function for $\HG$. Note that if we \referee{had} started with a potential, $V$, \referee{which} is periodic with $(\Gamma,\Gamma')$ its lattice of periods and the corresponding dual lattice, then the end perturbation  $\QG$ \referee{would} also \referee{be} periodic (with the same lattice of periods, but possibly with more non-zero Fourier coefficients than $V$ had). We now examine the structure of $\QG$ in the periodic case more carefully. Since we are trying to compute the spectral function for large $\rho$, we can concentrate on points $\bxi$ with $|\bxi|\sim\rho$. 
Let us look at the following special cases:

\begin{itemize}
\item[I.] $d=1$. Then $|\langle\bxi,\bth\rangle|=|\bxi||\bth|$ and so for $\langle \bxi,\bth\rangle$ to be small, we must have $\bth=0$ (recall that $|\bxi|\sim\rho$ and $\bth\in\Gamma'$), so the operator $\QG$ is truly diagonal; see~\cite{So:06}. 
\item[II.] $d=2$. Then $\QG$ does not need to be diagonal. However, the following is true. Suppose, $(\bxi_1,\bth_1)$ and $(\bxi_2,\bth_2)$ are two resonance pairs with $|\bxi_j|\sim\rho$ and $\bth_1$ not parallel to $\bth_2$. Then $\bxi_1\ne\bxi_2$. This observation allows us to construct a large family of invariant subspaces for $\HG$; a careful analysis of the action of $\HG$ inside each of these subspaces then enables us to compute the spectral function; see \cite{PaSh:12}.
\item[III.]In the case $d\ge 3$ we use similar considerations to the case $d=2$, only the decomposition into invariant subspaces is a bit more involved; see \cite{PaSh:12}.  
\end{itemize}

\referee{There is one more technical detail related to the `classical' GT that we need to discuss. When the exponents $\Psi_j$ are small (\referee{their $L^2\to L^2$ operator norms} tend to zero as $\rho\to\infty$), the conjugation operator, $U$, is a small perturbation of the identity and, in fact, the higher order terms in~\eqref{LP:GT} become smaller even without taking account of possible cancellations in the commutators. This is sometimes called a \emph{weak gauge transform}. There are, however, many situations where the $\Psi_j$ are \emph{not} small as operators from $L^2\to L^2$. In this case, the only way we can think of higher order terms in~\eqref{LP:GT} as small errors, is by taking advantage of cancellations in the algebraic structure of successive commutators. Indeed, individual remainder terms like $\Psi_jH\Psi_j$ which occur in the higher order commutators $[[H,\Psi_j],\Psi_j]$ can be larger than corresponding terms at the previous step, e.g $H\Psi_j$. When the gauge transform involves $\Psi_j$'s which are not small, it is sometimes referred to as a \emph{strong gauge transform}. (See \cite{LMPPS} for further discussion of the difference between the two procedures.) 

The concepts of strong and weak gauge transforms can be applied in different settings. For example, if $-\Delta$ is replaced by a non-principally scalar system then there are typically no additional cancellations in the commutators and hence we can only take advantage of smallness of $\Psi_j$ and therefore use a weak gauge transform. Similarly, if $V$ is replaced by a pseudodifferential perturbation whose symbol has derivatives in $\bxi$ which behave badly, this property will pass to the $\Psi_j$ and destroy many cancellations in commutators. On the other hand, if we replace $V$ by a pseudodifferential operator of order $m\in[1,2)$, it will not be possible to solve~\eqref{LP:equation} with $\Psi_j$ having small norm and hence we must take advantage of cancellations in successive commutators, using a strong gauge transform. In this article, it will be necessary to use $\Psi_j$ whose $L^2\to L^2$ norms are, in fact, growing quickly as a function of $\rho$ and hence we will need to take advantage of cancellations in commutators.}

We now discuss the modifications needed in this process if $d=1$. \referee{The crucial feature which allows us to handle all $\USB$ potentials in 1-dimension but does not occur in higher dimensions is that the denominator in~\eqref{LP:solution} can only be small for $|\xi|\sim \rho$ when $\theta$ is close to $0$.} We try to apply the GT to $\QM$ -- a periodic potential obtained from a $\USB$ potential $V$ by the process of mass transport discussed in the previous subsection. \referee{In fact, we will replace $V$ by $\QM$ with $\QM\equiv V$ on $B(0,\rho^{N'})$ and periodic at scale $\rho^{N'}$ \referee{so that~\eqref{LP:1} is satisfied}. We will denote this particular approximation to $V$ as ${}^PV$. In this article, this is the only `mass transport' of $V$ that is used.} Recall, in particular, that the dual lattice, $\Gamma'$, now has elements of size $\rho^{-N'}$. \referee{Because of this, the usual GT method does not suffice and we must modify it in a way described in the next subsection.}\referee{\st{ As stated above, we plan to consider the case of $d>1$ in further publications, where we will have to impose additional restrictions on $V$. }}

\subsubsection{Onion peeling}
\label{s:onionIntro}
We now assume that $d=1$ and, for a while, that the initial $V$ is a sum of a smooth periodic function, $\referee{V_p}$, and a smooth function with compact support, $\referee{V_c}$, (or, more generally, smooth \referee{rapidly} decaying function). We periodise $V$ to $\referee{{}^PV}$, a periodic function \referee{with very large period of size $\rho^{N'}$ and proportional to the period of $V_p$ so that ${}^PV=V_p+{}^PV_c$}. A simple calculation shows that if a denominator in \eqref{LP:solution}, \referee{with $V={}^PV$}, is non-zero, but small (recall that $|\bxi|\sim\rho$, so this can happen only if $|\bth|$ is small), then the numerator of the same formula is also small. \referee{Indeed, if $0<|\theta|$ and $\theta$ is in the support of $\hat{V}_p$, then $|\theta|>c>0$. On the other hand, since the Fourier transform of $V_c$ is a smooth, rapidly decaying function, the Fourier coefficients of ${}^PV_c$ are of size $\rho^{-N'}$.} \referee{We then define the operator $\Psi$} by \eqref{LP:solution} with $\hat{v}=\widehat{\referee{{}^PV-\langle {}^PV\rangle}}$ and 
$$
\referee{\langle {}^PV\rangle=\lim_{R\to \infty}\frac{1}{2R}\int_{-R}^R{}^PV(r)dr}.
$$ 
Then, \referee{$\langle {}^PV\rangle$ is the mean of $\referee{{}^PV}$ and $\Psi$} belongs to the standard class of pseudo-differential operators of order 0 and the whole process described in the previous sub-section can be carried out as a weak GT. 
We will comment on this case later (see Remarks~\ref{decay1} and \ref{decay2}) just to illustrate the main ideas of our approach without going into many technicalities necessary in the general setting.  
\begin{figure}[htbp]
\begin{tikzpicture}
\def \w{5};
\def \h{1.5};
\def \start{1};
\def \scale{2.7};
\def \num{3};
\def \shift{1.7};
\def \min{.5};

\draw node at(1.1*\w,{\num*\shift*\h+\h}){$\hat{V}(\bth)$};

\foreach \s in {0,...,3}{
\begin{scope}[yshift=\shift*\h*\s cm]
\draw[->](0,0)--(0,.8*\h);
\draw [opacity=0, name path=hatQ2]plot [smooth]coordinates{(-\w*.9,\h*.5)(-\w*.5,\h*.8)(-\w*.1,\h*.3)}--plot [smooth]coordinates{(-\w*.1,\h*.9)(0,.4*\h)(\w*.1,-\h*.45)(\w*.2,-\h*.3)(\w*.4,\h*.8)(\w*.9,\h*.5)}--(\w*.9,0)--(-\w*.9,0);

\foreach \s in{-.9, -.8,...,.91}{
\draw[opacity=0, name path=line] (\s*\w,-\min*\h)--(\s*\w,\h);
\draw[name intersections= {of =line and hatQ2,total=\t}, ultra thick,->](\s*\w,0)--(intersection-1);
}

\fill[fill=white] (-\w,-\min*\h)--(-\w,.8*\h)--({-\w*\start/((\scale)^(\s))},.8*\h)--({-\w*\start/((\scale)^(\s))},-\min*\h)--cycle;
\fill[fill=white] (\w,-\min*\h)--(\w,.8*\h)--({\w*\start/((\scale)^(\s))},.8*\h)--({\w*\start/((\scale)^(\s))},-\min*\h)--cycle;
\draw [->,thick] (-\w,0)--(\w,0)node[right]{$\bth$};
\if
\fi
\end{scope}
}

\foreach \s in {1,...,3}{
\begin{scope}[yshift=\shift*\h*\s cm]
\draw[dashed] ({\w*\start/((\scale)^(\s))},-.1*\h)--({\w*\start/((\scale)^(\s))},.7*\h)node[right]{$\rho^{-\s/4}$};
\draw[dashed] ({-\w*\start/((\scale)^(\s))},-.1*\h)--({-\w*\start/((\scale)^(\s))},.7*\h)node[left]{$-\rho^{-\s/4}$};
\end{scope}
}

\foreach \s[evaluate=\s as \y using {\s-1}] in {1,...,3}{
\begin{scope}[yshift=\shift*\h*\y cm]
\draw [->] plot [smooth]coordinates {(0,.9*\h)(.05*\w,1.15*\h)(0,1.4*\h)};
\draw node at (.2*\w,1.15*\h){$e^{i\Psi_{\s}}$};
\end{scope}
}

\end{tikzpicture}
\caption{\label{f:serial} The effect of the serial gauge transform. Each successive conjugation removes a layer from the Fourier transform of $\referee{{}^PV}$. }
\end{figure}
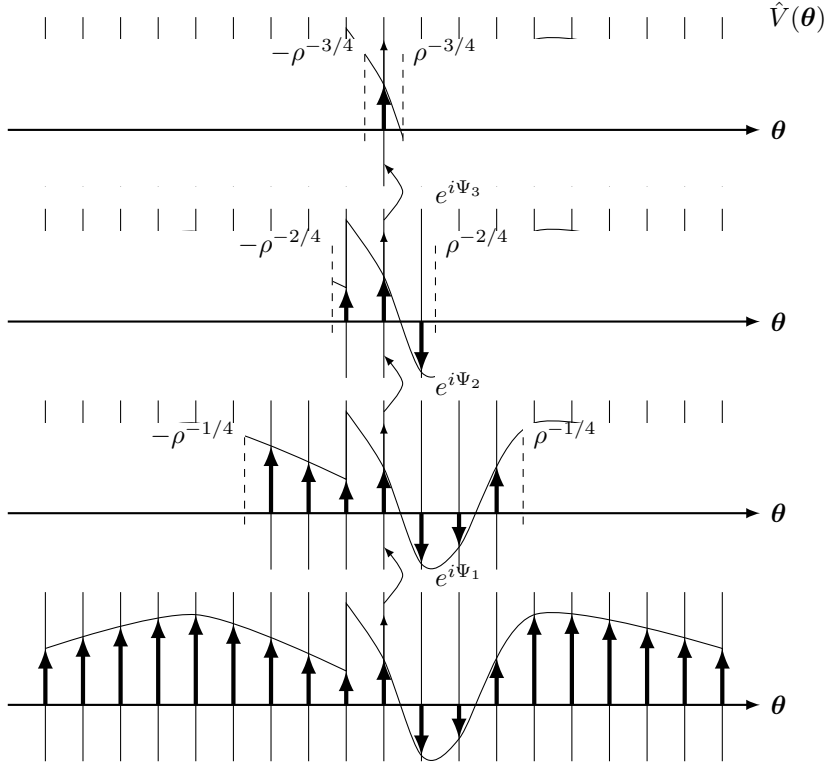 

Now we consider the general case of \referee{$V$ being a pseudodifferential operator with $\USB$ symbol and no further assumptions}. Then it can happen that $\bth$ is small, while $\hat v(\bth,\bxi)$ is not, so $\hat\psi(\bth,\bxi)$ obtained using \eqref{LP:solution} is large. This means that we cannot perform the weak GT, but we may ask whether there is a strong GT process that can succeed. Perhaps it is the case that, despite $\Psi$ being large, the commutator $[V,\Psi]$ is small nevertheless? The answer is yes and no. \referee{To illustrate this point, we consider the example $V=e^{i\theta_1x}+e^{i\theta_2 x}$. Notice that $[V,\Psi]$ appears in~\eqref{LP:GT} and hence we want this term to be smaller than $V$ itself. 
We compute
\begin{align*}
\{V,\psi\}(x,\xi)&=-\partial_x V(x)\partial_\xi \psi(x,\xi)= \frac{i}{2\xi^2}\Big(\theta_1 e^{i\theta_1x}+\theta_2e^{i\theta_2x}\Big)\Big(\frac{e^{i\theta_1x}}{\theta_1}+\frac{e^{i\theta_2x}}{\theta_2}\Big)\\
&=\frac{i}{2\xi^2}\Big[e^{2i\theta_1x}+e^{2i\theta_2x} +\Big(\frac{\theta_1}{\theta_2}+\frac{\theta_2}{\theta_1}\Big)e^{i(\theta_1+\theta_2)x}\Big].\end{align*}
This symbol is indeed smaller than that of $V$ (for $|\xi|\sim \rho$) if $\rho^{-2+0}\leq \frac{|\theta_1|}{|\theta_2|}\leq  \rho^{2-0}$ but not if the ratio $|\theta_1|/|\theta_2|$ is too small or too large. 
\begin{remark}We are aware that this example is not self-adjoint, but this is the simplest example to illustrate our point.\end{remark}
} 

\referee{\st{When we commute the parts of $V$ and $\Psi$ at frequencies $\bth$ of comparable length, the result is indeed small; however, when we commute the operator $e^{i\bth_1x}\hat v(\bth_1,D)$ with $e^{i\bth_2x}\hat\psi(\bth_2,D)$ and $\frac{|\bth_1|}{|\bth_2|}\gtrsim\rho$, the result will be large. Such terms ($V$ computed at some frequency commuting with $\Psi$ computed at a much smaller frequency) are the only contributions to our commutators that can be large; the rest of the terms are good in the sense that they decrease in size. }}

 This observation suggests the following modification of the basic GT. Given a $\USB$ potential $V$ (or rather a periodic potential $\referee{{}^PV}$ obtained from $V$ after mass transport), we first remove the part of $\widehat {\referee{{}^PV}}(\bth)$ corresponding to $|\bth|>1$. \referee{(See Lemma~\ref{l:layer0} for a precise description of this process.)} This can be done in one step since the potential $\referee{{}^PV}$ is smooth. \referee{The result of this, first, step of gauge transform is a pseudodifferential operator which we refer to as $\referee{{}^PV}_0$.}
 
 \referee{Next, we split $\referee{{}^PV}_0$ according to the size of the frequencies: we let $\referee{{}^PV}_{+,0}$ to be the part of $\referee{{}^PV}_0$ corresponding to frequencies $\theta$ satisfying $|\theta|\in[\rho^{-1/2},1]$. We then conjugate away $\referee{{}^PV}_{+,0}$ using the strong GT. This peels off the `outer layer' $\referee{{}^PV}_{+,0}$ (corresponding to the largest frequencies). Since all the frequencies in $\referee{{}^PV}_0$ are smaller than those of $\referee{{}^PV}_{+,0}$, during this process we never encounter the bad case of having to commute $V$ and  $\Psi$ where $\Psi$ has a frequency much smaller than some frequency of $V$. Strictly speaking, this process produces new terms with frequency $|\theta|\leq \rho^{-1/2}$, but we ignore this for simplicity. We now repeat this argument with $\referee{{}^PV}_{+,j}$ corresponding to frequencies $\theta$ satisfying $|\theta|\in[\rho^{-(j+1)/2},\rho^{-j/2}]$, always peeling away the piece of $\referee{{}^PV}$ with largest frequency first. At the end of $4N'$ steps of this process, we are left with the part of $\referee{{}^PV}$ with frequencies $|\theta|\leq \rho^{-(4N'+1)/2}$. Since we started with potential that was $\rho^{2N'}$-periodic, in fact, this part of $\referee{{}^PV}$ does not depend on $x$ and hence is a Fourier multiplier. \referee{(See Lemma~\ref{l:nextLayer} for a precise description of the iterative step.)}}

%

In fact, to peel each layer, we will need to perform a parallel GT with $U=e^{i\Psi_k}$. Each step in the parallel transform will decrease the size of $\widehat{\referee{{}^PV}}_k$ by a factor $\sim \rho^{-1}$. For instance, to peel the first layer, we find $\Psi_0\sim \sum _j \referee{\Phi_0^{(j)}}$, with $\referee{\Phi_0^{(j)}}$ naturally living in a certain class of of pseudodifferential operators such that 
$$
e^{-i\sum_{j=1}^N\referee{\Phi^{(j)}_0}}\referee{{}^P\!H} e^{i\sum_{j=1}^N\referee{\Phi^{(j)}_0}}= -\Delta +\QG_N,\qquad  \widehat{\vG_{_N}}1_{|\theta|\geq 1}=O(\rho^{-N-1});
$$
that is, the Fourier transform of $\QG_N$ is very small for $|\theta|\geq 1$. 
We then iterate this procedure for each layer, producing $\Psi_k\sim \sum _j \referee{\Phi_{k}^{(j)}}$ to remove the $k^{th}$ layer.  The reason for doing this mixed parallel-serial GT procedure is as follows: \referee{for each $k$, the $x$ derivatives of the symbol of $\referee{ \Phi_k^{(j)}}$, $j=1,\dots$ are comparable to one another}, for $k_1\neq k_2$, and any $j_1,j_2$, \referee{ the $x$ derivatives of the symbols of} $\referee{\Phi^{(j_1)}_{k_1}}$ and $\referee{\Phi^{(j_2)}_{k_2}}$ \referee{are not comparable}. In particular, as $k\to \infty$, $\referee{\Phi_k^{(j)}}$ may become very large for any fixed $j$. However, the derivatives of its symbols in $x$ will become correspondingly small. See Figure~\ref{f:parallel} for an illustration of the parallel process of removing a single layer and Figure~\ref{f:serial} for the serial process of removing successive layers.

\begin{figure}[htbp]
\begin{tikzpicture}
\def \w{5};
\def \h{1.8};
\def \start{1};
\def \scale{2};
\def \num{2};
\def \shift{1.7};
\def \scaley{.3};
\def \min{.5};

\draw node at(1.1*\w,{\num*\shift*\h+\h}){$\hat{V}(\bth)$};

\foreach \s in {0,...,1}{
\begin{scope}[yshift=\shift*\h*\s cm]
\draw[->](0,0)--(0,.8*\h);
\draw [opacity=0, name path=hatQ2]plot [smooth]coordinates{(-\w*.9,{(\scaley)^(\s)*\h*.5})(-\w*.7,{(\scaley)^(\s)*\h*.8})(-\w*.5,{(\scaley)^(\s)*\h*.6})}--plot[smooth]coordinates{(-\w*.5,\h*.6)(-\w*.4,\h*.3)(-\w*.3,-\h*.4)(-\w*.2,\h*.4)(-\w*.1,\h*.3)}--plot [smooth]coordinates{(-\w*.1,\h*.9)(\w*.2,\h*.3)(\w*.5,\h*.6)}--plot[smooth]coordinates{(\w*.5,{(\scaley)^(\s)*\h*.6})(\w*.7,{(\scaley)^(\s)*\h*.8})(\w*.9,{(\scaley)^(\s)*\h*.5})}--(\w*.9,0)--(-\w*.9,0);

\foreach \s in{-.9, -.8,...,.91}{
\draw[opacity=0, name path=line] (\s*\w,-\min*\h)--(\s*\w,\h);
\draw[name intersections= {of =line and hatQ2,total=\t}, ultra thick,->](\s*\w,0)--(intersection-1);
}

\draw [->,thick] (-\w,0)--(\w,0)node[right]{$\bth$};
\end{scope}
}

\foreach \s in {2}{
\begin{scope}[yshift=\shift*\h*\s cm]
\draw[->](0,0)--(0,.8*\h);
\draw [opacity=0, name path=hatQ2]plot [smooth]coordinates{(-\w*.9,{(\scaley)^(\s)*\h*.5})(-\w*.7,{(\scaley)^(\s)*\h*.8})(-\w*.5,{(\scaley)^(\s)*\h*.6})}--plot[smooth]coordinates{(-\w*.5,\h*.6)(-\w*.4,\h*.3)(-\w*.3,-\h*.4)(-\w*.2,\h*.4)(-\w*.1,\h*.3)}--plot [smooth]coordinates{(-\w*.1,\h*.9)(\w*.2,\h*.3)(\w*.49,\h*.6)}--plot[smooth]coordinates{(\w*.49,{(\scaley)^(\s)*\h*.6})(\w*.7,{(\scaley)^(\s)*\h*.8})(\w*.9,{(\scaley)^(\s)*\h*.5})}--(\w*.9,0)--(-\w*.9,0);

\foreach \s in{-.4, -.3,...,.49}{
\draw[opacity=0, name path=line] (\s*\w,-\min*\h)--(\s*\w,\h);
\draw[name intersections= {of =line and hatQ2,total=\t}, ultra thick,->](\s*\w,0)--(intersection-1);
}

\draw [->,thick] (-\w,0)--(\w,0)node[right]{$\bth$};
\end{scope}
}

\foreach \s in {1}{
\draw[dashed] ({\w*\start/((\scale)^(\s))},-.1*\h)--({\w*\start/((\scale)^(\s))},{\num*\shift*\h+.9*\h})node[right]{$\rho^{-\s/4}$};
\draw[dashed] ({-\w*\start/((\scale)^(\s))},-.1*\h)--({-\w*\start/((\scale)^(\s))},{\num*\shift*\h+.9*\h})node[left]{$-\rho^{-\s/4}$};
}

\begin{scope}[yshift=\shift*\h*0 cm]
\draw [->] plot [smooth]coordinates {(0,.9*\h)(.05*\w,1.15*\h)(0,1.4*\h)};
\draw node at (.18*\w,1.15*\h){$e^{i\Phi_{1}}$};
\end{scope}
\draw [->] plot [smooth]coordinates {(1.2*\w,.7*\h)(1.4*\w,{(.25*(\num*\shift+.2)+.4)*\h})(1.45*\w,{(.5*(\num*\shift+.2)+.4)*\h})(1.4*\w,{(.75*(\num*\shift+.2)+.4)*\h})(1.2*\w,{(\num*\shift+.3)*\h})};
\draw node at (1.7*\w,{(.5*(\num*\shift+.2)+.4)*\h}){$e^{i(\Phi_{1}+\Phi_2)}$};


\end{tikzpicture}
\caption{\label{f:parallel} The effect of the parallel gauge transform. Each successive conjugation  removes most of the outermost layer from $\hat{V}$. }
\end{figure}
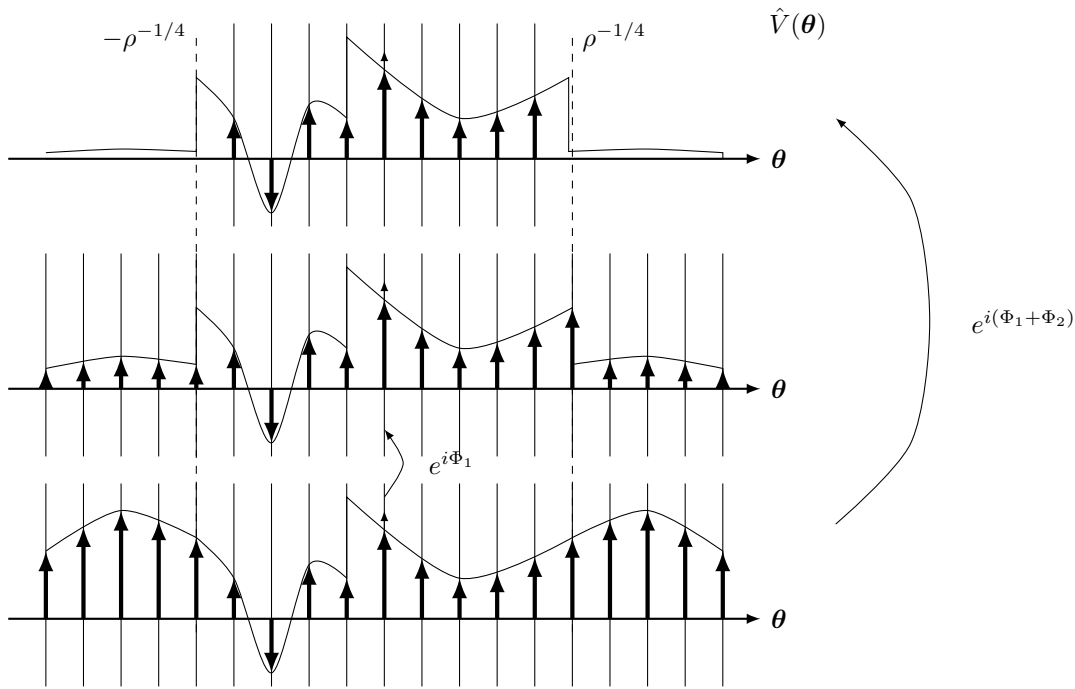

We remark that, unfortunately, this process, at least as formulated, cannot be used if $d\ge 2$.  This is because, in order to start the GT process in the second layer, we need to have removed {\bf{all}} the frequencies in the first (outer) layer.  In higher dimensions the resonant terms have to remain in the outermost layer. Thus, say, the commutators between $\Psi$ used to remove frequencies inside the second layer and the resonant part of $V$ from the outermost layer would still be large.  

\begin{remark}
\referee{Throughout the article, we attempt to present arguments in a way that is accessible to several communities: microlocal/semiclassical analysts, spectral geometers, and specialists in periodic and almost periodic operators. Often when we introduce terminology, we will try to give alternative versions familiar to each community. In addition, we try to include proofs of results that may be standard for one community but not the others. One consequence of this is that we first state our results on the spectral projector for a fixed operator $H$ as the energy, $\rho\to \infty$ in Section~\ref{s:formulate}. We then translate the results into their semiclassical formulation in Section~\ref{s:semiFormulate}, where we study families of operators, $\mathbf{H}(\hbar)$, depending on a small parameter, $\hbar\downarrow 0$. \referee{We write our proofs using the language of semiclassical analysis. There are two reasons why we do this. The first (and main) reason is that in the course of the proof, we will often need to quote results from microlocal analysis that exist in the literature in the semiclassical, but not the high-energy language. 
The second reason is that in the semiclassical setting we will be able to work with slightly more general classes of operators.} } 
\end{remark}

\subsection{Strategy of the proof}

\referee{The proof of Theorem~\ref{t:USBAsymptotics1} will proceed in four steps. Despite the fact that the proof is written in semiclassical language, we discuss it here using the language and notation of the high energy regime. The first step of the proof is to use mass transport to replace the potential $V$ by a periodic potential ${}^PV$ with period $R(\rho)\sim \rho^{N'}$ for some large $N'$.  This is done in Section~\ref{s:periodisePotential}, with the proof that periodising (or indeed any small mass transport) makes a small change to the spectral function done in Section~\ref{s:comparison}. Note that it is essential to periodise the potential before making any microlocal reductions because our theorem about the effect of mass transport on the spectral function applies only when one of the operators is differential.

The second step is to replace ${}^PV$ by a pseudodifferential operator ${}^P\tilde{V}$ whose full symbol satisfies
$$
{}^P\tilde{v}(x,\xi) =\chi(\rho^{-1}|\xi|){}^PV(x)
$$
for some $\chi \in C_c^\infty(0,\infty)$, with $\chi\equiv 1$ near $1$. This is done in Section~\ref{s:fourierMultiplier}. Next, we use the onion peeling gauge transform to replace ${}^P\tilde{V}$ by a Fourier multiplier, $V_1$, i.e. we find a unitary operator $U$ so that 
$$
U^*(-\Delta + {}^P\tilde{V})U= -\Delta +V_1.
$$
The existence of such an onion peeling operator, $U$, is proved in Section~\ref{s:gauge}, and this gauge transform is applied in Section~\ref{s:fourierMultiplier}.  It is then easy to compute  spectral function of $-\Delta +V_1$ in terms of $V_1$ and it remains to understand what conjugation by $U$ does to this function; in a sense `unpeeling' the onion. This final step is done in Section~\ref{s:unpeel}.
}

\subsection{Formulation of results on the local density of states}
\label{s:formulate}

\referee{Despite the fact that most of our results will be proved in dimension $1$, some will be proved in all dimensions. We therefore introduce notation in general dimension and then emphasize which results hold for $d=1$ and which for $d\geq 1$.}

We now \referee{\st{introduce the notation that will be used throughout the rest of the paper and}} formulate our results on the local density of states precisely. 
\referee{
\begin{definition}
\label{d:USBd}
We say that a smooth function, $V:\mathbb{R}^d\to \mathbb{C}$, is \emph{uniformly smoothly bounded (USB)}, writing $V\in \USB(\mathbb{R}^d)$, if for all $k\in \mathbb{N}$,
\begin{equation}
\label{e:Ck}
\|V\|_{C^k}:=\sum_{|\alpha|\leq k}\|\partial_x^\alpha V\|_{L^\infty}<\infty.
\end{equation}
We endow $\USB(\mathbb{R}^d)$ with the topology induced by the seminorms~\eqref{e:Ck}.
\end{definition}

\begin{definition}
We say that $Q:\mc{C}_c^\infty(\mathbb{R}^d) \to \mc{D}'(\mathbb{R}^d)$ is a \emph{differential operator of order $m$ with uniformly $C^k$ bounded coefficients}, and write $Q\in \Diff_k^m$ if 
$$
[Qu](x)=\sum_{|\alpha|\leq m } a_\alpha (x)D_x^\alpha u(x)
$$
with 
$$
\| a_\alpha \|_{C^k}<\infty,\qquad |\alpha|\leq m.
$$
We endow the space $\Diff_k^m(\mathbb{R})$ with the norm 
\begin{equation}
\label{e:semi}
\|Q\|_{\Diff_k^m}=\sum_{|\alpha|\leq m}\|a_\alpha\|_{C^k}.
\end{equation}
We also denote by $\Diff^m(\mathbb{R}^d)=\Diff^m_\infty(\mathbb{R}^d)=\cap_k\Diff^m_{k}(\mathbb{R}^d)$ the space of differential operators of order $m$ with uniformly smooth bounded coefficients. We endow $\Diff^m$ with the topology induced by the seminorms~\eqref{e:semi}.
\end{definition}
}
%

For a formally self-adjoint $Q\in \Diff_{k}^1(\mathbb{R}^{\referee{d}})$, consider an operator 
\begin{equation}
\label{e:Hv}
\normOp{Q}:=-\Delta+Q,
\end{equation}
acting in $L^2(\mathbb{R}^{\referee{d}})$. \referee{Recall that $\normOp{Q}$ is self-adjoint with domain $H^2(\mathbb{R}^d)$.}
\referee{
\begin{definition}
For a self-adjoint operator, $H$, we define
$$
\referee{E}(H)(\rho):=1_{(-\infty,\rho^2]}(H)
$$
 to be the spectral projector onto the spectrum of $H$ below $\rho^2$. For a subset $J\subset \mathbb{R}$, we also write 
  $$
  \referee{E}(H; J)=1_{J}(H)
  $$
  for the spectral projector onto the spectrum of $H(Q)$ in $J$.
 We define the \emph{spectral function for $\normOp{Q}$} to be the integral kernel
\begin{equation}
\label{e:spectralFunction}
\referee{E}(\normOp{Q})(\rho\referee{\,;\,}x,y)=1_{(-\infty,\rho^2]}(\normOp{Q})(x,y).
\end{equation}
  Note that, since $-\Delta+Q$ is elliptic, $\referee{E}(\normOp{Q})(\rho\referee{\,;\,x,y})$ is, in fact, a smooth function of $(x,y)$. (For a proof of smoothness see e.g.~\cite{AgKa:67},~\cite[B.7.1]{Si:82}~\cite{Si:84Err}. In fact, we also prove this below in~\eqref{e:smoothness}.) 
    \end{definition}}

\referee{Now, to state most of our main results, we specialise to the case $d=1$.}  Our first main theorem is a full asymptotic expansion for $\referee{E}(\normOp{Q})(\rho)$ and its derivatives on the diagonal. 
\begin{theorem}
\label{t:USBAsymptoticsNoh}
\referee{Let} $N,\referee{\tilde{M}}>0$. Then there is $K>0$ such that for any bounded subset $\mc{Q}\subset \Diff^1_{K}(\mathbb{R})$, there is $C>0$ such that for all formally self-adjoint $Q\in \mc{Q}$ and \referee{$\alpha,\beta\in \mathbb{N}$ with} $\alpha,\beta\leq \referee{\tilde{M}}$, there are $\coA_{j,\alpha,\beta}\in L^\infty(\mathbb{R})$ such that
\begin{equation}
\label{e:asymptotic2}
\Big|\partial_x^\alpha \partial_y^\beta \referee{E}(\normOp{Q})(\rho\,;\,x,y)|_{y=x}- \sum_{j=0}^{N-1} \coA_{j,\alpha,\beta}(x) \rho^{1+|\alpha|+|\beta|-j}\Big|\leq C \rho^{1+|\alpha|+|\beta|-N}.
\end{equation}
Moreover, $\coA_{0,0,0}\equiv\frac{1}{\pi}$, $\coA_{2\ell+1,0,0}\equiv 0$ for $\ell\geq 0$, $\coA_{j,0,0}(x)$ can be computed explicitly in terms of the coefficients of $Q$ and their derivatives at $x$, $\coA_{0,\alpha,\beta}\equiv0$ for $\alpha+\beta$ odd, and $\coA_{0,\alpha,\beta}\equiv\frac{(-1)^\beta}{\pi(\alpha+\beta+1)}$ for $\alpha+\beta$ even.
\end{theorem}

\begin{remark}
In fact, we prove Theorem~\ref{t:USBAsymptoticsNoh} (and all further theorems) for $N=K=+\infty$ \referee{in the sense that~\eqref{e:asymptotic2} holds for any $N$ with constant $C_N$ depending on $N$}. However, the proofs for finite $N, K$ follow in exactly the same way. To see this, it is only necessary to recall that pseudodifferential calculi modulo an error of a certain order, $\rho^{-N}$, rely on only finitely many derivatives of the symbols involved (see also Remark~\ref{r:roughSymbols}). \referee{The reason why finite regularity statements are important for us is that we need, for each $N$ and $M$, the map sending $Q\in \Diff^1$ to the constant, $C$, on the right-hand side of~\eqref{e:asymptotic2} to be continuous. We use this to derive uniform asymptotics for all $x\in \mathbb{R}$ from the asymptotics at a fixed point $x$.}
\end{remark}

\referee{
\begin{remark}
We emphasize here that Theorem~\ref{t:USBAsymptoticsNoh}, as well as the rest of the theorems in this paper, is proved only for \emph{differential} perturbations of the Laplacian. The reason we are unable to treat pseudodifferential perturbations here is that our proof uses crucially finite speed of propagation for the wave group corresponding to $H(Q)$ (see Lemma~\ref{l:waveCompare}). We suspect that the results still hold for pseudodifferential perturbations but do not pursue this.
\end{remark}
}

Our next theorem is a full asymptotic expansion for $\referee{E}(\normOp{Q})(\rho;x,y)$ and its derivatives when \referee{$x$ is not too close to $y$}.

\begin{theorem}
\label{t:USBAsymptoticsOffNoh}
\referee{Let} $N,\referee{\tilde{M}},\delta, R>0$. Then there is $K>0$ such that for any bounded subset $\mc{Q}\subset \Diff^1_{K}(\mathbb{R})$, there is $C>0$ such that for all formally self-adjoint $Q\in \mc{Q}$ there are $\coB^+_{j}(x,y)$ and $\coB^-_{j}(x,y)$, $j=0,1,\dots $, such that for all $\alpha,\beta\in \mathbb{N}$ with \referee{$\alpha,\beta\leq \referee{\tilde{M}}$} and all $(x,y)\in \mathbb{R}$, with $\rho^{-1+\delta}\leq |x-y|\leq R$, we have
\begin{multline*}
\Big|\partial_x^\alpha\partial_y^\beta\Big(\referee{E}(\normOp{Q})(\rho\,;\,x,y)- \sum_{j=0}^{N-1}\rho^{-j} \big(e^{i\rho|x-y|}{\coB}^+_{j}(x,y)+e^{-i\rho|x-y|}\coB^-_{j}(x,y)\big)\Big)\Big|\\\leq C \rho^{-N+|\alpha|+|\beta|}|x-y|^{-N}.
\end{multline*}
Moreover, $\coB_0^{\pm}(x,y)=\pm\frac{1}{2i\pi|x-y|}$.
\end{theorem}
One slightly surprising aspect of these results is that they hold no matter what type the high-energy spectrum $\normOp{Q}$ has: absolutely or singular continuous or even pure point. 
An immediate (trivial) corollary of these results is this:
\begin{corollary}
\label{c:1}
For all $N>0$ there is $C_N>0$ such that for all $\rho >1$, if  $[\rho^2,\referee{\rho^2+\e}]$ is a spectral gap of $H$, then $\referee{\e}<C_N\rho^{-N}$. 
\end{corollary}
Another observation, also quite obvious, is that for any $\rho$ we have
$$
\sup_x\Big(\lim_{\rho'\to\rho^+}\referee{E}(H(Q))(\rho'\,;\,x,x)-\lim_{\rho'\to\rho^-}\referee{E}(H(Q))(\rho'\,;\,x,x) \Big)<C_N\rho^{-N}
$$
for any natural $N$, uniformly in $x$. This immediately implies:
\begin{corollary}
\label{c:2}
For all $Q\in \referee{\Diff^1}(\mathbb{R})$ and any natural $N$ there is a constant $C_N>0$ such that for all $\rho>1$ and any non-trivial eigenfunction $u\in L^2(\R)$ of $\normOp{Q}$ with eigenvalue $\rho^2$ we have
\begin{equation}
\frac{||u||_{L^{\infty}(\R)}}{||u||_{L^{2}(\R)}}\le C_N\rho^{-N}.
\end{equation}
\end{corollary}
Note that this result holds not just for one particular  eigenfunction, but for \referee{any linear combination} of eigenfunctions from a thin energy window $[\rho^2,\rho^2+O(\rho^{-\infty})]$.

Another (less obvious) corollary of Theorem \ref{t:USBAsymptoticsNoh} is related to the behaviour of any solution of the equation  
\begin{equation}
\label{LP:sol1}
\normOp{Q}u=\rho^2u,
\end{equation}
not just solutions belonging to $L^2(\R)$. 
We define, for any differentiable function $u$, the (renormalised) {\it  energy density} of $u$ at $x$ by
\begin{equation}
ED(u;x)=ED(u;x,\rho):=|u(x)|^2+\rho^{-2}|u'(x)|^2.
\end{equation}

\begin{corollary}
\label{LP:cor1}
For all $Q\in \referee{\Diff^1}(\mathbb{R})$ and any natural $N$ there is a constant $C_N$ such that for any non-trivial solution $u$ of \eqref{LP:sol1} and any $a,b\in\R$, $|b-a|\le\rho^N$ we have
\begin{equation}
\label{e:energyEstimate}
\frac{ED(u;b)}{ED(u;a)}<(1+C_N\rho^{-1}).
\end{equation}
\end{corollary}
Estimate~\eqref{e:energyEstimate} shows that any solution to~\eqref{LP:sol1} is very close to a plane wave on extremely large scales.


Corollary \ref{LP:cor1} immediately implies the following result that was already obtained in \cite{DeFo:86}:

\begin{corollary}
\label{LP:corLya}
For all $Q\in \referee{\Diff^1}(\mathbb{R})$ and any natural $N$ there is a constant $C_N$ such that  any non-trivial solution $u$ of \eqref{LP:sol1} satisfies
\begin{equation}
\label{LP:Lya}
\lim\sup_{x\to\pm\infty} \frac {\ln|ED(u;x)|}{|x|}\le C_N\rho^{-N}.
\end{equation}
\end{corollary}
\referee{Indeed, it is easy to see from Corollary~\ref{LP:cor1} that for any $N$ and any large enough $\rho$ the energy density cannot grow or decay by more than a factor of $2$ over distance $\rho^N$.}

As was noticed in \cite{DeFo:86}, a trivial consequence of this is the following bound on the Lyapunov exponents. Consider the situation when $Q$ is a random potential sampled from a uniformly $\USB(\R)$ family of potentials. 
Then (under standard conditions on the randomness) with probability one the spectrum of $H$ is pure point and the limit of the LHS of \eqref{LP:Lya} exists for each eigenfunction; this limit is called {\it the Lyapunov exponent}. In this case, Corollary \ref{LP:corLya} shows that the Lyapunov exponent decays faster than any power of $\rho$ as $\rho\to\infty$. 

One interpretation of \referee{these corollaries} is that, despite the possibility that the high-energy spectrum of $H$ may have  point or singular continuous components, it `wants' to be absolutely continuous. 

\referee{
We will reformulate Corollaries~\ref{c:1} to~\ref{LP:corLya} in semiclassical language and prove them in Section~\ref{s:consequences}.}

\subsection{\referee{Formulation of results on the local density of states for semiclassical operators}}
\label{s:semiFormulate}

Throughout most of this paper, we prefer to work in the semiclassical setting, studying a family of operators depending on a small parameter $\hbar>0$, where one should think of $\hbar$ as $\rho^{-1}$. 
 \referee{When confusion may arise between a semiclassical object and its non-semiclassical counterpart, we denote the semiclassical object with bold letters. }
\referee{
\begin{definition}

We say that $\mathbf{Q}=\mathbf{Q}(\hbar):\mc{C}_c^\infty(\mathbb{R}^d) \to \mc{D}'(\mathbb{R}^d)$ is a \emph{semiclassical differential operator of order $m$ with uniformly $C^k$ bounded coefficients} and write $\mathbf{Q}\in \sem{Diff}_{k}^m(\mathbb{R}^d)$ if
$$
[\mathbf{Q}u](x)=\sum_{|\beta|\leq m} \mathbf{q}_\beta(x;\hbar)(\hbar D_x)^\beta u(x)
$$
and there are $\mathbf{q}_{\beta,l}\in \USB(\mathbb{R}^d)$, $l=0,1,\dots$ independent of $\hbar$ such that for all $N$
\begin{gather}
\label{e:semiDiff}
 \|\mathbf{Q}\|_{\sem{Diff}_{k}^m,N}:=\sup_{0<\hbar<1}\hbar^{-N}\sum_{|\alpha|\leq k}\Big\|\partial_x^\alpha\Big( \mathbf{q}_\beta(\,\cdot\,;\hbar)-\sum_{l=0}^{N-1}\mathbf{q}_{\beta,l}(\cdot)\hbar^l\Big)\Big\|_{L^\infty}<\infty.
\end{gather}
We endow $\sem{Diff}_k^m$ with the seminorms~\eqref{e:semiDiff}. We also denote by $\sem{Diff}^m(\mathbb{R}^d)=\cap_k\sem{Diff}^m_{k}(\mathbb{R}^d)$ the space of semiclassical differential operators of order $m$ with uniformly smooth bounded coefficients and endow it with the topology induced by the seminorms~\eqref{e:semiDiff}.
\end{definition}
}
 Finally, for a self-adjoint $ \mathbf{Q}\in \sem{Diff}^1_{k}(\mathbb{R}^d)$ we denote 
\begin{equation}
\label{e:semiHQ}
\semOp{\mathbf{Q}}:=-\hbar^2\Delta +\hbar \mathbf{Q}.
\end{equation}

\referee{
\begin{definition}
For $\omega\in \mathbb{R}$, we define
$$
\sem{E}(\semOp{\mathbf{Q}})(\omega):=1_{(-\infty, \omega^2]}(\semOp{\mathbf{Q}}).
$$
 to be the spectral projector onto the spectrum of $\semOp{\mathbf{Q}}$ below $\omega^2$ and the \emph{spectral function for $\semOp{\mathbf{Q}}$} to be its integral kernel
\begin{equation}
\label{e:semiSpectralFunction}
\sem{E}(\semOp{\mathbf{Q}})(\omega\,;\,x,y):=1_{(-\infty, \omega^2]}(\semOp{\mathbf{Q}})(x,y).
\end{equation}
  \end{definition}}

We note that if $Q\in \Diff^1_{k}(\mathbb{R}^d)$, then 
\bee
\label{LP:semi}
\hbar^2\normOp{Q}=\semOp{\mathbf{Q}}
\ene 
for some $\mathbf{Q}\in \sem{Diff}^1_{k}(\mathbb{R}^d)$ and, thus, $E(\normOp{Q})(\hbar^{-1})=\sem{E}(\semOp{\mathbf{Q}})(1)$. However, the opposite is not necessarily true: there are operators $\mathbf{Q}\in\sem{Diff}^1_{k}$ that cannot be obtained using \eqref{LP:semi} from any operator $Q\in \Diff^1_{k}$. \referee{For instance, $\mathbf{H}(\mathbf{Q})$ with $\mathbf{Q}=\hbar V(x)$ for some $V\in \USB$ cannot be written as $\hbar^2H(Q)$ for some $Q\in \Diff_k^1$ since the zeroth order term in $\hbar^2H(Q)$ is $O(\hbar^2)$.} 
As a result, we can recover our Theorems~\ref{t:USBAsymptoticsNoh} and~\ref{t:USBAsymptoticsOffNoh} by putting $\hbar=\rho^{-1}$, $\omega=1$ in the following, more general, results about the asymptotic behaviour of \referee{$\sem{E}(\semOp{\mathbf{Q}})$}. \referee{The next two theorems assume $d=1$.}
\begin{theorem}
\label{t:USBAsymptotics}
Let \referee{\st{$d=1$ and}} $N,\referee{\tilde{M}},a,b>0$ with $a\leq b$. Then there is $K>0$ such that for any bounded subsets $\mc{Q}\subset \sem{Diff}^1_{K}(\mathbb{R})$, there is $C>0$ such that for all formally self-adjoint $\mathbf{Q}\in \mc{Q}$ and \referee{$\alpha,\beta\in \mathbb{N}$ with} $\alpha,\beta\leq \referee{\tilde{M}}$, there are $\coA_{j,\alpha,\beta}\in \referee{\USB}( [a,b]\times \mathbb{R})$ such that for all $x\in \mathbb{R}$, $\omega\in[a,b]$, we have
\begin{equation}
\label{e:semiAsymptotic2}
\Big|\partial_x^\alpha \partial_y^\beta \sem{E}(\semOp{\mathbf{Q}})(\omega\,;\,x,y)|_{y=x}- \sum_{j=0}^{N-1} \coA_{j,\alpha,\beta}(\omega,x) \hbar^{-1-|\alpha|-|\beta|+j}\Big|\leq C \hbar^{-1-|\alpha|-|\beta|+N}.
\end{equation}
Moreover, $\coA_{0,0,0}=\frac{\omega^2}{\pi}$, $\coA_{2j+1,0,0}\equiv 0$ for $j\geq 0$, $\coA_{j,0,0}(x)$ can be computed explicitly in terms of the coefficients of $\mathbf{Q}$ and its derivatives at $x$, $\coA_{0,\alpha,\beta}\equiv 0$ for $\alpha+\beta$ odd, and $\coA_{0,\alpha,\beta}\equiv\frac{(-1)^\beta\omega^2}{\pi(\alpha+\beta+1)}$ for $\alpha+\beta$ even.
\end{theorem}

\begin{theorem}
\label{t:USBAsymptoticsOff}
Let \referee{\st{$d=1$ and}} $N,\referee{\tilde{M}},\delta, R,a,b>0$ with $a\leq b$. Then there is $K>0$ such that for any bounded subsets $\mc{Q}\subset \sem{Diff}^1_{K}(\mathbb{R})$, there is $C>0$ such that for all formally self-adjoint  $\mathbf{Q}\in \mc{Q}$ there are $\coB^+_{j}(\omega,x,y)$ and $\coB^-_{j}(\omega,x,y)$, $j=0,1,\dots $ such that for all $\alpha,\beta\in \mathbb{N}$ \referee{with $\alpha,\beta\leq \referee{\tilde{M}}$} and all $(x,y)\in \mathbb{R}$, with $\hbar^{1-\delta}\leq |x-y|\leq R$, $\omega\in[a,b]$, we have
\begin{multline*}
\Big|\partial_x^\alpha\partial_y^\beta\Big(\sem{E}(\semOp{\mathbf{Q}})(\omega\,;\,x,y)- \sum_{j=0}^{N-1}\hbar^{j} \Big(e^{i|x-y|\omega/\hbar}\coB^+_{j}(\omega,x,y)+e^{-i|x-y|\omega/\hbar}\coB^-_{j}(\omega,x,y)\Big)\Big|\\\leq C \hbar^{N-|\alpha|-|\beta|}|x-y|^{-N}.
\end{multline*}
Moreover, $\coB^{\pm}_{0}\equiv\pm\frac{1}{2i\pi|x-y|}$. 
\end{theorem}

\begin{remark}
In fact, the coefficients in Theorems~\ref{t:USBAsymptotics} and~\ref{t:USBAsymptoticsOff} can be differentiated in $\omega$. Note, however, that the error is in general \emph{not} differentiable in $\omega$. (See also Lemma~\ref{l:derivativesOmega}.)
\end{remark}

\begin{remark}
\referee{Given Theorems~\ref{t:USBAsymptotics} and~\ref{t:USBAsymptoticsOff}, one might wonder whether it is possible to write a single oscillatory integral that is equal to $\sem{E}(\semOp{\mathbf{Q}})(\omega,x,y)$ modulo $O(\hbar^\infty)$ for all $(x,y)$ in any compact set. Unfortunately, we do not see how this is possible using our current methods. See Remark~\ref{r:noGlue} for further explanation. }
\end{remark}

\subsection{Comparison of spectral functions with large perturbations at infinity}

\referee{We now discuss results that hold in any dimension $d\geq 1$.}
In Section~\ref{s:comparison}, we show that, for differential operators, $\mathbf{Q}$, acting on $C^\infty(\mathbb{R}^d)$, one can make large perturbations of $\semOp{\mathbf{Q}}$  `at infinity' without modifying the spectral function of $\semOp{\mathbf{Q}}$ substantially. Indeed, one can even make changes to $\semOp{\mathbf{Q}}$ which completely change the nature of the spectrum of $\semOp{\mathbf{Q}}$ but, nevertheless, result in small changes to the spectral function in compact sets. We postpone the statement of the precise results to Section~\ref{s:comparison} and instead give a simpler \referee{version of these results here}. 

Let $\mathbf{Q}\in \sem{Diff}^{\referee{0}}(\mathbb{R}^d)$ be formally self-adjoint and put 
$$
\referee{\sem{H}_0:=\semOp{\mathbf{Q}}=-\hbar^2\Delta+\hbar \mathbf{q}(x;\hbar).}
$$ 
Let also
\referee{
\begin{equation}
\label{e:H1Temp}
\sem{H}_1:=-\hbar^2\Delta+\hbar\tilde{\mathbf{q}}(x;\hbar)).
\end{equation}}
We assume that \referee{\st{ $\tilde{\mathbf{q}}_1\in \USB(\mathbb{R}^d)$, that}} $\tilde{\mathbf{q}}\geq -C$ for some $C>0$, that $\tilde{\mathbf{q}}\in C^\infty(\mathbb{R}^d)$, \referee{and that $\sem{H}_1$ is essentially self-adjoint}. \referee{Note that $\tilde{\bf{q}}$ need not be bounded above and so is not necessarily $\USB$.}  In our next theorem, we compare the spectral function of $\sem{H}_0$ with that of $\sem{H}_1$ \referee{under certain closeness assumptions on $\mathbf{q}$ and $\tilde{\mathbf{q}}$. (See Example~\ref{example:1}.) }


Let $x\in \mathbb{R}^d$, $0<a<b$, \referee{$Z>0$} and define
\begin{multline*}
T_{\max}(\hbar,x,a,b,\referee{Z})\\
:=\sup \Big\{T>0\,:\, \sup_{\omega\in[a,b]}|\sem{E}(\sem{H}_1)(\omega;x,x)-\sem{E}(\sem{H}_1)(\omega-\hbar T^{-1};x,x)|\leq  \frac{\referee{Z}}{T}\hbar^{1-d}\Big\}.
\end{multline*}

\referee{
\begin{remark}
We use the notation $T_{\max}$ because it determines the minimal scale at which we can smooth the spectral projector while still maintaining control on the error and, hence, will determine the maximal time for which we use the wave propagator in the proof of Theorem~\ref{t:simpleCompare} below.
\end{remark}
}

Let $B(0,R)$ denote the ball of radius $R$ in $\mathbb{R}^d$ centered at $0$ and \referee{let} $\cutMan \in C_c^\infty(B(0,2))$ with $\cutMan \equiv 1$ on $B(0,1)$. Then put $\cutMan_R(x):=\cutMan(R^{-1}x)$. Define also
$$
\delta_k(R;\hbar):=\hbar^{-1}\|(\sem{H}_0-\sem{H}_1)\cutMan_{R}\|_{H_{\hbar}^{-k}\to H_{\hbar}^{-k-1}}+\hbar^{-1}\|(\sem{H}_0-\sem{H}_1)\cutMan_{R}\|_{H_{\hbar}^k\to H_{\hbar}^{k-1}}.
$$
Then a simple consequence of Proposition~\ref{p:distantPerturb} is the following theorem.
\begin{theorem}
\label{t:simpleCompare}
Let $d\geq 1$, $\referee{Z}>0$, $\referee{0<a<b}$ \referee{and suppose that} $\sem{H}_0$ and $\sem{H}_1$ are as above. There is $k>0$ such that for any $\e>0$, $R_0>0$, there are $C>0$ and $\hbar_0>0$ such that for $0<\hbar<\hbar_0$, $R(\hbar)>R_0+2$, $x\in B(0,R_0)$, $\omega\in[a+\e,b-\e]$, $T(\hbar)<\min(T_{\max}(\hbar,x,a,b,\referee{Z}), \hbar^{-1}(R(\hbar)-R_0-2)/2)$, we have
\begin{equation}
\label{e:aprilShowers}
\Big|\sem{E}(\sem{H}_1)(\omega;x,x)-\sem{E}(\sem{H}_0)(\omega;x,x)\Big|\leq C\Big(\frac{\hbar^{1-d}}{T(\hbar)}+\hbar^{-d}T(\hbar)\delta_k(R(\hbar);\hbar)\Big).
\end{equation}
\end{theorem}
\referee{\begin{remark}
\label{r:123}
Below, to make Theorem~\ref{t:simpleCompare} useful, we will find $\mathbf{H}_1$, $T(\hbar)$, and $R(\hbar)$ so that $T(\hbar)\geq \hbar^{-N}$ but $T(\hbar)\delta_k(R(\hbar);\hbar)\leq \hbar^N$ so that the right hand side of~\eqref{e:aprilShowers} is very small. \end{remark}}

In fact, an analogue of Theorem~\ref{t:simpleCompare} holds much more generally and we can, for example, make $\sem{H}_1$ a \referee{pseudodifferential} operator or even replace the infinite end of $\mathbb{R}^d$ by a \referee{boundary lying} sufficiently far away without changing $\sem{E}(\sem{H}_0)$ substantially.  We now give a few examples where we can effectively apply Theorem~\ref{t:simpleCompare}.
\begin{example}
\label{example:1}
\begin{enumerate}
\item[]
\item $\sem{H}_1=\semOp{\mathbf{Q}+\delta(\hbar)[R(\hbar)]^{-2}|x|^2}$. Here,  $\delta_k(R(\hbar);\hbar)\leq C\delta(\hbar)$. Notice that, despite the fact that $\sem{H}_1$ has discrete spectrum, the kernel of its spectral projector is close to that of $\sem{H}_0$ in compact sets.
\item Assume that $\mathbf{Q}=\mathbf{q}(x)$, with. Let $\sem{H}_1=\semOp{\referee{{}^{\mc{M}}\!\mathbf{q}}}$, where
\referee{
$$
\Big|\partial_x^\alpha (\mathbf{q}-{}^{\mc{M}}\!\mathbf{q})(x)\Big|\leq C_\alpha \delta(\hbar)\qquad \text{for}\qquad |x|\leq R(\hbar).
$$}
\item Assume that $\mathbf{Q}= \mathbf{q}(x)$. Our aim is to make $\mathbf{Q}$ periodic. To do this, we introduce ${}^P\!\mathbf{q}\in \USB(\mathbb{R}^d)$, such that ${}^P\!\mathbf{q}$ is periodic and ${}^P\!\mathbf{q}(x)=\mathbf{q}(x)$ for $x\in B(0,\referee{2}R(\hbar))$. We then define $\sem{H}_1=\semOp{{}^P\!\referee{\mathbf{q}}}$.
\referee{This is, in fact, the type of modification we make use of to prove our main theorems. In this case, $\delta_k(R(\hbar);\hbar)=0$ and we will see below that $T_{\max}(\hbar,x,a,b)\geq c_N\hbar^{-N}$ for any $N$ and hence, provided $R(\hbar)\leq C\hbar^{-N}$, we may take $T(\hbar)\geq c\hbar^{-N}$ so that the right-hand side of~\eqref{e:aprilShowers} is small.} 
\end{enumerate}
\end{example}
\referee{\begin{remark} If  $\widehat{\referee{\mathbf{q}}}$ a measure, and the $1$-Wasserstein distance (see e.g.~\cite[Chapter 6]{Vi:09}) between $\widehat{\partial_{x}^\alpha \mathbf{q}}$ and $\widehat{\partial_{x}^\alpha {}^{\mc{M}}\!\mathbf{q}}$ is bounded by $C_\alpha R(\hbar)^{-1}\delta(\hbar)$, then one can check that the conditions in (2) are satisfied. This reformulation in terms of  measures is the reason why we (admittedly somewhat loosely) call ${}^\mc{M}\mathbf{Q}$ the mass transport of $\mathbf{Q}$. \end{remark}}
\begin{remark}
In fact, one can check \referee{a posteriori} from Theorem~\ref{t:USBAsymptotics} that for all of the above cases in 1 dimension and $x\in B(0,R_0)$, we have, for any $N>0$, \referee{there is $Z>0$ such that} $T_{\max}(\hbar,x,a,b,\referee{Z})\gtrsim_N \min(R(\hbar), \hbar^{-N})$. \referee{Indeed Theorem~\ref{t:USBAsymptotics} implies that 
\begin{align*}
&\sem{E}(\semOp{\mathbf{Q}})(\omega; x, x)-\sem{E}(\semOp{\mathbf{Q}})(\omega-\hbar T^{-1}; x, x)\\
&\leq \sum_{j=0}^{N}(f_{j,0,0}(\omega,x)-f_{j,0,0}(\omega-\hbar T^{-1},x))\hbar^{j-1} +C_N\hbar^{N}\\
&\leq \frac{C_N}{T}+C_N\hbar^{N}.
\end{align*}}
\end{remark}

\subsection*{Outline of the paper}
Section~\ref{s:notation} introduces some notation and conventions used throughout the paper. \referee{Section~\ref{s:technical} then introduces some technical lemmata used in the proof.} Next, Section~\ref{s:comparison} proves that changing a differential operator outside a large ball has a small effect on the spectral function at the origin, in particular proving Theorem~\ref{t:simpleCompare}. In Section~\ref{s:pdo}, we review the standard notions of semiclassical pseudodifferential operators and semiclassical Sobolev spaces. We also introduce and collect some facts about an anisotropic pseudodifferential calculus which will be used in the gauge transform procedure.  Section~\ref{s:gauge} implements the parallel-serial gauge transform via a layer peeling argument, and Section~\ref{s:compute} combines the results of the gauge transform and modification of the potential outside a large ball to compute the asymptotic formulae for the spectral function; proving Theorems~\ref{t:USBAsymptotics} and~\ref{t:USBAsymptoticsOff}. Section~\ref{s:consequences} then extracts various consequences of our main theorem on generalized eigenfunctions of Schr\"odinger operators, proving \referee{the semiclassical analogues of} Corollaries~\ref{c:1} to~\ref{LP:corLya}. Finally, Appendix~\ref{a:firstTerm} computes the first term of the asymptotic expansion for the spectral function.
\bigskip

\noindent{\textsc{Acknowledgements:}} The authors would like to thank Y. Karpeshina, G. Rozenblioum, and M. Zworski for comments on an earlier draft which improved the exposition. The authors also thank S. Sodin for pointing out the paper~\cite{DeFo:86}. \referee{We are also grateful to the anonymous referees for their detailed reading of the text and many helpful comments.} J.G. would like to thank the EPSRC for support under Early Career Fellowship:  EP/V001760/1 and grant EP/V051636/1. L.P. is grateful to the EPSRC for support under grants EP/P024793/1 and EP/V051636/1. R.S. acknowledges support from the NSF through grant DMS-1814664.

\section{Basic Notation}
\label{s:notation}

Before proceeding to the main body of the paper, we introduce some notation that will be used throughout the text. 

\subsection{Spaces of smooth functions}
For $A\subset \mathbb{R}^d$, $\{0\}\subset B\subset \mathbb{C}$,  we use the notation
\begin{gather*}
C^\infty(A;B):=\{ u: \mathbb{R}^d\to \mathbb{C}\,|\, u\text{ is smooth},\, \supp u\subset \overline{A},\,u(z)\in B\text{ for all }z\in \mathbb{R}^d\},\\
C_c^\infty(A;B):=\{ u: \mathbb{R}^d\to \mathbb{C}\,|\, u\text{ is smooth},\, \supp u\Subset A,\,u(z)\in B\text{ for all }z\in \mathbb{R}^d\}.
\end{gather*}
When $B=\mathbb{C}$, we sometimes write $C^\infty(A;\mathbb{C})$ and $C_c^\infty(A;\mathbb{C})$ as, respectively, $C^\infty(A)$ and $C_c^\infty(A)$. Furthermore, if $A=\mathbb{R}^d$, we write $C^\infty(\mathbb{R}^d)=C^\infty$ and $C_c^\infty(\mathbb{R}^d)=C_c^\infty$. 

Finally, we write $\mathscr{S}(\mathbb{R}^d)$ for the space of Schwartz functions and $\mathscr{S}'(\mathbb{R}^d)$ for its dual space.

\referee{Below, we will allow functions in the spaces $C_c^\infty$ and $\mathscr{S}$ to depend on the small parameter $\hbar$. In this case, we will assume that the seminorms of these functions are uniformly bounded in $\hbar$ and, in the case of $C_c^\infty$, that the union of their supports is bounded. }

\subsection{Fourier transforms}

For $f\in \mathscr{S}'$, we recall that 
\begin{equation}
\label{e:FourierDef}
\hat{f}(\xi):=\int e^{-i\langle x,\xi\rangle }f(x)dx,\qquad\text{and}\qquad \check{f}(x):=\frac{1}{(2\pi)^d}\int e^{i\langle x,\xi\rangle}f(\xi)d\xi
\end{equation}
denote the Fourier transform of $f$ and the inverse Fourier transform of $f$ respectively.

\subsection{Semiclassical Sobolev spaces}

Next, we define the semiclassically weighted Sobolev spaces, $H_{\hbar}^s(\mathbb{R}^d)$ as the closure of $\mathscr{S}(\mathbb{R}^d)$ with respect to the norm
$$
 \|u\|_{H_{\hbar}^s}^2:=\frac{1}{(2\pi \hbar)^d}\int |\hat{u}(\xi/\hbar)|^2\langle \xi\rangle^{2s}d\xi,\qquad \langle \xi\rangle:=(1+|\xi|^2)^{1/2}.
$$

\subsection{Big O notation}
\referee{For a function $f=f(\hbar):(0,1]\to \mathbb{R}_+$, a family of topological vector spaces $X=X(\hbar)$ with topology induced by the seminorms $\{\|\cdot \|_{\alpha_X(\hbar)}\}_{\alpha_X\in \mc{A}(\hbar)}$, and $u=u(\hbar):(0,1]\to X$ we write $u=O(f(\hbar))_{X}$ when for every $\alpha_X\in \mc{A}$, there exists $C>0$ such that
$$
 \|u\|_{\alpha_X(\hbar)}\leq Cf(\hbar),\, 0<\hbar\leq1.
$$
In a similar way, for two families of Banach spaces $X=X(\hbar)$, $Y=Y(\hbar)$, and $A=A(\hbar):X(\hbar)\to Y(\hbar)$, we write $A=O(f(\hbar))_{X\to Y}$ when $A=O(f(\hbar))_{\mc{B}(X,Y)}$. Here $\mc{B}(X,Y)$ denotes the Banach space of bounded operators from $X$ to $Y$. We write $ u=O(\hbar^\infty)_X$  if $u=O(\hbar^N)_{X}$ for any $N>0$.}

\subsection{Cutoffs}

Throughout the text, we require a variety of smooth cutoff functions. Although we do not wish to fix these cutoffs once and for all, we introduce notation that indicates the role of each cutoff function.
\begin{center}
\renewcommand{\arraystretch}{1.3}
\begin{tabular}{|c|c|}
\hline
$\cutMan$ & Cutoffs in the physical space (where $x$ lives)\\
\hline
$\cutFreq$ & Cutoffs in the momentum space (where $\xi$ lives)\\
\hline
$\cutPhase$ & Cutoffs in the phase space (where $(x,\xi)$ lives)\\
\hline
$\cutEnergy$ & Compactly supported cutoffs in energy\\
\hline
$\cutEnergyFourier$ & Small scale ($\ll \hbar $) cutoffs in energy, usually with compact Fourier support\\
\hline
$\cutThetaFreq$ & Cutoffs in the dual to the physical space (usually with variable $\theta$)\\
\hline
$\auxCut$ & Other types of auxilliary cutoffs\\
\hline
\end{tabular}
\end{center}

When using these cutoffs in our analysis, we will not distinguish between the cutoff and the operator of multiplication by the cutoff. For example, we will write $\cutMan$ for both a function $\cutMan\in C^\infty(\mathbb{R}^d)$ and for the operator of multiplication by $\cutMan$ given by $[\cutMan(u)](x):=\cutMan(x)u(x).$

\subsection{Conventions on a discrete valued large parameter}
\label{s:muN}

Throughout the text, we work with functions of a small parameter $\hbar\in (0,1\referee{]}$. We will also want a discrete valued large parameter which plays the role of the scale of $\hbar^{-1}$.  To this end, we \referee{let} 
\begin{equation}
\label{e:defMoo}
\mu_n:=2^n
\end{equation}
 and work with functions $n=n(\hbar):(0,1\referee{]}\to \mathbb{N}$ such that 
\begin{equation}
\label{e:muNBound}
\frac{1}{1024}\hbar^{-1}\leq \mu_{n(\hbar)}\leq  1024\hbar^{-1}.
\end{equation}
In other words, 
$$
-10\leq n(\hbar)+\log_2\hbar\leq 10.
$$
The main ingredient in the proof of Theorem~\ref{t:USBAsymptotics} (and similarlty for Theorem~\ref{t:USBAsymptoticsOff}) is to establish that there are $\tilde{f}_{j,\alpha,\beta}:[a,b]\times \mathbb{R}\times \mathbb{N}\to \mathbb{R}$ such that for any $n(\hbar)$ satisfying~\eqref{e:muNBound}, we have
\begin{equation*}
\Big|\partial_x^\alpha \partial_y^\beta \sem{E}(\semOp{\mathbf{Q}})(\omega\,;\,x,y)|_{y=x}- \sum_{j=0}^{N-1}\tilde{ \coA}_{j,\alpha,\beta}(\omega,x,n(\hbar)) \hbar^{-1-|\alpha|-|\beta|+j}\Big|\leq C \hbar^{-1-|\alpha|-|\beta|+N}.
\end{equation*}
\referee{Since}, for each $\hbar\in (0,1\referee{]}$ small enough, we have several possible choices of $n(\hbar)$, we will then be able to use gluing arguments from~\cite{PaSh:16} to establish Theorem~\ref{t:USBAsymptotics}. 

\referee{Our goal is to obtain a full, polyhomogenous expansion of the spectral function in powers of $\hbar$. The reason we do not directly work with $\lp=\hbar^{-1}$ instead of $\lp_{n(\hbar)}$ is that, with the former choice, many of the operations we perform would not preserve polyhomogeneity in $\hbar$; for example, the decomposition used in the onion peeling argument does not preserve polyhomogeneity if $\lp=\hbar^{-1}$. We would like to emphasise that choosing $\lp=\hbar^{-1}$ still results in a formula for the spectral projector, it is just not clear that this formula has an expansion in powers of $\hbar$. Our method for recovering polyhomogeneity is inspired by that in~\cite{PaSh:09} and is based on the idea that this decomposition should \emph{not} depend on $\hbar$ for $\hbar$ in some small interval and hence, since we have several choices of the decomposition, we may glue the asymptotics in each interval. The reader familiar with~\cite{PaSh:09,PaSh:12,PaSh:16} should notice that $\lp_n$ here plays the same role as $\rho_n$ there.}

\referee{
\section{Abstract technical estimates}
\label{s:technical}
In this section we present technical estimates inspired by~\cite{PaSh:16} which will be used below.

Before proceeding to these estimates, we discuss the natural requirements for the spectral function of two operators to be close. First, notice that closeness of two operators, $H_1$ and $H_2$ in \emph{any} norm does not suffice for the spectral projectors, $E(H_j)(\lambda)$ to be close to each other. Indeed, an eigenvalue of $H_1$ may be perturbed out of $(-\infty,\lambda]$ and hence, a small perturbation may cause a large change in the spectral projector.  In addition to closeness of $H_1$ and $H_2$, we use the fact that
$$
E(H_1)(\lambda,x,y)=\langle E(H_1)(\lambda)\delta_x,\delta_y\rangle.
$$ 
In particular, an important ingredient in the proof is the smallness of 
\begin{equation}
\label{e:lipLP}
E(H_2;(\lambda-\iota,\lambda+\iota])\delta_x= E(H_2)(\lambda+\iota,x,x)-E(H_2)(\lambda-\iota,x,x)
\end{equation}
for small $\iota$.

We first recall~\cite[Lemma 4.2]{PaSh:16} that states that if two operators are close, then one can control the difference between their spectral projectors in the strong topology.
\begin{lemma}
\label{l:abstralSpectral}
Let $\mc{H}$ be a Hilbert space, $a\in \mathbb{R}$, $s\geq 0$, and $H_1, H_2$ be self-adjoint operators on $\mc{H}$ with $H_j\geq a$ for $j=1,2$. Define 
$$
\e:=\|(H_1-H_2)(H_2+(1-a)I)^s\|.
$$
Then, if $\e<1$, for any $f\in \mc{H}$, $\lambda\geq a+1$, and $\iota>0$,  we have
\begin{multline*}
\|[E(H_1)(\sqrt{\lambda})-E(H_2)(\sqrt{\lambda})]f\|_{\mc{H}}\leq 2\|E(H_2;[\lambda-\iota,\lambda+\iota])f\|_{\mc{H}}\\
+\frac{2\pi \e}{\iota}\Big(\|E(H_2)(\sqrt{\lambda})f\|_{\mc{H}}+\|(H_2+(1-a)I)^{-s}f\|_{\mc{H}}\Big).
\end{multline*}
\end{lemma}

We will actually need a slightly stronger version of Lemma~\ref{l:abstralSpectral} which, heuristically, says that if two operators are close near a particular energy level, then their spectral projectors are close in the strong topology near that energy level (see Lemma~\ref{l:abstralSpectralNew}). First, we prove the following lemma. 
\begin{lemma}
\label{l:abstralSpectralA}
Let $\mc{H}$ be a Hilbert space, $a\in \mathbb{R}$, $s\geq 0$, $J\subset \mathbb{R}$ an interval and $H_1, H_2$ be self-adjoint operators on $\mc{H}$ with $H_j\geq a$ for $j=1,2$. Define $J_-:=J^c\cap (-\infty,\inf J]$ and $J_+:=J^c\cap [\sup J,\infty)$, and
\begin{equation}
\label{e:theEpsilons}
\begin{aligned}
\e_1&:=\|E(H_1;J_-)(H_1-H_2)E(H_2;J_+)(H_2+(1-a)I)^s\|,\\
\e_2&:=\|(H_1-H_2)E(H_2;J)(H_2+(1-a)I)^s\|,\\
\e_3&:=\|E(H_1;J)(H_1-H_2)(H_2+(1-a)I)^s\|.
\end{aligned}
\end{equation}
Suppose that $\lambda -a \geq 1$ and $[\lambda-\iota,\lambda+\iota]\subset J$. Then,
$$
\|E(H_1;(-\infty,\lambda-\iota])E(H_2;[\lambda+\iota;\infty))(H_2-a+1)^s\|\leq \frac{\pi (\e_1+\e_2+\e_3)}{\iota}.
$$
\end{lemma}
\begin{proof}
We follow the proof of~\cite[Lemma 4.1]{PaSh:16}. Assume that
\begin{equation}
\label{e:phiForm}
\phi = E(H_1;(-\infty,\lambda-\iota])\phi,\qquad (H_2-a+1)^s\psi=E(H_2;[\lambda+\iota,\infty))(H_2-a+1)^s\psi,
\end{equation}
with $\|\phi\|=\|\psi\|=1$. Then we need to establish $|(\phi,(H_2-a+1)^s\psi)|\leq \frac{\pi(\e_1+\e_2+\e_3)}{\iota}$. Following the algebra in~\cite[Lemma 4.1]{PaSh:16}, we have 
\begin{align*}
(\phi, (H_s-a+1)^s\psi)&=\int_\gamma \langle (H_1-z)^{-1}\phi,(H_2-a+1)^s\psi\rangle dz\\
&=\int_\gamma \langle \phi, (H_1-\bar{z})^{-1}(H_2-a+1)^s\psi\rangle dz\\
&=\int_\gamma\langle\phi,(H_2-\bar{z})^{-1}+(H_1-\bar{z})^{-1}(H_1-H_2)(H_2-\bar{z})^{-1}(H_2-a+1)^s\psi\rangle dz\\
&=\int_\gamma \langle \phi, (H_1-\bar{z})^{-1}(H_1-H_2)(H_2-\bar{z})^{-1} (H_2-a+1)^s\psi\rangle dz\\
&=\int_\gamma ((H_1-z)^{-1}\phi, (H_1-H_2)(H_2-a+1)^s(H_2-\bar{z})^{-1}\psi\rangle dz
\end{align*}
where $\gamma=\gamma_N$ is the closed rectangular contour in the complex plane symmetric about $\mathbb{R}$ and intersecting $\mathbb{R}$ at $\lambda$ and $-N$ where $N>-a$ is large. Note that in the next to last line we have used that with $\bar{\gamma}$ the contour conjugate to $\gamma$,
$$
\int_{\bar{\gamma}}\langle (H_2-\bar{z})^{-1}E((\lambda+\iota,\infty];H_2)d\bar{z}=0.
$$
Now,
\begin{multline*}
(H_1-H_2)(H_2-a+1)^s\\
= E(H_1;J)(H_1-H_2)(H_2+(1-a)I)^s+E(H_1;J^c)(H_1-H_2)E(H_2;J^c)(H_2+(1-a)I)^s\\
+E(H_1;J^c)(H_1-H_2)E(H_2;J)(H_2+(1-a)I)^s.
\end{multline*}
Therefore, we need only to estimate the three terms
\begin{align*}
I&:=\Big|\int_\gamma ((H_1-z)^{-1}\phi, E(H_1;J^c)(H_1-H_2)E(H_2;J^c)(H_2-a+1)^s(H_2-\bar{z})^{-1}\psi)dz\Big|,\\
II&:= \Big|\int_\gamma ((H_1-z)^{-1}\phi, E(H_1;J)(H_1-H_2)(H_2-a+1)^s(H_2-\bar{z})^{-1}\psi)dz\Big|,\\
III&:=\Big|\int_\gamma ((H_1-z)^{-1}\phi, E(H_1;J^c)(H_1-H_2)E(H_2;J)(H_2-a+1)^s(H_2-\bar{z})^{-1}\psi)dz\Big|.
\end{align*}
For $I$, we observe using~\eqref{e:phiForm} that
\begin{align*}
I&= \Big|\int_\gamma ((H_1-z)^{-1}\phi, E(H_1;J_-)(H_1-H_2)E(H_2;J_+)(H_2-a+1)^s(H_2-\bar{z})^{-1}\psi)dz\Big|\\
&\leq \e_1\Big(\int_\gamma \|(H_1-z)^{-1}\phi\|^2|dz|\Big)^{1/2}\Big(\int_\gamma \|(H_2-z)^{-1}\psi\|^2|dz|\Big)^{1/2}\leq \frac{\pi\e_1}{\iota}.
\end{align*}
Similarly, we estimate 
$$
II+III\leq \frac{\pi (\e_2+\e_3)}{\iota}
$$
to finish the proof.
\end{proof}

The proof of the next lemma is identical to that of~\cite[Lemma 4.2]{PaSh:16} after replacing references to~\cite[Lemma 4.1]{PaSh:16} with references to Lemma~\ref{l:abstralSpectralA}.
\begin{lemma}
\label{l:abstralSpectralNew}
Let $\mc{H}$ be a Hilbert space, $a\in \mathbb{R}$, $s\geq 0$, and $H_1, H_2$ be self-adjoint operators on $\mc{H}$ with $H_j\geq a$ for $j=1,2$. Define $\e_1,\e_2,\e_3$ as in~\eqref{e:theEpsilons}.
Then, if $\e_1+\e_2+\e_3<1$, for any $f\in \mc{H}$, $\lambda\geq a+1$, and $\iota>0$, 
\begin{multline}
\label{e:strongProject}
\|[E(H_1)(\sqrt{\lambda})-E(H_2)(\sqrt{\lambda})]f\|_{\mc{H}}\leq 2\|E(H_2;[\lambda-\iota,\lambda+\iota])f\|_{\mc{H}}\\
+\frac{2\pi (\e_1+\e_2+\e_3) }{\iota}\Big(\|E(H_2)(\sqrt{\lambda})f\|_{\mc{H}}+\|(H_2+(1-a)I)^{-s}f\|_{\mc{H}}\Big).
\end{multline}
\end{lemma}
\begin{remark}
Given an operator $H_1$ our strategy will be to find an operator $H_2$ so that: 
\begin{itemize}
\item $H_1$ is close to $H_2$ in some sense
\item \eqref{e:lipLP} is small and hence the first term on the right-hand side of~\eqref{e:strongProject} is small.
\end{itemize} In fact, the smallness of~\eqref{e:lipLP} will be guaranteed by the existence of a full asymptotic expansion for the spectral function of $H_2$.
\end{remark}

We now state a small generalisation of~\cite[Lemma 3.6]{PaSh:16} which will be used to glue asymptotic expansions that work in closed intervals of $\hbar$ into a uniform asymptotic expansion for $\hbar\in (0,1]$. While the proof is almost identical to ~\cite[Lemma 3.6]{PaSh:16}, we present it here in Appendix~\ref{s:glue} for completeness and to accommodate the semi-classical notations from the present text.
\begin{lemma}[The gluing lemma]
\label{l:glue}
Let $p,\iota>0$,  $\xi_1,\xi_2\in\mathbb{R}$ with $\xi_1\neq \xi_2$, and suppose that for any $M>0$, there is $N>0$ such that 
\begin{equation}\label{preasymp}
f(\hbar)= e^{\frac{i}{\hbar}\xi_1}\sum_{j=0}^{N} a_{j,n(\hbar)} \hbar^{jp} + e^{\frac{i}{\hbar}\xi_2}\sum_{j=0}^{N} b_{j,n(\hbar)}  \hbar^{jp} +O(\lp_{n(\hbar)}^{-M}),
\end{equation}
for 
$$
-10\leq n(\hbar)+\log_2\hbar\leq 10,
$$
where $a_{j,n},b_{j,n}\in \mathbb{C}$, $j=0,1,\dots,$ and
\begin{equation}\label{asympest}
|a_{j,n}|+|b_{j,n}|\leq C_{j}\lp_n^{jp(1-\iota)}.
\end{equation}
Then there are $a_j',b_j'\in \mathbb{C}$, $j=0,1,\dots$ and for any $M>0$ there is $N'>0$ such that 
\begin{equation}\label{asymp}
f(\hbar)= e^{\frac{i}{\hbar}\xi_1}\sum_{j=0}^{N'} a_{j}' \hbar^{jp} + e^{\frac{i}{\hbar}\xi_2}\sum_{j=0}^{N'} b_{j}'  \hbar^{jp} +O(\hbar^{M}).
\end{equation}
If \eqref{preasymp} is uniform on a compact subset of $0<|\xi_1-\xi_2|<\infty$ then \eqref{asymp} is uniform on the same set.
\end{lemma}
}

\section{Comparison of spectral functions}
\label{s:comparison}

In this section, we show that one can make large perturbations of the potential outside a very large ball without modifying the local density of states for $\semOp{\mathbf{Q}}$ substantially. 
  
 In our applications, \referee{$\mathbf{Q}\in \sem{Diff}^1(\mathbb{R}^d)$ and} the change we make to $\semOp{\mathbf{Q}}$ replaces $\mathbf{Q}$ by a differential operator with periodic coefficients, $\referee{{}^P}\mathbf{Q}$, and hence does not change the \referee{domain of $\semOp{\mathbf{Q}}$}. However, we will see below that the fact that waves for $\semOp{\mathbf{Q}}$ travel at finite speed implies that \emph{any} reasonable perturbation of $\semOp{\mathbf{Q}}$ made outside of a large ball affects the local density of states for $\semOp{\mathbf{Q}}$ only mildly.

 We now set up some abstract assumptions with which we work throughout this section. Let $\referee{\mathscr{M}}$ be a smooth (potentially non-compact) manifold \referee{without boundary} with a Riemannian metric $g$ and 

 \begin{equation}
 \label{e:H0a}
 \sem{H}_0:=-\hbar^2\Delta_g + \hbar \mathbf{Q}:\mathcal{D}(\sem{H}_0)\to L^2(\referee{\mathscr{M}}),\qquad \mathcal{D}(\sem{H}_0)\subset L^2(\mathscr{M})
 \end{equation}
 with $\mathbf{Q}\in \sem{Diff}^1(\referee{\referee{\mathscr{M}}})$ formally self-adjoint. 
 
We assume that for all $s\in \mathbb{R}$, there is $C_s>0$ such that 
 \begin{equation}
 \label{e:assumed0}
 \|\mathbf{Q}\|_{H_{\hbar}^s\to H_{\hbar}^{s-1}}\leq C_s,\qquad 0<\hbar<1.
 \end{equation}
 
 \referee{
 \begin{definition}
We say that a family of functions $\cutMan=\{\cutMan(\hbar)\}_{0<\hbar<1}$ with $\cutMan(\hbar)\in C^\infty(\referee{\mathscr{M}})$ is \emph{semiclassical USB} and write $\cutMan\in \USB(\referee{\mathscr{M}})$ if for all $s$
 $$
 \sup_{0<\hbar<1}\|\cutMan(\hbar)\|_{H_\hbar^s(\mathscr{M})\to H_\hbar^s(\mathscr{M})}<\infty.
 $$
 \end{definition}
 }

 \referee{We now set up an abstract scheme which will allow us to compare the spectral projector of an operator with that of $\sem{H}_0$.}
 \referee{\begin{definition}
\referee{Let} $x_0\in \referee{\mathscr{M}}$ and a decreasing, positive function $R=R(\hbar)$ (usually, $\lim_{\hbar\to 0^+}R(\hbar)=\infty$) .  Let  $B_{\referee{\mathscr{M}}}(x_0,R(\hbar))$ be the metric ball of radius $R(\hbar)$ around $x_0$. We say that a \emph{family of expanding box Hilbert spaces} is
 $$\mc{H}=\mc{H}(\hbar):=L^2(B_{\referee{\mathscr{M}}}(x_0,R(\hbar)))\oplus \mc{H}_{\infty}$$
for some family of Hilbert spaces $\mc{H}_\infty= \mc{H}_\infty(\hbar)$. We call $\mc{H}_\infty$ the \emph{exterior Hilbert space}. 
\end{definition}
\begin{remark}
In all of the items from Example~\ref{example:1}, $\mc{H}_\infty=L^2(\mathbb{R}^d\setminus B(0,R(\hbar)))$. 
\end{remark}

\begin{definition}
We write $1_{B_{\referee{\mathscr{M}}}(x_0,R(\hbar))}:\mc{H}\to L^2(B_{\referee{\mathscr{M}}}(x_0,R(h)))$ for the orthogonal projection and, for $\cutMan\in\USB(\mathscr{M})$ with $\supp \cutMan\subset B_{\referee{\mathscr{M}}}(x_0,R(\hbar))$ and $(u_1,u_2)\in \mc{H}$, we write
 $$
 \cutMan u=(\cutMan u_1,0),
 $$
 and identify $\cutMan u$ with an element of $L^2(B_{\referee{\mathscr{M}}}(x_0,R(\hbar)))$. 
\end{definition}

}
\referee{
  \begin{definition}
  \label{d:expandBox}
  Let $\mc{H}$ be an expanding box Hilbert space with exterior Hilbert space $\mc{H}_\infty(\hbar)$. Let $\sem{H}_1(\hbar):\mc{H}(\hbar)\to \mc{H}(\hbar)$ be a family of unbounded, self-adjoint operators with dense domain $\referee{\mc{D}_{\hbar}}$.
  For $s\geq 0 $, we let $\referee{\mc{D}_{\hbar}}^s$ be the domain of $\sem{H}_1^s$ with the norm
 $$
 \|u\|_{\referee{\mc{D}_{\hbar}}^s}:=\|u\|_{\mc{H}}+\|\sem{H}_1^su\|_{\mc{H}}
 $$ 
 and for $s<0$, we let $\referee{\mc{D}_{\hbar}}^s:=(\referee{\mc{D}_{\hbar}}^{-s})^*$ with the implied norm. 
We say $\sem{H}_1$ is a \emph{family of expanding box operators for $\sem{H}_0$} if:
  \begin{itemize}
  \item $\sem{H}_1\geq -C_{\sem{H}_1}\hbar$.  
  \item $1_{B_{\referee{\mathscr{M}}}(x_0,R(h))}\referee{\mc{D}_{\hbar}}\subset H_{\hbar}^2(B_{\referee{\mathscr{M}}}(x_0,R(\hbar)))$. 

  \item   For any $\cutMan,\cutMan_+\in \USB(\referee{\mathscr{M}})$, with $\supp \cutMan_+ \subset B_{\referee{\mathscr{M}}}(x_0,R(\hbar))$ and \referee{$\supp (1-\cutMan_+)\cap \supp \cutMan=\emptyset$} the following holds. For all $s>0$ there is $C_{\cutMan,s}>0$ such that
 \begin{equation}
 \label{e:assumed1}
 \| \cutMan \sem{H}_1 (1-\cutMan_+)\|_{\referee{\mc{D}_{\hbar}}^{-s}\to H_\hbar^{s}}+ \|(1-\cutMan_+) \sem{H}_1\cutMan\|_{H_{\hbar}^{-s}\to \referee{\mc{D}_{\hbar}}^s}\leq C_{\cutMan,s}\hbar^s,\qquad 0<\hbar<1.
 \end{equation}
 \item To guarantee that the spectral functions for $\sem{H}_0$ and $\sem{H}_1$ are close near $x_0$, we similarly assume that for all $\cutMan, \cutMan_+$, and $s$ as above there are $C_{\cutMan,s}>0$ and $\referee{\tilde{\delta}}=\referee{\tilde{\delta}}(\hbar):(0,1]\to [0,1]$ such that
 \begin{equation}
 \label{e:assumed}
 \begin{gathered}
\|(\sem{H}_0-\sem{H}_1)\cutMan\|_{H_\hbar^s\to \referee{\mc{D}_{\hbar}}^{(s-1)/2}}\leq C_{\cutMan,s}\hbar \referee{\tilde{\delta}}(\hbar),\\
  \|\cutMan u\|_{H_{\hbar}^s}\leq C_{\cutMan,s}\|u\|_{\referee{\mc{D}_{\hbar}}^{s/2}},\quad u\in \mc{D}_{\hbar}^{s/2},\\ \|\cutMan u\|_{\referee{\mc{D}_{\hbar}}^{s/2}}\leq C_{\cutMan,s}\|\cutMan_+u\|_{H_\hbar^s},\quad u\in H_\hbar^s(\referee{\mathscr{M}}),
 \end{gathered}\quad\text{ for all }0<\hbar<1.
 \end{equation}

 \end{itemize}
 
   \end{definition}
 } 

  \begin{remark}
 Since our operators, $\sem{H}_1$, will be close to $\sem{H}_0$ on $L^2(B_{\referee{\mathscr{M}}}(x_0,R(\hbar)))$,  we think of the subspace $\mc{H}_\infty$ as the part of $\mc{H}$ `at infinity'.
 \end{remark}
 \begin{remark}
 The function $\referee{\tilde{\delta}}(\hbar)$ controls how closely $\sem{H}_1$ approximates $\sem{H}_0$ \referee{on the ball of radius $R(\hbar)$ around $x_0$}. Similarly, we will later choose a function $T=T(\hbar):(0,1]\to (0,\infty) $ which controls the length of time we will propagate waves in our arguments.
 \end{remark}
 
 \referee{
 \begin{remark}
 The language used in defining expanding box operators is inspired by the black box formalism from~\cite{SjZw:91}, but notice that in our setting the `black box' is exterior rather than interior.
 \end{remark}
 }

 We now provide some examples of $\sem{H}_1$ when $\sem{H}_0=\semOp{Q}$ for some $Q\in \USB(\mathbb{R}^d)$.
 
 \noindent{\bf{Examples:}}
\begin{enumerate}
\item\label{ex:1} $\sem{H}_1=\semOp{Q_1}$ for $Q_1$ a $100R(h)\mathbb{Z}^d$-periodic function with $Q_1(x)=Q(x)$ for $x\in B(0,R(h))$.  In this case, $\referee{\tilde{\delta}}(\hbar)=0$, $\mc{H}=L^2(\mathbb{R}^d)$. This is the transformation we will use to prove Theorems~\ref{t:USBAsymptotics} and~\ref{t:USBAsymptoticsOff}.
\item\label{ex:2} $\sem{H}_1$ is the Dirichlet realization of $\sem{H}_0$ on $B(0,R(h)+1)$ i.e. $\sem{H}_1 =\semOp{Q}$, $\mc{H}=L^2(B(0,R(h)+1))$, $\referee{\mc{D}_{\hbar}}=H_{0}^1(B(0,R(h)+1))\cap H^2(B(0,R(h)+1))$. 
\item\label{ex:3} $\sem{H}_1=\semOp{Q+\referee{\tilde{\delta}}(\hbar)R(\hbar)^{-2}|x|^2}$ with $\mc{H}=L^2(\mathbb{R}^d)$. 
 \end{enumerate}
 Notice that Examples~\ref{ex:2} and~\ref{ex:3} both have \emph{discrete} spectrum, while $\sem{H}_0$ may have pure a.c. spectrum. Nevertheless, our next proposition shows that one can approximate the spectral projector of $\sem{H}_0$ using that of $\sem{H}_1$ (or vice versa).

 
In this section we prove the following proposition which allows us to compare the spectral functions for $\sem{H}_0$ and $\sem{H}_1$.
 \begin{proposition}
 \label{p:distantPerturb}
Let $x,y\in B_{\referee{\mathscr{M}}}(x_0,R_0)\subset \referee{\mathscr{M}}$, $0<a<b$, $R(\hbar)>0$, $\referee{\tilde{\delta}}(\hbar)\geq 0$,  $\frac{a}{2}>\e>0$, $\e\leq T(\hbar)\leq (R(\referee{\hbar})-R_0-2)/2$, $C_{\sem{H}_1}>0$, $C_s>0$ and $C_{\cutMan,s}>0$. Then for all $C_1>0$ there is $C_0>0$ such that the following holds.  Suppose \referee{$H_0$ satisfies~\eqref{e:H0a},~\eqref{e:assumed0}, that $\sem{H}_1$ is a family of expanding box operators for $\sem{H}_0$} and for $0<\hbar<1$, $\omega\in[a-\e,b+\e]$, and $\lambda\in [-\e,\e]$,
\begin{equation}
\label{e:lip}
\begin{aligned}\Big|\sem E(\sem{H}_1)(\omega+\lambda;x,x)-\sem E(\sem{H}_1)(\omega;x,x)\Big|\leq C_1\frac{\hbar^{1-d}}{T(\hbar)}(1+|\hbar^{-1}T(\hbar)\lambda|).
\end{aligned}
\end{equation}
Then for $\omega\in[a,b]$ and $0<\hbar<1$,
\begin{equation}
\label{e:bigBallOk}
\begin{aligned}
|\sem E(\sem{H}_0)(\omega;x,x)-\sem E(\sem{H}_1)(\omega;x,x)|\leq \frac{C_0 \hbar^{1-d}}{T(\hbar)}\referee{\Big(1+\hbar^{-1}\referee{\tilde{\delta}}(\hbar)T^2(\hbar)\Big)}.\end{aligned}
\end{equation}
If, in addition, for all $\alpha,\beta\in \mathbb{N}^d$ with $|\alpha|\leq k,\,|\beta|\leq l$, 
\begin{equation}
\label{e:lip2}
\Big|(\hbar \partial_{x})^{\alpha}(\hbar \partial_y)^\beta\Big(\sem E(\sem{H}_1)(\omega+\lambda;x,y)-\sem E(\sem{H}_1)(\omega;x,y)\Big)\Big|\leq C_1\frac{\hbar^{1-d}}{T(\hbar)}(1+|\hbar^{-1}T(\hbar)\lambda|),
\end{equation}
then for all $\alpha,\beta\in \mathbb{N}^d$ with $|\alpha|\leq k,|\beta|\leq l$,
\begin{equation}
\label{e:bigBallOk2}
\Big|(\hbar \partial_{x})^{\alpha}(\hbar \partial_y)^\beta\Big(\sem E(\sem{H}_0)(\omega;x,y)-\sem E(\sem{H}_1)(\omega;x,y)\Big)|\leq \frac{C_0 \hbar^{1-d}}{T(\hbar)}\referee{\Big(1+\hbar^{-1}\referee{\tilde{\delta}}(\hbar)T^2(\hbar)\Big)}.
\end{equation}
\end{proposition}
\referee{
\begin{remark}
Observe that the assumption~\eqref{e:lip} is precisely the same as the assumption that~\eqref{e:lipLP} is small and hence that the first term on the right-hand side of~\eqref{e:strongProject} is small.
\end{remark}
}

Proposition~\ref{p:distantPerturb} immediately implies Theorem~\ref{t:simpleCompare}. Let $\mathbf{Q}\in \sem{Diff}^{\referee{0}}$, $\sem{H}_0=\semOp{\mathbf{Q}}$, and $\sem{H}_1$ as in~\eqref{e:H1Temp}. Then~\eqref{e:assumed} automatically holds and $\sem{H}_1\geq -C\hbar$  as required. Next, observe that Theorem~\ref{t:simpleCompare} is trivial when $\referee{\tilde{\delta}}_k(R(\hbar);\hbar)T(\hbar)\gg 1$ or $|T(\hbar)|\ll \hbar$. Therefore, we will assume that $\referee{\tilde{\delta}}_k\leq C \hbar^{-1}$ and $T(\hbar)\geq c\hbar$. In particular, for $k$ large enough depending on $s$, this implies that \referee{the assumptions in Definition~\ref{d:expandBox} hold} with \referee{$R(\hbar)$ replaced by $\frac{1}{2}R(\hbar)$\st{$R=\frac{1}{2}R(\hbar)$}}. It only remains to check that~\eqref{e:lip} holds for $T\leq T_{\max}(\hbar)$, $\omega\in [a+\frac{2\e}{3},b-\frac{2\e}{3}]$ and $\lambda\in [-\frac{\e}{3},\frac{\e}{3}]$. 

To see this, observe that $\sem{E}(\sem{H}_1)(x,x,\omega)$ is monotone increasing in $\omega$. Therefore, for $\lambda\geq 0$, $T(\hbar)<T_{\max}(\hbar,x,a,b,\referee{Z})$, $ \omega\in [a+\frac{2\e}{3},b-\frac{2\e}{3}]$ and $\lambda\in [-\frac{\e}{3},\frac{\e}{3}]$, we have
\begin{align*}
0&\leq \sem E(\sem{H}_1)(x,x,\omega+\lambda)-\sem E(\sem{H}_1)(x,x,\omega)\\
&\leq \sum_{j=1}^{\lceil T\lambda\hbar^{-1}\rceil}\sem E(\sem{H}_1)(x,x,\omega+j\hbar/T(\hbar))-\sem E(\sem{H}_1)(x,x,\omega+(j-1)\hbar/T(\hbar))\\
&\leq \sum_{j=1}^{\lceil T\lambda\hbar^{-1}\rceil}C\hbar^{1-d}/T(\hbar)\leq C \lceil T\lambda\hbar^{-1}\rceil\hbar^{1-d}/T(\hbar)\leq C (1+\lambda  T\hbar^{-1})\hbar^{1-d}/T(\hbar),
\end{align*}
as claimed. A similar argument now applies for $\lambda\leq 0$ and this concludes the proof of~\eqref{e:lip} and hence also of Theorem~\ref{t:simpleCompare}.

We now outline the strategy for proving Proposition~\ref{p:distantPerturb}. The proof will use the `wave' approach to spectral asymptotics. That is, we will study certain smoothed versions of the spectral projector.  Using the Fourier transform, one can write these smoothed spectral projectors in terms of the half-wave propagator for $ \sem{H}_{\referee{j}}$. 
In order to take advantage of the finite speed of propagation for $\cos(t\sqrt{\sem{H}_0}/\hbar)$, we will, at the cost of an acceptable error, rewrite these smoothed spectral projectors in terms of the cosine propagator (\S\ref{s:cosine}) and use the finite speed of propagation property for $\cos(t\sqrt{\sem{H}_0}/\hbar)$ to show that $\cos (t\sqrt{\sem{H}_0}/\hbar)$ and $\cos(t\sqrt{\sem{H}_1}/\hbar)$ are close in an appropriate sense (\S\ref{s:finiteSpeed}). This will show that the smoothed projectors for $\sem{H}_0$ and $\sem{H}_1$ are close. Once this is done, we use standard Tauberian lemmas with minor modifications (\S\ref{s:taub}) to show that the unsmoothed spectral projectors are close to their smoothed versions. The proof of Proposition~\ref{p:distantPerturb} is implemented in \S\ref{s:distantPerturb}.

Before proceeding with the proof, we show that we can reduce \referee{the problem} to the case $\sem{H}_j\geq c>0$. First, observe that, for $\iota$ small enough,
\begin{equation}
\label{e:projectorsEqual}
\sem{E}\big(\sem{H}_j+\referee{\iota}\big)(\omega)=\sem{E}(\sem{H}_j)\big(\sqrt{\omega^2-\referee{\iota}}\big).
\end{equation}
Therefore, taking $0<\referee{\iota}\ll \e$,~\eqref{e:lip} and~\eqref{e:lip2} imply the corresponding estimates for $\sem{H}_1+\referee{\iota}$ when $\lambda\in [-\e,\e]$ and $\sqrt{\omega^2-\referee{\iota}}\in [a-\e,b+\e]$. Fix such an $\referee{\iota}$. Then, since $\sem{H}_j\geq -C\hbar$, we see that $\hbar<\hbar_0(\referee{\iota})$ implies $\sem{H}_j\geq \frac{\referee{\iota}}{2}$. Using~\eqref{e:projectorsEqual} again, we see that~\eqref{e:bigBallOk} and~\eqref{e:bigBallOk2} with $\sem{H}_j$ replaced by $\sem{H}_j+\referee{\iota}$ and $\sqrt{\omega^2-\referee{\iota}}\in[a,b]$ imply the estimates~\eqref{e:bigBallOk} and~\eqref{e:bigBallOk2}. We thus are allowed to assume from now on that $\sem{H}_j\geq c>0$. 
 
 \subsection{Basic properties of the wave group}
 \label{s:finiteSpeed}
To begin with, we need a lemma comparing the solution of two wave problems: one with a local potential $\mathbf{Q}$, and the other with an, in principle pseudodifferential, potential that agrees with $\mathbf{Q}$ on a large ball. For this, we recall the standard finite speed of propagation lemma \referee{and prove it for the sake of completeness.}
\begin{lemma}
 \label{l:finiteSpeed}
 Let $\sem{H}_0$ \referee{satisfy~\eqref{e:H0a}}, let $R_0>0$, and suppose $u_0\in H^1(\referee{\mathscr{M}})$,  $u_1\in L^2(\referee{\mathscr{M}})$ with $\supp u_i\cap B_{\referee{\mathscr{M}}}(x_0,R_0)=\emptyset$. Let $u(t,x)$ be the unique solution of 
 $$
 (\hbar^2\partial_t^2+\sem{H}_0)u=0,\qquad u|_{t=0}=u_0,\qquad u_t|_{t=0}=u_1.
 $$
 Then, 
 $
 u(t,x)\equiv 0\text{ on } B_{\referee{\mathscr{M}}}(x_0,R_0-|t|).
 $
 \referee{In particular, 
 $$
 \cos(t\sqrt{\sem{H}_0})u_0=\frac{\sin(t\sqrt{\sem{H}_0})}{\sqrt{\sem{H}_0}}u_1=0,\qquad \text{ on }B_{\referee{\mathscr{M}}}(x_0,R_0-|t|).
 $$}
 \end{lemma}
 \begin{proof}
 Let 
 $
 \referee{\mathcal{K}}_t:= B_{\referee{\mathscr{M}}}(x_0,R_0-|t|)
 $
 and define 
 $$
 E(t):= \frac{1}{2}\Big(\int_{\referee{\mathcal{K}}_t}|\hbar \partial_t u(t,x)|^2+|\hbar d u(t,x)|_g^2+|u(t,x)|^2\referee{d\vol_g(x)}\Big)\geq 0,
 $$
 \referee{where the $1$-form $du$ is the exterior derivative of $u$.}
 Then, for $|t|<R_0$,
 $$
 \begin{aligned}
 &\hbar\partial_t E(t)\\
 &= \Re \Big(\int_{\referee{\mathcal{K}}_t} \big(\hbar^2\partial_t^2u\overline{\hbar\partial_tu}+ u\overline{\hbar \partial_t u}+ \langle \hbar du,\hbar^2d\partial_t  u\rangle_g \big)\referee{d\vol_g(x)}\Big)\\
 &\qquad-\frac{\hbar}{2}\int_{\partial \referee{\mathcal{K}}_t}\big(|\hbar\partial_tu|^2+|\hbar d u|_g^2+|u|^2\big)dS(x) \\
  &=\Re \Big(-\int_{\referee{\mathcal{K}}_t} \big(( \mathbf{Q}_0-1) u\overline{\hbar\partial_t u}\big)\referee{d\vol_g(x)}\Big)-\frac{\hbar}{2}\int_{\partial \referee{\mathcal{K}}_t}\big(|\hbar\partial_tu|^2+|\hbar du|_g^2-2\Re(\hbar\partial_\nu u\overline{\hbar\partial_t u})+|u|^2\big)dS(x) \\
  &\leq \Re \Big(-\int_{\referee{\mathcal{K}}_t}\big( ( \mathbf{Q}_0-1) u\overline{\hbar\partial_t u}\big)\referee{d\vol_g(x)}\Big)\leq CE(t).
 \end{aligned}
 $$
 \referee{Here, we have used Green's formula on the third term in the first line to obtain the second equality.}
 Therefore, since $E(0)=0$, Gr\"onwall's inequality implies $E(t)\equiv 0$ for $|t|<R_0$ and, in particular, $u\equiv 0$ on $\referee{\mathcal{K}}_t$.
 \end{proof}

With Lemma~\ref{l:finiteSpeed} in place, we can now compare the wave problem for $\sem{H}_0$ with that for $\sem{H}_1$.
 \begin{lemma}
\label{l:waveCompare}
Suppose that \referee{$H_0$ satisfies~\eqref{e:H0a} and~\eqref{e:assumed0} and that $\sem{H}_1$ is a family of expanding box operators for $\sem{H}_0$.} Then for $u\in H_{\hbar}^1(B_{\referee{\mathscr{M}}}(x_0,R_0))$ and $|t|\leq R(\hbar)-R_0-1$ we have
$$
\Big\|\big[\cos\big(t\sqrt{ \sem{H}_0}/\hbar\big)-\cos\big(t\sqrt{\sem{H}_1}\referee{/}\hbar\big)\big]u\Big\|_{\referee{\mc{D}_{\hbar}}^{1/2}}\leq C\referee{\tilde{\delta}}|t|\|u\|_{H_{\hbar}^1}.
$$
\end{lemma}
\referee{
\begin{remark}
Our proof of Lemma~\ref{l:waveCompare} uses crucially finite speed of propagation for $\sem{H}_0$. Since finite speed of propagation only holds for differential operators $\sem{H}_0$, we are unable to prove Lemma~\ref{l:waveCompare} for e.g. pseudodifferential perturbations of the Laplacian.
\end{remark}
}
\begin{proof}
Let $w_j=\cos\big(t\sqrt{  \sem{H}_j}/\hbar\big)u$, $\referee{j=0,1}$. Then, since there is $C>0$, depending only on $C_{\referee{1}}, C_{\cutMan,\referee{1}}$, \referee{(where $C_{\cutMan,1}$ and $C_1$ are defined in~\eqref{e:assumed} and~\eqref{e:assumed0})} such that
$$
\tfrac{1}{C}\|v\|_{H_{\hbar}^1}\leq (\|\sqrt{\sem{H}_0}v\|_{L^2}+\|v\|_{L^2})\leq C\|v\|_{H_{\hbar}^1},
$$
we have 
$$
\|w_0(t)\|_{H_{\hbar}^1}\leq C\|u\|_{H_{\hbar}^1},\qquad \|w_1(t)\|_{\referee{\mc{D}_{\hbar}}^{1/2}}\leq C\|u\|_{\referee{\mc{D}_{\hbar}}^{1/2}}.
$$

In order to compare $w_j$, $j=0,1$, we claim $w_0\in\mc{H}$. Indeed, by Lemma~\ref{l:finiteSpeed} $\cos\big(t\sqrt{\sem{H}_0}/\hbar\big)$ has unit speed of propagation and, in particular, for $\auxCut\in C^\infty(\re;[0,1])$, with $\auxCut \equiv 1$ on $(-\infty, R_0)$, $\supp \auxCut\subset (-\infty,R_0+\frac{1}{2})$, we have
$$
\auxCut(\referee{dist}(x_0,x)-|t|)w_0=w_0.
$$
Thus, $w_0\in \mc{H}$, since $|t|\leq R-R_0-1$ implies $\supp \auxCut(\referee{dist}(x_0,\cdot)-|t|)\subset B_{\referee{\mathscr{M}}}(x_0,R-\frac{1}{2})$.

We may then observe that
$$
(\hbar^2\partial_t^2 +\sem{H}_1)(w_1-w_0)=(\sem{H}_0-\sem{H}_1)\auxCut(\referee{dist}(x_0,x)-|t|)w_0,\qquad (w_1-w_0)|_{t=0}=\partial_t(w_1-w_0)|_{t=0}=0.
$$
Using again that $|t|\leq R-R_0-1$ implies $\supp \auxCut(\referee{dist}(x_0,\cdot)-|t|)\subset B_{\referee{\mathscr{M}}}(x_0,R-\frac{1}{2})$ and letting $\cutMan \in  \USB(\referee{\mathscr{M}})$ with $\cutMan \equiv 1$ on $B_{\referee{\mathscr{M}}}(x_0,R-\frac{1}{2})$ and $\supp \cutMan\subset B_{\referee{\mathscr{M}}}(x_0,R)$, we have, \referee{ by~\eqref{e:assumed}},
\begin{align*}
\|(\sem{H}_0-\sem{H}_1)\auxCut(\referee{dist(x_0,x)}-|t|)w_0\|_{\mc{H}}&=\|(\sem{H}_0-\sem{H}_1)\cutMan\auxCut(\referee{dist}(x_0,x)-|t|)w_0\|_{\mc{H}}\\
&\leq C_{\cutMan,\referee{1}}\hbar  \referee{\tilde{\delta}}\|w_0\|_{H_{\hbar}^1}\leq C_{\cutMan\referee{,1}}\hbar\referee{\tilde{\delta}} \|u\|_{H_{\hbar}^1}.
\end{align*}
\referee{We then have} Duhamel's formula
\begin{gather*}
\begin{pmatrix}w_1-w_0\\ \hbar\partial_t(w_1-w_0)\end{pmatrix}\referee{(t)}=\hbar^{-1}\int_0^t U(t-s)\begin{pmatrix}0\\ (\sem{H}_0-\sem{H}_1)\referee{f(dist(x_0,x)-|s|)}w_0(s)\end{pmatrix}ds,\\
U(t):=\begin{pmatrix}\cos\big(t\sqrt{\sem{H}_1}/\hbar\big)&\frac{\sin\big(t\sqrt{\sem{H}_1}/\hbar\big)}{\sqrt{\sem{H}_1}}\\-\sqrt{\sem{H}_1}\sin\big(t\sqrt{\sem{H}_1}/\hbar\big)&\cos\big(t\sqrt{ \sem{H}_1}/\hbar \big)
\end{pmatrix}.
\end{gather*}
Using that $\sem{H}_1\geq c>0$, we have \referee{
\begin{align*}
\|(w_1-w_0)(t)\|_{\mc{D}_{\hbar}^{1/2}}&\leq \hbar^{-1}\int_0^t\Big\|\frac{\sin((t-s)\sqrt{\mathbf{H}_1}/\hbar)}{\sqrt{\mathbf{H}_1}}(\mathbf{H}_0-\mathbf{H}_1)f(dist(x_0,x)-|s|)w_0(s)\Big\|_{\mc{D}_h^{\frac{1}{2}}}ds\\
&\leq C\hbar^{-1}\int_0^t\Big\|(\mathbf{H}_0-\mathbf{H}_1)f(dist(x_0,x)-|s|)w_0(s)\Big\|_{\mc{H}}ds\\
&\leq C_{\cutMan,1}\referee{\tilde{\delta}}|t| \|u\|_{H_{\hbar}^1}.
\end{align*}}
\end{proof}

\subsection{Tauberian lemmas}
\label{s:taub}

Before proceeding to our analysis of the local density of states, we recall two Tauberian lemmas which will allows us to compare smoothed local densities of states to their unsmoothed counterparts.

The first Lemma shows that if the local density of states $\sem E(\sem{H}_j)(x,y,\omega)$ is Lipschitz at sufficiently small scales, then it is close to its smoothed version.
\begin{lemma}[Lemma 5.3 \cite{CaGa:20}]
\label{l:lip}
Let $\{K_j\}_{j=0}^\infty\subset \mathbb{R}_+$. Then there exists $C>0$ and for all $N\in \mathbb{R}$, $N>0$, there is $C\sub{N}>0$ such that the following holds. Let $\{\cutEnergyFourier_\hbar \}_{\hbar >0}\subset \mathscr{S}(\mathbb{R})$ be a family of functions and $\referee{\sigma_{\hbar}=\sigma(\hbar):(0,1]\to \mathbb{R}_+}$ such that for all $j\geq 1$, $\hbar >0$, and $s\in \mathbb{R}$ we have
$$
|\cutEnergyFourier_\hbar (s)|\leq \sigma_\hbar  K_j\langle \sigma_\hbar  s\rangle^{-j}. 
$$
Let \referee{$L_\hbar =L(\hbar):(0,1]\to \mathbb{R}_+$}, \referee{$B_\hbar =B(\hbar):(0,1]\to \mathbb{R}_+$}, $\{\referee{\tilde{w}_\hbar} :\mathbb{R}\to \mathbb{C}\}_{\hbar >0}$, $I_\hbar \subset [-K_0,K_0]$, $\hbar _0>0$ and $\e_0>0$ be such that 
$$
|\referee{\tilde{w}_\hbar} (t-s)-\referee{\tilde{w}_\hbar} (t)|\leq L_\hbar \langle \sigma_\hbar  s\rangle,\,\,t\in I_\hbar ,\,|s|\leq \e_0,\qquad |\referee{\tilde{w}_\hbar} (s)|\leq B_\hbar \langle s\rangle^{N_0}\,\text{ for all }s\in \mathbb{R}.
$$
Then for all $0<\hbar <\hbar _0$ and $t\in I_\hbar $ we have
$$
\Big|(\cutEnergyFourier_\hbar *\referee{\tilde{w}_\hbar} )(t)-\referee{\tilde{w}_\hbar} (t)\int \cutEnergyFourier_\hbar (s)ds\Big|\leq CL_\hbar +C\sub{N}B_\hbar \sigma_\hbar ^{-N}\e_0^{-N}.
$$
\end{lemma}
\begin{proof}
For all {$0<\hbar<\hbar_0$} and $t\in I_{\hbar}$ we have
\begin{align*}
&\Big|(\cutEnergyFourier_{\hbar} *w_{\hbar})(t)-w_{\hbar}(t)\int_{\re} \cutEnergyFourier_{\hbar}(s)ds\Big|
=\Big|\int_{\re} \cutEnergyFourier_{\hbar}(s)\big(w_{\hbar}(t-s)-w_{\hbar}(t)\big)ds\Big|\\
& \leq L_h\int_{|s|\leq \e_0} |\cutEnergyFourier_{\hbar}(s)|\langle \sigma_{\hbar} s\rangle ds
+B_{\hbar}\int_{|s|\geq \e_0}|\cutEnergyFourier_{\hbar}(s)|\Big( \langle t-s\rangle^{N_0}+\langle t\rangle^{N_0}\Big)ds\\
&\leq L_{\hbar}\int_{|s|\leq \e_0}\!\!\! \sigma_h K_{3} \langle \sigma_{\hbar} s\rangle^{-2} ds
+B_{\hbar}\!\!\int_{|s|\geq \e_0} \!\!\!K\sub{N_0+2+N}\sigma_{\hbar} \langle \sigma_{\hbar} s\rangle^{-(N_0+2+N)}\!\Big( \langle t-s\rangle^{N_0}+\!\langle t\rangle^{N_0}\Big)ds.
\end{align*}
The existence of $C$ and $C\sub{N}$ follows from the integrability of each term {and the boundedness of $I_{\hbar}$}.
\end{proof}

The next lemma is similar to~\cite[Lemma 17.5.6]{Ho:07} and will be used to show that $\sem E(\sem{H}_0)(x,x,\omega)$ inherits the Lipschitz nature of $\sem E(\sem{H}_1)(x,x,\omega)$. 
\begin{lemma}
\label{l:monotoneTaub}
Let $\phi\in \mathscr{S}(\mathbb{R};[0,\infty))$ with $\phi>0$ on $[-1,1]$ and \referee{for $\gamma>0$} put $\phi_{\referee{\gamma}}(t):={\referee{\gamma}}^{-1}\phi({\referee{\gamma}}^{-1}t)$. Then there is $C>0$ and for all $N>0$ there is $C\sub{N}>0$ such that the following holds. Suppose that $\{\mu_\hbar \}_{\hbar >0}$ is a family of monotone increasing functions, $\{\alpha_\hbar \}_{\hbar >0}$ is a family of functions of locally bounded variation and that there are $\e>0$, $\referee{\iota}>0$, \referee{$\gamma_\hbar=\gamma(\hbar):(0,1]\to \mathbb{R}_+$},  \referee{$M_\hbar=M(\hbar):(0,1]\to \mathbb{R}_+$}, $N_0>0$,  \referee{$B_\hbar=B(\hbar):(0,1]\to \mathbb{R}_+$}, $C>0$, and $ \hbar _0>0$ such that for $0<\hbar <\hbar _0$,  $\referee{\gamma_{\hbar}}<\hbar^\referee{\iota}$, and we have
\begin{gather*}
\int_{\omega-\referee{\gamma_{\hbar}} }^{\omega+\referee{\gamma_{\hbar}} }|d\alpha_\hbar |\leq \referee{\gamma_{\hbar}} M_\hbar ,\qquad
|(d\mu_\hbar -d\alpha_\hbar )*\phi_{\referee{\gamma_{\hbar}} }(\omega)|\leq B_\hbar ,\qquad \referee{\omega\in[a-\e,b+\e]},\\
|\mu_\hbar (\omega)|+|\alpha_\hbar (\omega)|\leq \hbar ^{-N_0} \langle \omega\rangle^{N_0},\qquad \referee{\omega\in \mathbb{R}}.
\end{gather*}
\referee{(For a function $f$ of bounded variation, we denote by $df$ the derivative of $f$ considered as a measure and $|df|$ its total variation.)}
Then for $|s|\leq \e/2$ and $\omega\in[a,b]$ we have
$$
|\mu_\hbar (\omega)-\mu_\hbar (\omega-s)|\leq C\referee{\gamma_{\hbar}} (M_\hbar +B_\hbar +C\sub{N}\hbar ^N)\langle \referee{\gamma_{\hbar}} ^{-1}s\rangle.
$$
\end{lemma}
\begin{proof}
Let $\omega_0\in [a,b]$. Since $d\mu_\hbar  \geq 0$, for $\omega\in[a-\e,b+\e]$,
\begin{align*}
|\mu_\hbar (\omega)-\mu_\hbar (\omega-\referee{\gamma_{\hbar}} )|= \int_{\omega-a_{\hbar}}^\omega d\mu_h(s)&\leq C_\phi \referee{\gamma_{\hbar}} \Big( \int \phi_{\referee{\gamma_{\hbar}} }(\omega-s)d\big(\mu_\hbar (s)-\mu_\hbar (\omega_0)\big)\Big)\\
&=C\referee{\gamma_{\hbar}} ( \phi_{\referee{\gamma_{\hbar}} }*d(\mu_\hbar -\mu_\hbar (\omega_0)))(\omega).
\end{align*}
First, we estimate
\begin{align*}
|( \phi_{\referee{\gamma_{\hbar}} }*d\mu_\hbar )(\omega)|&\leq |\phi_{\referee{\gamma_{\hbar}} }*d(\alpha_\hbar -\alpha_\hbar (\omega_0))(t)|+|(\phi_{\referee{\gamma_{\hbar}} }*d(\mu_\hbar -\mu_\hbar (\omega_0)-\alpha_\hbar +\alpha_\hbar (\omega_0)))(\omega)|=:I +II.
\end{align*}
Now, 
\begin{align*}
I&\leq  \referee{\gamma_{\hbar}} ^{-1}\int \phi(\referee{\gamma_{\hbar}} ^{-1}(\omega-s))|d(\alpha_\hbar (s)-\alpha_\hbar (\omega_0))|\leq \referee{\gamma_{\hbar}} ^{-1}\int_{|\omega-s|\leq \referee{\gamma_{\hbar}} \hbar ^{-\referee{\iota}/2}} \langle \referee{\gamma_{\hbar}} ^{-1}(\omega -s)\rangle^{-N}|d\alpha_\hbar (s)|+O(\hbar^\infty)\\
&\leq \sum_{|k|\leq \hbar ^{-\referee{\iota}/2}}\langle k\rangle^{-N} M_\hbar + O(\hbar ^\infty)\leq CM_\hbar +O(\hbar ^\infty).
\end{align*}
Next, 
$$
II= \int_{\omega_0}^\omega |(d\mu_\hbar -d\alpha_\hbar )*\phi_{\referee{\gamma_{\hbar}}}(s)|ds\leq B_\hbar |\omega-\omega_0|\leq \referee{(b-a+\e)}B_\hbar .
$$
Therefore, 
\begin{equation}
\label{e:step}
|\mu_\hbar(\omega)-\mu_\hbar(\omega-\referee{\gamma_{\hbar}})|\leq C \referee{\gamma_{\hbar}}(M_\hbar+B_\hbar +O(\hbar^\infty)).
\end{equation}
The claim now follows from adding terms like~\eqref{e:step}.
\end{proof}

\subsection{Local densities of states and the cosine propagator}
\label{s:cosine}

We need two more preliminary lemmas before analyzing the local density of states. These lemmas, modulo \referee{controllable} errors, rewrite the spectral projection operator and its derivatives in terms of the cosine propagator. This crucial step allows us to use Lemma~\ref{l:waveCompare} to show that the smoothed densities of states for $\sem{H}_0$ and $\sem{H}_1$ are close.
For $\cutEnergyFourier\in\mathscr{S}(\mathbb{R})$ and $T>0$, we \referee{recall that}
$$
\cutEnergyFourier_{_{\!\referee{T/\hbar}}}(s)=\hbar^{-1}T\cutEnergyFourier(\hbar^{-1}Ts).
$$
\begin{lemma}
\label{l:compareTo1}
Let $\cutEnergyFourier \in \mathscr{S}(\mathbb{R})$ with $\supp \hat{\cutEnergyFourier}\subset (-2,2)$, $\e>0$, and  $T=T(\hbar)\geq \e$. Then for $\omega\in [a-2\e,b-2\e]$, $j=0,1$, and all $N\geq 0$ we have
\begin{equation}
\label{e:toFiniteSpeed}
\partial_\omega \big(\cutEnergyFourier_{_{\!\referee{T/\hbar}}}*\sem E(\sem{H}_j)\big)(\omega)=\frac{1}{\pi \hbar}\int \hat{\cutEnergyFourier}(T^{-1}\tau) e^{it\tau\omega/\hbar}\cos \big(\tau \sqrt{ \sem{H}_j}/\hbar\big) d\tau +O(\hbar^N)_{\referee{\mc{D}}^{-N}\to \referee{\mc{D}}^N},
\end{equation}
\referee{where $\mc{D}^N$ denotes the domain of the corresponding operator $\mathbf{H}_j^N$ and $\mc{D}^{-N}$ that of $\mathbf{H}_j^{-N}$}.
\end{lemma}
\begin{proof}
First, recall that 
\begin{align}
\partial_\omega\big( \cutEnergyFourier_{_{\!\referee{T/\hbar}}}*\sem E(\sem{H}_j)\big)(\omega)&=\cutEnergyFourier_{_{\!\referee{T/\hbar}}}(\omega-\sqrt{\sem{H}_j})=\frac{1}{2\pi \hbar}\int \hat{\cutEnergyFourier}(T^{-1}\tau)e^{i\frac{\tau}{\hbar} (\omega-\sqrt{\sem{H}_j})}d\tau\label{e:flowers}\\
&=\frac{1}{2\pi \hbar}\int \hat{\cutEnergyFourier}(T^{-1}\tau)e^{i\frac{\tau}{\hbar} \omega}\big(2 \cos\big(\tau \sqrt{\sem{H}_j}/\hbar\big)-e^{i\tau\sqrt{\sem{H}_j}/\hbar}\big)d\tau\nonumber\\
&=\frac{1}{\pi \hbar}\int \hat{\cutEnergyFourier}(T^{-1}\tau)e^{i\frac{\tau}{\hbar} \omega}\cos\big(\tau \sqrt{\sem{H}_j}/\hbar\big)d\tau -\cutEnergyFourier_{_{\!\referee{T/\hbar}}}(\omega+\sqrt{\sem{H}_j}).\nonumber
\end{align}
Next, since $\sem{H}_j\geq 0$, we have
\begin{align*}
\|\cutEnergyFourier_{_{\!\referee{T/\hbar}}}(\omega+\sqrt{\sem{H}_j})(1+\sem{H}_j)^N\|_{L^2\to L^2}&\leq \sup_{s\geq 0}\cutEnergyFourier_{_{\!\referee{T/\hbar}}}(\omega+s)(1+s^2)^N\\
&\leq \sup_{s\geq 0}C_N \hbar^{-1}T\langle \hbar^{-1}T(\omega+s)\rangle^{-2N-1}(1+s^2)^N\leq C_N \hbar^{N}T^{-N}.
\end{align*}
Therefore, since
$$
\|u\|_{\referee{\mc{D}}^{N}}\leq C_N\|(1+\sem{H}_j)^Nu\|_{L^2},
$$
the estimate~\eqref{e:toFiniteSpeed} follows.
\end{proof}

\begin{lemma}
\label{l:differenceNew}
Let $\cutEnergyFourier \in \mathscr{S}(\mathbb{R})$ with $\hat{\cutEnergyFourier}$ even, $\e>0$, and $T(\hbar)\geq \e$. Then for $\omega\in [a-2\e,b+2\e]$, $j=0,1$, we have
\begin{equation}
\label{e:compareNew}
 \cutEnergyFourier_{_{\!\referee{T/\hbar}}}*\sem E(\sem{H}_j)(\omega)=\frac{1}{\pi }  \int\frac{1}{\tau}\hat{\cutEnergyFourier}(T^{-1}\tau)\sin( h^{-1}\tau \omega) \cos\big(\tau \sqrt{\sem{H}_j}/\hbar\big)d\tau + (\cutEnergyFourier_{_{\!\referee{T/\hbar}}}*\sem E(\sem{H}_j))(-\omega).
\end{equation}
\end{lemma}

\begin{proof}
\referee{Using formula~\eqref{e:flowers} in the second line, we have}
\begin{align*}
 \big(\cutEnergyFourier_{_{\!\referee{T/\hbar}}}*\sem E(\sem{H}_j)\big)(\omega)&= \int_{-\omega}^\omega ( \partial_\omega\cutEnergyFourier_{_{\!\referee{T/\hbar}}}*\sem E(\sem{H}_j))(s)ds+ (\cutEnergyFourier_{_{\!\referee{T/\hbar}}}*\sem E(\sem{H}_j))(-\omega)\\
&=\frac{1}{2\pi \hbar} \int_{-\omega}^\omega \int\hat{\cutEnergyFourier}(T^{-1}\tau)e^{i\frac{\tau}{\hbar} (s-\sqrt{\sem{H}_j})}d\tau ds+ (\cutEnergyFourier_{_{\!\referee{T/\hbar}}}*\sem E(\sem{H}_j))(-\omega)\\
&=\frac{1}{\pi }  \int\frac{1}{\tau}\hat{\cutEnergyFourier}(T^{-1}\tau)\sin( h^{-1}\tau \omega) e^{-i\frac{\tau}{\hbar} \sqrt{\sem{H}_j}}d\tau + (\cutEnergyFourier_{_{\!\referee{T/\hbar}}}*\sem E(\sem{H}_j))(-\omega).
\end{align*}
Note that after changing variables, $\tau\to -\tau$, we have
$$
\frac{1}{2\pi }  \int\frac{1}{\tau}\hat{\cutEnergyFourier}(T^{-1}\tau)\sin( h^{-1}\tau \omega) e^{-i\frac{\tau}{\hbar} \sqrt{\sem{H}_j}}d\tau=\frac{1}{2\pi }  \int\frac{1}{\tau}\hat{\cutEnergyFourier}(T^{-1}\tau)\sin( h^{-1}\tau \omega) e^{i\frac{\tau}{\hbar} \sqrt{\sem{H}_j}}d\tau.
$$
Therefore, 
$$
 \cutEnergyFourier_{_{\!\referee{T/\hbar}}}*\sem E(\sem{H}_j)(\omega)=\frac{1}{\pi }  \int\frac{1}{\tau}\hat{\cutEnergyFourier}(T^{-1}\tau)\sin( h^{-1}\tau \omega) \cos\big(\tau \sqrt{\sem{H}_j}/\hbar\big)d\tau + (\cutEnergyFourier_{_{\!\referee{T/\hbar}}}*\sem E(\sem{H}_j))(-\omega).
$$
\end{proof}

\referee{We estimate the last term in~\eqref{e:compareNew} in the next lemma.
\begin{lemma}
\label{l:difference}
Let $\cutEnergyFourier \in \mathscr{S}(\mathbb{R})$ with $\hat{\cutEnergyFourier}$ even, $\e>0$, and $T(\hbar)\geq \e$. Then for $\omega\in [a-2\e,b+2\e]$, $j=0,1$, and all $N\geq 0$,
\begin{equation}
\label{e:compareTo1}
\cutEnergyFourier_{_{\!\referee{T/\hbar}}}*\sem E(\sem{H}_j)(\omega)=\frac{1}{\pi }  \int\frac{1}{\tau}\hat{\cutEnergyFourier}(T^{-1}\tau)\sin( h^{-1}\tau \omega) \cos\big(\tau \sqrt{\sem{H}_j}/\hbar\big)d\tau+O(\hbar^\infty)_{\referee{\mc{D}}^{-N}\to \referee{\mc{D}}^N}. 
\end{equation}
\end{lemma}

\begin{proof}
Using~\eqref{e:compareNew}, it remains to check that 
$$
 (\cutEnergyFourier_{_{\!\referee{T/\hbar}}}*\sem E(\sem{H}_j))(-\omega)=O(\hbar^\infty)_{\referee{\mc{D}}^{-N}\to \referee{\mc{D}}^N}.
$$

Since $\sem{H}_j\geq 0$, $\sem E(\sem{H}_j)(s)\referee{=1_{(-\infty,s]}(\sqrt{\sem{H}_j})}\equiv 0$ for $s<0$. Thus, for all $N,L\geq 0$ {there is $C\sub{L,N}>0$ such that}
\begin{align*}
\|\cutEnergyFourier_{_{\!\referee{T/\hbar}}}*\sem E(\sem{H}_j)(-\omega)\|_{\referee{\mc{D}}^{-N}\to \referee{\mc{D}}^N}&\leq \int_{\re}\tfrac{T}{\hbar}\cutEnergyFourier\big(\tfrac{T}{\hbar}s\big)\|\sem E(\sem{H}_j)(-\omega-s)\|_{\referee{\mc{D}}^{-N}\to \referee{\mc{D}}^N}ds\\
&\leq C\sub{L,N} \int_{s\leq -\frac{\omega}{2}}\tfrac{T}{h}\big\langle\tfrac{T}{h}s\big\rangle^{-L}\langle s\rangle^{N}.
\end{align*}
The claim follows after choosing $L$ large enough.
\end{proof}}

\subsection{Comparison of the local densities of states}
\label{s:distantPerturb}

This section contains the proof of Proposition~\ref{p:distantPerturb}. We start by showing that, when smoothed at scale $\sim\!\!1$, spectral projectors for $\sem{H}_0$ and $\sem{H}_1$ are close.  In other words, \referee{when $\cutEnergy\in\mathscr{S}$,} $\cutEnergy(\sem{H}_0)$ and $\cutEnergy(\sem{H}_1)$ are close when acting on subsets of $B(0,R(\hbar))$. 
\begin{lemma}
\label{l:basicCutoff}
Let $R_0>0$, $R(\hbar)>R_0+1$, $\referee{\tilde{\delta}}(\hbar)>0$, and suppose that \referee{$H_0$ satisfies~\eqref{e:H0a} and~\eqref{e:assumed0} and that $\sem{H}_1$ is a family of expanding box operators for $\sem{H}_0$.} Let $\cutEnergy \in \mathscr{S}(\mathbb{R})$, $\cutMan,\tilde{\cutMan}\in C_c^\infty(B(0,R_0))$.  Then, for all $N\geq 0$,
\begin{equation}
\label{e:noCoffee}
\tilde{\cutMan}[\cutEnergy(\sem{H}_0)-\cutEnergy(\sem{H}_1)]\cutMan=O(\hbar \referee{\tilde{\delta}}(\hbar))_{\Psi^{-\infty}}+O(\hbar^\infty)_{\Psi^{-\infty}}.
\end{equation}
Moreover,  if $\tilde{\cutMan}\equiv 1$ in a neighbourhood of $\supp\cutMan$, then
\begin{gather}
\label{e:H0Local}
(1-\tilde{\cutMan})\cutEnergy(\sem{H}_0)\cutMan=O(\hbar^\infty)_{\Psi^{-\infty}},\qquad \cutMan\cutEnergy(\sem{H}_0)(1-\tilde{\cutMan})=O(\hbar^\infty)_{\Psi^{-\infty}},\\
\label{e:H1Local}(1-\tilde{\cutMan})\cutEnergy(\sem{H}_1)\cutMan=O(\hbar \referee{\tilde{\delta}}(\hbar)+\hbar^\infty)_{H_\hbar^{-N}\to \referee{\mc{D}_{\hbar}}^N},\qquad \cutMan\cutEnergy(\sem{H}_1)(1-\tilde{\cutMan})=O(\hbar\referee{\tilde{\delta}}(\hbar)+\hbar^\infty)_{\referee{\mc{D}_{\hbar}}^{-N}\to H_{\hbar}^N}.
\end{gather}
\end{lemma}
\begin{proof}
Put $\cutEnergy_1(t):=\cutEnergy(t^2)$. Then $\cutEnergy_1\in \mathscr{S}$ and, since $\sem{H}_i\geq 0$, we have $\cutEnergy_1(\sqrt{\sem{H}_i})=\cutEnergy(\sem{H}_i)$. Next, observe that $\cutEnergy_1$ is even and hence so is $\hat{\cutEnergy}_1$. Therefore,
$$
\begin{aligned}
\cutEnergy(\sem{H}_i)=\cutEnergy_1(\sqrt{\sem{H}_i})&=\frac{1}{2\pi}\int \hat{\cutEnergy}_1(t)e^{it\sqrt{\sem{H}_0}} dt=\frac{1}{2\pi\hbar}\int \hat{\cutEnergy}_1(s/\hbar)\cos(s\sqrt{\sem{H}_i}/\hbar)ds.
\end{aligned}
$$

We first \referee{prove~\eqref{e:H0Local} and~\eqref{e:H1Local}. Thus, we assume that} $\tilde{\cutMan}\equiv 1$ in a neighbourhood of $\supp \cutMan$. Let $r>0$ be chosen so that $dist\big(\supp \cutMan,\supp( 1-\tilde{\cutMan})\big)>r$ and let $\auxCut\in C_c^\infty((-r,r))$ with $\auxCut\equiv 1$ near $0$. Then, using Lemma~\ref{l:finiteSpeed} \referee{to pass from the second to the third line, we have}
\begin{align*}
&(1-\tilde{\cutMan})\cutEnergy(\sem{H}_0)\cutMan\\
&=\frac{1}{2\pi\hbar}\int \hat{\cutEnergy}_1(s/\hbar)(1-\tilde{\cutMan})\cos(s\sqrt{\sem{H}_0}/\hbar)\cutMan ds\\
&=\frac{1}{2\pi\hbar}\int \hat{\cutEnergy}_1(s/\hbar)(1-\tilde{\cutMan})(1-\auxCut(s))\cos(s\sqrt{\sem{H}_0}/\hbar)\cutMan ds=O(\hbar^\infty)_{L^2\to L^2}.
\end{align*}
Since 
\begin{equation}
\label{e:yellowPaint}
(\cdot+i)^k\cutEnergy(\cdot)\in \mathscr{S}\text{ for any $k$},
\end{equation}
this implies
$$
(1-\tilde{\cutMan})\cutEnergy(\sem{H}_0)\cutMan=O(\hbar^\infty)_{\Psi^{-\infty}},
$$
which, taking adjoints, implies~\eqref{e:H0Local}.

To prove~\eqref{e:H1Local}, we again write
\begin{align*}
&(1-\tilde{\cutMan})\cutEnergy(\sem{H}_1)\cutMan=\frac{1}{2\pi\hbar}\int \hat{\cutEnergy}_1(s/\hbar)(1-\tilde{\cutMan})\cos(s\sqrt{\sem{H}_1}/\hbar)\cutMan ds\\
&=\frac{1}{2\pi\hbar}\int \hat{\cutEnergy}_1(s/\hbar)(1-\tilde{\cutMan})\auxCut(s)\cos(s\sqrt{\sem{H}_1}/\hbar)\cutMan ds\\
&\qquad+\frac{1}{2\pi\hbar}\int \hat{\cutEnergy}_1(s/\hbar)(1-\tilde{\cutMan})(1-\auxCut(s))\cos(s\sqrt{\sem{H}_1}/\hbar)\cutMan ds\\
&=\frac{1}{2\pi\hbar}\int \hat{\cutEnergy}_1(s/\hbar)(1-\tilde{\cutMan})(1-\auxCut(s))\cos(s\sqrt{\sem{H}_1}/\hbar)\cutMan ds+O(\hbar \referee{\tilde{\delta}})_{H_\hbar^1\to \mc{D}^{1/2}_{1}}\\
&=(\hbar \referee{\tilde{\delta}}(\hbar)+\hbar^\infty)_{H_\hbar^1\to \mc{H}},
\end{align*}
where in the \referee{fourth} line we use Lemma~\ref{l:finiteSpeed} and Lemma~\ref{l:waveCompare}. Using~\eqref{e:yellowPaint} again, this implies~\eqref{e:H1Local}.

Finally, \referee{we prove~\eqref{e:noCoffee}, no longer assuming that $\tilde{\cutMan}\equiv 1$ in a neighbourhood of $\supp \cutMan$.}  Write
$$
\begin{aligned}
&\tilde{\cutMan}(\cutEnergy(\sem{H}_0)-\cutEnergy(\sem{H}_1))\cutMan\\
&=\frac{1}{2\pi\hbar}\Big(\int_{|s|\leq R-R_0-1} \hat{\cutEnergy}_1(s/\hbar)\tilde{\cutMan}(\cos(s\sqrt{\sem{H}_0}/\hbar)-\cos(s\sqrt{\sem{H}_1}/\hbar))\cutMan ds \\
&\qquad +\int_{|s|> R-R_0-1} \hat{\cutEnergy}_1(s/\hbar)\tilde{\cutMan}(\cos(s\sqrt{\sem{H}_0}/\hbar)-\cos(s\sqrt{\sem{H}_1}/\hbar))\cutMan ds\Big)\\
&= \frac{1}{2\pi\hbar}\Big(\int_{|s|\leq R-R_0-1} C_N\langle s/\hbar\rangle^{-N}O(\referee{\tilde{\delta}} s)_{H_{\hbar}^1\to H_\hbar^1 } ds +\int_{|s|> R-R_0-1} C_N\langle s/\hbar\rangle^{-N}O(1)_{L^2\to L^2}ds\Big)\\
&=O(\hbar \referee{\tilde{\delta}}(\hbar))_{H_{\hbar}^1\to H_{\hbar}^1}+O(\hbar^\infty)_{L^2\to L^2}.
\end{aligned}
$$

Next, using~\eqref{e:yellowPaint}, we have
\begin{equation}
\label{e:aNumber}
\tilde{\cutMan}[(\sem{H}_0+i)^k\cutEnergy(\sem{H}_0)-(\sem{H}_1+i)^k\cutEnergy(\sem{H}_1)]\cutMan = O(\hbar \referee{\tilde{\delta}}(\hbar)+\hbar^\infty)_{H_\hbar^1\to \referee{H_{\hbar}^1}}.
\end{equation}

Let $\cutMan_1,\cutMan_2\in C_c^\infty(B_{\referee{\mathscr{M}}}(x_0,R_0))$ with $\cutMan_1\equiv 1$ on $\supp \tilde{\cutMan}\cup\supp \cutMan$ and $\cutMan_2\equiv 1$ on $\supp\cutMan_1$ .
\referee{Next, observe that 
\begin{align*}
&\cutMan_1(\sem{H}_0+i)^k\cutMan_2[\cutEnergy(\sem{H}_0)-\cutEnergy(\sem{H}_1)]\cutMan \\
&=\cutMan_1(\sem{H}_0+i)^k\cutMan_2\cutEnergy(\sem{H}_0)\cutMan-\cutMan_1(\sem{H}_1+i)^k\cutMan_2\cutEnergy(\sem{H}_1)\cutMan +O(\hbar \referee{\tilde{\delta}}(\hbar)+\hbar^\infty)_{\Psi^{-\infty}}
\end{align*}
by~\eqref{e:assumed} together with $\cutEnergy(\sem{H}_0)=O(1)_{\Psi^{-\infty}}$, $\cutEnergy(\sem{H}_1)=O(1)_{\referee{\mc{D}_{\hbar}}^{-\infty}\to \referee{\mc{D}_{\hbar}}^\infty}$. Next, using again $\cutEnergy(\sem{H}_0)=O(1)_{\Psi^{-\infty}}$, $\cutEnergy(\sem{H}_1)=O(1)_{\referee{\mc{D}_{\hbar}}^{-\infty}\to \referee{\mc{D}_{\hbar}}^\infty}$, together with~\eqref{e:H0Local},~\eqref{e:H1Local}, we have
\begin{align*}
&\cutMan_1(\sem{H}_0+i)^k\cutMan_2\cutEnergy(\sem{H}_0)\cutMan-\cutMan_1(\sem{H}_1+i)^k\cutMan_2\cutEnergy(\sem{H}_1)\cutMan +O(\hbar \referee{\tilde{\delta}}(\hbar)+\hbar^\infty)_{\Psi^{-\infty}}\\
&=\cutMan_1[(\sem{H}_0+i)^k\cutEnergy(\sem{H}_0)-(\sem{H}_1+i)^k\cutEnergy(\sem{H}_1)]\cutMan +O(\hbar \referee{\tilde{\delta}}(\hbar)+\hbar^\infty)_{\Psi^{-\infty}}.
\end{align*}
Finally, using~\eqref{e:aNumber}, we obtain
\begin{align*}
\cutMan_1[(\sem{H}_0+i)^k\cutEnergy(\sem{H}_0)-(\sem{H}_1+i)^k\cutEnergy(\sem{H}_1)]\cutMan +O(\hbar \referee{\tilde{\delta}}(\hbar)+\hbar^\infty)_{\Psi^{-\infty}}=O(\hbar\referee{\tilde{\delta}}(\hbar)+\hbar^\infty)_{H_{\hbar}^{\referee{1}}\to \referee{H_\hbar^1}}.
\end{align*}
In particular, 
$$
\cutMan_1(\sem{H}_0+i)^k\cutMan_2[\cutEnergy(\sem{H}_0)-\cutEnergy(\sem{H}_1)]\cutMan =O(\hbar\referee{\tilde{\delta}}(\hbar)+\hbar^\infty)_{H_{\hbar}^{\referee{1}}\to \referee{H_\hbar^1}}.
$$
}
%
Therefore, by local elliptic regularity, 
$$
\tilde{\cutMan}(\cutEnergy(\sem{H}_0)-\cutEnergy(\sem{H}_1))\cutMan=O(\hbar\referee{\tilde{\delta}}(\hbar)+\hbar^\infty)_{H_{\hbar}^1\to H_{\hbar}^{\infty}}.
$$
Making a similar argument for $\cutMan[\cutEnergy(\sem{H}_0)-\cutEnergy(\sem{H}_1)]\cutMan_1(\sem{H}_0+i)^k\cutMan_2 $ then completes the proof of the lemma.
\end{proof}

The next lemma shows that the spectral projectors for $\sem{H}_1$ and $\sem{H}_0$ smoothed at scale $\hbar/T$ are close when acting on compact sets.
\begin{lemma}
\label{l:marcus}
Let $\cutMan \in C_c^\infty(B(0,R_0))$, $\e>0$, $R(\hbar)>0$, $\referee{\tilde{\delta}}(\hbar)>0$,  $\e<T(\hbar)\leq (R(h)-R_0-2)/2$, and suppose \referee{$\sem{H}_0$ satisfies~\eqref{e:H0a} and~\eqref{e:assumed0} and that $\sem{H}_1$ is a family of expanding box operators for $\sem{H}_0$}. Let $\cutEnergyFourier \in \mathscr{S}(\mathbb{R})$ with $\supp \hat{\cutEnergyFourier}\subset (-2,2)$.  Then, for all $N\geq 0$ and $\omega\in[a-2\e,b+2\e]$ we have
\begin{gather}
\label{e:octopus}
\cutMan\partial_\omega \big(\cutEnergyFourier_{_{\!\referee{T/\hbar}}}*\sem E(\sem{H}_0)\big)(\omega)\cutMan= \cutMan \partial_\omega \big(\cutEnergyFourier_{_{\!\referee{T/\hbar}}}*\sem E(\sem{H}_1)\big)(\omega)\cutMan +O(\hbar^{-1}\referee{\tilde{\delta}}(\hbar)T\referee{(\hbar)}^2+\hbar^\infty)_{H_{\hbar}^{-N}\to H_{\hbar}^N}.
\end{gather}
If, in addition, $\referee{\hat\nu}$ is even, then
\begin{gather}
\label{e:squid}
\cutMan\big( \cutEnergyFourier_{_{\!\referee{T/\hbar}}}*\sem E_{\sem H _0}\big)(\omega)\cutMan=\cutMan \big(\cutEnergyFourier_{_{\!\referee{T/\hbar}}}*\sem E(\sem{H}_1)\big)(\omega)\cutMan+O(\referee{\tilde{\delta}}(\hbar)\referee{T(\hbar)}+\hbar^\infty)_{H_{\hbar}^{-N}\to H_{\hbar}^N}.
\end{gather}
\end{lemma}
\begin{proof}
By Lemma~\ref{l:compareTo1}, 
\begin{equation}
\label{e:propIsCosine}
\begin{aligned}
\cutEnergyFourier_{_{\!\referee{T/\hbar}}}(\omega-\sqrt{\sem{H}_0})
& = \frac{1}{\pi \hbar }\int \hat{\cutEnergyFourier}(T^{-1}\tau)e^{it\tau\omega/\hbar} \cos\big(\tau\sqrt{\sem{H}_0}/\hbar\big)d\tau +O(\hbar^\infty)_{H_\hbar^{-N}\to H_\hbar^{N}}.
\end{aligned}
\end{equation}
Next, let $\cutEnergy\in C_c^\infty(\mathbb{R})$ with $\cutEnergy \equiv 1$ on $[\frac{a}{2}-\e,2(b+2\e)]$. Then
$$
\cutEnergyFourier_{_{\!\referee{T/\hbar}}}(\omega-\sqrt{\sem{H}_0})=\cutEnergy(\sem{H}_0)\cutEnergyFourier_{_{\!\referee{T/\hbar}}}(\omega-\sqrt{\sem{H}_0})\cutEnergy(\sem{H}_0) +O(\hbar^\infty)_{\Psi^{-\infty}}
$$
and, by~\eqref{e:H0Local}, for $\tilde{\cutMan}\in C_c^\infty(B(0,R_0))$ with $\tilde{\cutMan} \equiv 1$ on $\supp \cutMan$, we have
$$
(1-\tilde{\cutMan})\cutEnergy(\sem{H}_0)\cutMan=O(\hbar^\infty)_{H_\hbar^{-N}\to H_\hbar^N},\qquad\cutMan\cutEnergy(\sem{H}_0)(1-\tilde{\cutMan})=O(\hbar^\infty)_{H_\hbar^{-N}\to H_\hbar^N}.
$$

 Therefore, by Lemma~\ref{l:waveCompare},
\begin{multline}
\label{e:itsACone}
\cutMan\cutEnergy(\sem{H}_0)\cos\big(\tau\sqrt{\sem{H}_0}/\hbar\big)\cutEnergy(\sem{H}_0)\cutMan=\\\cutMan\cutEnergy(\sem{H}_0)\tilde{\cutMan} \cos\big(\tau\sqrt{\sem{H}_1}/\hbar\big)\tilde{\cutMan}\cutEnergy(\sem{H}_0)\cutMan +O(|\tau|\referee{\tilde{\delta}}(\hbar))_{H_{\hbar}^{-N}\to H_{\hbar}^{N}}+O(\hbar^\infty)_{\Psi^{-\infty}}
\end{multline}
for $\tau \leq R(h)-R_0-1$.  In particular, since $T(\hbar)\leq (R(h)-R_0-2)/2$, and $\supp \hat \cutEnergyFourier\subset (-2,2)$, Lemma~\ref{l:compareTo1} implies that
$$
\begin{aligned}\cutMan\partial_\omega \big(\cutEnergyFourier_{_{\!\referee{T/\hbar}}}*\sem E(\sem{H}_0)\big)(\omega)\cutMan&=\frac{1}{2\pi \hbar}\int \hat{\cutEnergyFourier}(T^{-1}\tau) e^{it\tau\omega/\hbar}\cutMan\cutEnergy(\sem{H}_0)\tilde{\cutMan}\cos\big(\tau\sqrt{\sem{H}_1}/\hbar\big)\tilde{\cutMan}\cutEnergy(\sem{H}_0)\cutMan d\tau \\
&\qquad+O(\hbar^{-1}\referee{\tilde{\delta}}(\hbar)T^2+\hbar^\infty)_{\Psi^{-\infty}}.
\end{aligned}
$$
Finally, using Lemma~\ref{l:basicCutoff} to replace $\cutEnergy(\sem{H}_0)$ by $\cutEnergy(\sem{H}_1)$, $\cutMan\cutEnergy(\sem{H}_1)\tilde{\cutMan}$ by $\cutMan\cutEnergy(\sem{H}_1)$, and $\tilde{\cutMan}\cutEnergy(\sem{H}_1)\cutMan$ by $\cutEnergy(\sem{H}_1)\cutMan$, we obtain
$$
\begin{aligned}
\cutMan\partial_\omega \big(\cutEnergyFourier_{_{\!\referee{T/\hbar}}}*\sem E(\sem{H}_0)\big)(\omega)\cutMan&= \cutMan \partial_\omega\big(\cutEnergyFourier_{_{\!\referee{T/\hbar}}}*\sem E(\sem{H}_1)\big)(\omega)\cutMan +O(\hbar^{-1}\referee{\tilde{\delta}}(\hbar)T^2+\hbar^\infty)_{\Psi^{-\infty}},
\end{aligned}
$$
which is~\eqref{e:octopus}.

To prove~\eqref{e:squid}, we use Lemma~\ref{l:difference}. Indeed, 
$$\begin{aligned}\cutMan \big(\cutEnergyFourier_{_{\!\referee{T/\hbar}}}*\sem E(\sem{H}_0)\big)(\omega)\cutMan&=\cutMan\cutEnergy(\sem{H}_0)\big(\cutEnergyFourier_{_{\!\referee{T/\hbar}}}*\sem E(\sem{H}_0)\big)(\omega)\cutEnergy(\sem{H}_0)\cutMan+O(\hbar^\infty)_{\Psi^{-\infty}}\\
&=-i\int \frac{1}{\tau} \hat{\cutEnergyFourier}\Big(\frac{\tau}{T}\Big) \sin(\hbar^{-1}\tau \omega)\cutMan\cutEnergy(\sem{H}_0)\cos\big(\tau\sqrt{\sem{H}_0}/\hbar\big)\cutEnergy(\sem{H}_0)\cutMan d\referee{\tau}+O(\hbar^\infty)_{\Psi^{-\infty}}.
\end{aligned}$$
Then, using~\eqref{e:itsACone}, Lemma~\ref{l:basicCutoff} and Lemma~\ref{l:difference} once again,~\eqref{e:squid} follows.
\end{proof}

We now prove Proposition~\ref{p:distantPerturb}
\begin{proof}[Proof of Proposition~\ref{p:distantPerturb}]
Let $\cutEnergyFourier \in \mathscr{S}$ with \referee{$\nu\geq 0$}, $\hat{\cutEnergyFourier}\equiv 1$ on $[-1,1]$, $\supp \hat{\cutEnergyFourier} \subset (-2,2)$, and $\hat{\cutEnergyFourier}$ even. Observe that
$$
\begin{aligned}&(\partial_\omega \sem E(\sem{H}_i)-\partial_\omega \cutEnergyFourier_{_{\!\referee{T/\hbar}}}*\sem E(\sem{H}_i))*\cutEnergyFourier_{_{\!\referee{T/3\hbar}}}(\omega)\\
&=\frac{1}{2\pi \hbar}\int (1-\hat{\cutEnergyFourier}(T^{-1}\tau))\hat{\cutEnergyFourier}(3T^{-1}\tau)e^{i\tau/\hbar (\omega-\sqrt{\sem{H}_i})}d\tau =0.
\end{aligned}
$$
Observe that for $\tilde{\cutMan}\equiv 1$ on $\supp \cutMan$ and any $s\in \mathbb{R}$,
$$
(1-\tilde{\cutMan})(\sem{H}_0+1)^s\cutMan=O(\hbar^\infty)_{H_{\hbar}^{-N}\to H_{\hbar}^N},\qquad \cutMan (\sem{H}_0+1)^s(1-\tilde{\cutMan})=O(\hbar^\infty)_{H_{\hbar}^{-N}\to H_{\hbar}^N}.
$$ 
Therefore, using~\eqref{e:octopus}, for any $s_1,s_2\in \mathbb{R}$ we have
\begin{equation}
\label{e:aardvark}
\cutMan(\sem{H}_0+1)^{s_1}\big(\partial_\omega \sem E(\sem{H}_0)-\partial_\omega (\cutEnergyFourier_{_{\!\referee{T/\hbar}}}*\sem E(\sem{H}_1))\big)(\sem{H}_0+1)^{s_2}\cutMan*\cutEnergyFourier_{\hbar,T/3}(\omega)=O(\hbar^{-1}\referee{\tilde{\delta}}(\hbar)T^2+\hbar^\infty)_{H_{\hbar}^{-N}\to H_{\hbar}^N}.
\end{equation}

Next, using~\eqref{e:lip} and Lemma~\ref{l:lip} with $\cutEnergyFourier_{\hbar}=\cutEnergyFourier_{_{\!\referee{T/\hbar}}}$, we have
\referee{
$$
|\cutEnergyFourier_{_{\!\referee{T/\hbar}}}*\sem E(\sem{H}_1)(x,x,s)-\sem E(\sem{H}_1)(x,x,s)|\leq C\frac{\hbar^{1-d}}{T(\hbar)}+C_N\hbar^N.
$$
Then, using that $\nu\geq 0$, and hence that $\cutEnergyFourier_{_{\!\referee{T/\hbar}}}*\sem E(\sem{H}_1)(x,x,s)$ is monotone in $s$, and~\eqref{e:lip} again we have
\begin{align*}
&\int_{\omega-2\hbar T^{-1}}^{\omega+2\hbar T^{-1}}\Big|\partial_{s} \big(\cutEnergyFourier_{_{\!\referee{T/\hbar}}}*\sem E(\sem{H}_1)\big)(x,x,s)\Big|ds\\
&=\cutEnergyFourier_{_{\!\referee{T/\hbar}}}*\sem E(\sem{H}_1)(x,x,\omega+2\hbar T^{-1})-\cutEnergyFourier_{_{\!\referee{T/\hbar}}}*\sem E(\sem{H}_1)(x,x,\omega-2\hbar T^{-1})\\
&=\sem E(\sem{H}_1)(x,x,\omega+2\hbar T^{-1})-\sem E(\sem{H}_1)(x,x,\omega-2\hbar T^{-1})+O\Big(\frac{\hbar^{1-d}}{T(\hbar)}+\hbar^N\Big)\\
&\leq C\frac{\hbar^{1-d}}{T(\hbar)}.
\end{align*}
}

Therefore, using 
\begin{equation}
\label{e:deltaEst}
\|(\sem{H}_0+1)^k\delta_x\|_{H_{h}^{-s-2k}}\leq Ch^{-\frac{d}{2}}\qquad\text{for }s>\frac{d}{2}
\end{equation}
together with Lemma~\ref{l:monotoneTaub} with $\mu_\hbar=\sem E (\sem{H}_0)(x,x,\cdot)$, $\alpha_\hbar=\cutEnergyFourier_{_{\!\referee{T/\hbar}}}*\sem E(\sem{H}_1)(x,x,\cdot)$, $a_{\hbar}=c\hbar T^{-1}$, $M_\hbar=C\hbar^{-d}$ $B_\hbar=C\hbar^{-1-d}\referee{\tilde{\delta}}(\hbar)T^2+O(\hbar^\infty)$,  we have that the hypotheses of Lemma~\ref{l:lip} hold with $\referee{\tilde{w}_\hbar}=\sem E (\sem{H}_0)(x,x,\cdot)$ $\sigma_\hbar= c\hbar^{-1} T$, $L_{\hbar}= c\hbar T^{-1}(C\hbar^{-d}+C\hbar^{-1-d}\referee{\tilde{\delta}}(\hbar)T^2+O(\hbar^\infty))$, $B_\hbar=\hbar^{-d}$, and hence
\begin{equation}
\label{e:unsmooth0}
|\sem E(\sem{H}_0)(x,x,\omega)-\cutEnergyFourier_{_{\!\referee{T/\hbar}}}* \sem E (\sem{H}_0)(x,x,\omega)|\leq C\frac{\hbar^{1-d}}{T(\hbar)} +C\hbar^{-d}\referee{\tilde{\delta}}(\hbar)T(\hbar).
\end{equation}

Again using~\eqref{e:lip} and Lemma~\ref{l:lip} with $\cutEnergyFourier_{\hbar}=\cutEnergyFourier_{_{\!\referee{T/\hbar}}}$, we obtain
\begin{equation}
\label{e:unsmooth1}
\Big|(\sem E(\sem{H}_1)(x,x,\omega)-\cutEnergyFourier_{_{\!\referee{T/\hbar}}}*\sem E(\sem{H}_1)(x,x,\omega)\Big|\leq C \frac{\hbar^{1-d}}{T(\hbar)}+O(\hbar^\infty).
\end{equation}

Thus,~\eqref{e:squid} and~\eqref{e:deltaEst} imply
\begin{equation}
\label{e:firstCompare}
\begin{aligned}
\cutEnergyFourier_{_{\!\referee{T/\hbar}}}*\sem E(\sem{H}_0)(x,x,\omega)
&=\cutEnergyFourier_{_{\!\referee{T/\hbar}}}*\sem E(\sem{H}_1)(x,x,\omega) + O(\hbar^{-d}T(\hbar)\referee{\tilde{\delta}}(\hbar)+\hbar^\infty).
\end{aligned}
\end{equation}
Combining this with~\eqref{e:unsmooth0} and~\eqref{e:unsmooth1}, we have~\eqref{e:bigBallOk}.

Now, appealing to~\eqref{e:lip2} rather than~\eqref{e:lip} and using that 
$$
(\sem{H}_0+1)^k\sem E(\sem{H}_0)(\sem{H}_0+1)^l(x,x,\omega)
$$
is monotone increasing in $\omega$ we may make the same argument to obtain 
$$
(\sem{H}_0+1)^k\sem E(\sem{H}_0)(\sem{H}_0+1)^l(x,x,\omega)=(\sem{H}_0+1)^k\sem E(\sem{H}_1)(\sem{H}_0+1)^l(x,x,\omega)+ O(\hbar^{-d}T(\hbar)\referee{\tilde{\delta}}(\hbar)+\hbar^\infty).
$$


With this in hand, we can complete the proof. Indeed, notice that \referee{since for $s\geq 0$, the operator $(\sem{H}_0+1)^k[\sem E(\sem{H}_0)(\omega+s)-\sem E(\sem{H}_0)(\omega)](\sem{H}_0+1)^l$ is positive, we have 
\begin{align*}
0&\leq \langle (\sem{H}_0+1)^k(\sem E(\sem{H}_0)(\omega+s)-\sem E(\sem{H}_0)(\omega))(\sem{H}_0+1)^l(\delta_x+\delta_y),\delta_x+\delta_y\rangle \\
&=2((\sem{H}_0+1)^k\sem E(\sem{H}_0)(\sem{H}_0+1)^l)(x,y,\omega+s)+((\sem{H}_0+1)^k\sem E(\sem{H}_0)(\sem{H}_0+1)^l)(x,x,\omega+s)\\
&\qquad +((\sem{H}_0+1)^k\sem E(\sem{H}_0)(\sem{H}_0+1)^l)(y,y,\omega+s)\\
&\qquad-2((\sem{H}_0+1)^k\sem E(\sem{H}_0)(\sem{H}_0+1)^l)(x,y,\omega)+((\sem{H}_0+1)^k\sem E(\sem{H}_0)(\sem{H}_0+1)^l)(x,x,\omega)\\
&\qquad +((\sem{H}_0+1)^k\sem E(\sem{H}_0)(\sem{H}_0+1)^l)(y,y,\omega).
\end{align*}
In particular, the function}
\begin{multline*}
\alpha_0(\omega):=((\sem{H}_0+1)^k\sem E(\sem{H}_0)(\sem{H}_0+1)^l)(x,y,\omega)\\+\frac{1}{2}(((\sem{H}_0+1)^k\sem E(\sem{H}_0)(\sem{H}_0+1)^l)(x,x,\omega)+((\sem{H}_0+1)^k\sem E(\sem{H}_0)(\sem{H}_0+1)^l)(y,y,\omega))
\end{multline*}
is monotone increasing in $\omega$ and, using~\eqref{e:aardvark} and~\eqref{e:deltaEst}, we have
$$
\cutMan\partial_\omega \cutEnergyFourier_{_{\!\referee{T/\hbar}}}*\alpha_0(\omega)\cutMan=\cutMan\partial_\omega \cutEnergyFourier_{_{\!\referee{T/\hbar}}}*\alpha_1(\omega)\cutMan+O(\hbar^{-1-d}\referee{\tilde{\delta}}(\hbar)T^{2}),
$$
where
\begin{multline*}
\alpha_1(\omega):=((\sem{H}_0+1)^k\sem E(\sem{H}_1)(\sem{H}_0+1)^l)(x,y,\omega)\\+\tfrac{1}{2}(((\sem{H}_0+1)^k\sem E(\sem{H}_1)(\sem{H}_0+1)^l)(x,x,\omega)+((\sem{H}_0+1)^k\sem E(\sem{H}_1)(\sem{H}_0+1)^l)(y,y,\omega)).
\end{multline*}
Therefore, by exactly the same argument we used to obtain~\eqref{e:firstCompare}, but using~\eqref{e:lip2} instead of~\eqref{e:lip}, we have
$$
(\sem{H}_0+1)^k\sem E(\sem{H}_0)(\sem{H}_0+1)^l(x,y,\omega)=(\sem{H}_0+1)^k\sem E(\sem{H}_1)(\sem{H}_0+1)^l(x,y,\omega)+ O(\hbar^{-d}T(\hbar)\referee{\tilde{\delta}}(\hbar)+\hbar^\infty).
$$
Finally, the fact that for $U\Subset V$ and $s\in \mathbb{R}$ we have
$$
\|v\|_{H_{\hbar}^{2s}(U)}\leq C_s\|(\sem{H}_0+1)^sv\|_{L^2(V)}+O_{N,s}(\hbar^N)\|v\|_{H_{\hbar}^{-N}},
$$
completes the proof.
\end{proof}

\section{Pseudodifferential calculus in anisotropic symbol classes}
\label{s:pdo}

We first recall the standard notation for semiclassical pseudodifferential operators on $\mathbb{R}^d$ in the Weyl calculus. Throughout this article, we will work with the calculus of polyhomogeneous symbols, although we will need a slight modification. 

\referee{
\begin{definition}
We say that $a\in C^\infty(\mathbb{R}^{2d})$ is a \emph{symbol of order $m$} and write $a\in S^m(\mathbb{R}^{2d})$ if $a=a(x,\xi;\hbar)=a(x,\xi)\in C^\infty(\mathbb{R}^{2d})$ for all $\alpha,\beta\in\mathbb{N}^d$  (where we write $\mathbb{N}=\{0,1,\dots\}$) there is $C_{\alpha\beta}>0$ such that 
\begin{equation}
\label{e:symbolClass}
\sup_{0<\hbar<1}|\partial_x^\alpha\partial_\xi ^\beta a(x,\xi;\hbar)|\leq C_{\alpha\beta}\langle \xi\rangle^{m-|\beta|}.
\end{equation}
Below, we often implicitly allow symbols to depend on $\hbar$, suppressing $\hbar$ in the notation. We write $S^{-\infty}=\bigcap_m S^{m}$ and $S^\infty=\bigcup_m S^m$. \referee{\st{We write $C_c^\infty(\mathbb{R}^{2d})$ for the class of functions in $S^0$ with support in an $\hbar$-independent compact set and with norm inherited from $S^0$.}}
\end{definition}
}

We will need a small variation on the set of polyhomogeneous symbols. To this end, we let 
$$
\referee{\mu_n\equiv \mu_n(\hbar)=\mu_{n(\hbar)}}
$$
 as in Section~\ref{s:muN} with $n(\hbar)$ satisfying~\eqref{e:muNBound}. 
\referee{
\begin{definition}
For $0\leq \delta<1$, we define the \emph{(semiclassically) polyhomogeneous symbols}, $\Sp^m$ as follows.  We say $a\in \Sp^m$ if there are $\{a_j\}_{j=0}^\infty$, $a_j\in \lp_n^{j\delta}S^{m}$, independent of $\hbar$, but depending on $\lp_n$ such that 
\begin{equation}
\label{e:borscht}
a-\sum_{j=0}^{N-1}\hbar^{j}a_j\in \hbar^{N}\lp_{n}^{N\delta}S^{m}.
\end{equation}
Here, we write $a\in f(\hbar,\lp_n)S^m$ if~\eqref{e:symbolClass} holds with $C_{\alpha\beta}$ replaced by $f(\hbar,\lp_n)C_{\alpha\beta}$. 
\end{definition}}

\referee{\begin{remark}
We recall that, as discussed in Section~\ref{s:muN}, it is crucial that $\lp_n$ is locally constant as a function $\hbar$ so that we may glue asymptotics together across intervals. Choosing $\lp_n=\hbar^{-1}$ would not suffice and, although many statements below hold for $\lp_n$ replaced by any $\lp\leq C\hbar^{-1}$, we choose to keep the $n$ in the notation to emphasize the importance of this local constancy.
\end{remark}}
\begin{remark}
One can, of course, replace $\hbar^N\lp_n^{N\delta}$ by $\lp_n^{N(\delta-1)}$ \referee{on} the right-hand side of~\eqref{e:borscht}, but, since these estimates usually occur when the remainder consists of a function whose failure to have one-step polyhomogeneity comes only from the large parameter, $\lp_n$, we choose to keep the notation as is to help the reader.
\end{remark}

\begin{remark}
\referee{The reason that we cannot simply take $\delta=0$ is that, in the onion peeling procedure, we are only able to take finitely many (i.e. a number independent of $\hbar$) steps. On the other hand, if we took $\delta=0$, then to gauge transform away a potential periodic at some scale $\sim \hbar^{-N}$ for some $N$, we would need $|\log \hbar|$ steps. Therefore, we take $\delta>0$ and,
for most purposes,} the reader may think of $\delta=\frac{1}{4}$. \referee{For instance, if the reader is only interested in on-diagonal asymptotics of the spectral function, it suffices to take $\delta=\frac{1}{4}$.}  It is only at the very end of the proof, \referee{when $0<|x-y|=o(1)$,} where we will take $\delta$ arbitrarily small, \referee{see Remark~\ref{r:deltaIsSmall}}. \end{remark}

\begin{definition}
 We define the set of \emph{pseudodifferential operators of order $m$}, $\Psi^m_\delta$, by saying that $A\in \Psi_{\delta}^m$ if there is $a\in \Sp^m$ such that for all $N\in \mathbb{R}$
$$
A=\Op{\hbar}(a) +O(\hbar^{\infty})_{H_{\hbar}^{-N}\to H_\hbar^N},\quad \big[\Op{\hbar}(a)u\big](x;\hbar):=\frac{1}{(2\pi\hbar)^d}\int e^{i\langle x-y,\xi\rangle/\hbar}a\Big(\frac{x+y}{2},\xi;\hbar\Big)u(y)dyd\xi.
$$ 
Here the superscript $\operatorname{W}$ stands for Weyl.
\end{definition}

\begin{remark}
Since we use the Weyl quantisation, $\Op{\hbar}(a)$ \referee{with $a\in \Sp^0$} is self-adjoint on $L^2(\referee{\mathbb{R}^d})$ if $a$ is real valued.
\end{remark}


\referee{
\begin{definition}
We write $a\in S_\delta^{\fcomp}$ and say $a$ is \emph{momentum compactly supported} if $a\in \Sp^0$ and there is an $\hbar$-independent, compact set $\referee{\mathcal{K}}\subset \mathbb{R}^d$ such that for all $\hbar\in (0,1]$
$$
\supp a\subset \mathbb{R}^d\times \referee{\mathcal{K}}.
$$
We write $\Psi_\delta^{\fcomp}$ for the corresponding class of operators; here, $\fcomp$ stands for momentum compact.
\end{definition}}

\referee{\begin{definition}
We say that a distribution, $u$, is \emph{$\hbar$-tempered} if there is $N>0$ and $C>0$ such that for all $\hbar \in(0,1]$, we have
$$
\|u\|_{H_{\hbar}^{-N}}\leq C\hbar^{-N}.
$$
\end{definition}

\begin{definition}
For an $\hbar$-tempered distribution, $u$, we define the \emph{wavefront set of $u$}, $\WF_{\hbar}(u)$, as follows.  For $(x_0,\xi_0)\in \mathbb{R}^{2d}$, we say that $(x_0,\xi_0)\notin \WF_{\hbar}(u)$ if there is $a\in C_c^\infty(\mathbb{R}^{2d})$ independent of $\hbar$ such that $a(x_0,\xi_0)=1$ and for all $N$ and $\hbar\in (0,1]$
$$
\|\Op{\hbar}(a)u\|_{H_{\hbar}^{N}}\leq C_N \hbar^{N}.
$$
\end{definition}
\begin{definition}
We say that $u$ is \emph{$\hbar$-compactly microlocalized} if there is $\cutPhase \in C_c^\infty(\mathbb{R}^{2d})$ independent of $\hbar$ and for all $N$ there is $C_N>0$ such that 
$$
\|\Op{\hbar}(1-\cutPhase)u\|_{H_{\hbar}^{N}}\leq C_N \hbar^{N}.
$$
\end{definition}}

\subsection{Anisotropic Pseudodifferential operators}
In this subsection, we study a class of pseudodifferential operators which improve after differentiation in $x$. These classes will be required in the onion peeling process (see Section~\ref{s:gauge}). 

\referee{\begin{definition} Let $r:[1,\infty)\to (0,1]$ be non-increasing. We write $a\in S^{m}_{r,\delta}$ if $a\in \Sp^m$ with $a\sim \sum_j \hbar^{j}a_j$, and for all $\alpha,\beta\in\mathbb{N}^d$, there is $C_{\alpha\beta j}>0$ such that
$$
|\partial_x^\alpha\partial_\xi ^\beta a_j(x,\xi;\mu_n)|\leq C_{\alpha\beta j}\lp_{n}^{j\delta}r(\lp_{n})^{|\alpha|}\langle \xi\rangle^{m-j-|\beta|}.
$$
We write $\Psi_{r,\delta}^m$ for the corresponding class of operators, with $S^{-\infty}_{r,\delta},\,S^{\infty}_{r,\delta}, S^{\fcomp}_{r,\delta}$, and $\Psi^{-\infty}_{r,\delta},\Psi^\infty_{r,\delta},\Psi_{r,\delta}^{\fcomp}$ as above. Note that $\Sp^m=S^m_{1,\delta}$. 
\end{definition}}

\begin{remark}
\label{r:roughSymbols}
Although we make the assumption that $a_j$ $j=0,1,\dots$ are infinitely smooth, it is clear from standard results in the pseudodifferential calculus (see e.g.~\cite[Theorem 4.23]{Zw:12}) that, if one is only interested in the pseudodifferential calculus modulo remainders of size $h^N$ for some $N$, then there is a $K>0$ such \referee{that bounds on the $C^K$ norm of the $a_j$'s} are enough for proving the results of this paper. 
\end{remark}

\begin{remark}
In reality, we will need only the values of $r$ at the discrete points $\lp_n$ and we will be interested only in $r(\mu)=\mu^{-\gamma}$ for some $\gamma>0$. However, for notational convenience we use a function $r$. Below, when we write the letter $r$, we will mean the function $r(\mu_{n(\hbar)})$. 
\end{remark}

\referee{
We will often use the following analogue of Borel summation for our symbols, the proof of which follows the standard one (see e.g.~\cite[Theorem 4.5]{Zw:12}).
\begin{lemma}
\label{l:borelSum}
Let $0\leq \delta <1$, $\referee{\mathcal{K}}\subset \mathbb{R}^n$ compact and $\{g_j\}_{j=0}^\infty \in S_{r,\delta}^{\fcomp}$ such that $\supp g_j\subset \mathbb{R}^n\times \referee{\mathcal{K}}$. Then there is $g\in S_{r,\delta}^{\fcomp}$ such that 
$$
g\sim \sum_j h^j\mu_n^{j\delta}g_j
$$
in the sense that 
$$
g-\sum_{j=0}^{N-1}h^j\mu_n^{j\delta}g_j\in h^N\mu_n^{N\delta}S_{r,\delta,}^{\fcomp}
$$
and, moreover, $\supp g\subset \mathbb{R}^n\times \referee{\mathcal{K}}$.
\end{lemma}
}

\referee{
\begin{definition}
For $r>0$, we define unitary operators $U_{r}:L^2\to L^2$ by 
$$
U_{r}u(x):=r^{\frac{d}{2}}u(r x).
$$ 
Their adjoints are given by $U_{r}^*:L^2\to L^2$ with
 $$
 U_{r}^*u(x):=r^{-\frac{d}{2}}u(r^{-1}x).
 $$
 Note also that $U_{r}U_s=U_{rs}$. 
\end{definition}
} 
\begin{lemma}
\label{l:unitaryEquiv}
Let $0\leq \delta<1$, and $a\in S_{r,\delta}^m$. Then
$$
U_{r}^*\Op{\hbar}(a)U_{r} = \Op{r\hbar}(\tilde{a}_r),
$$
where $\tilde{a}_r\in \Sp^m$ is defined by 
\begin{equation}
\label{e:ar}\tilde{a}_r(x,\xi;\hbar):=a(r^{-1}x,\xi;\hbar ).
\end{equation}
\end{lemma}
\begin{proof}
\begin{align*}
(U_{r}^*\Op{\hbar}(a)U_{r} u)(x)&=r^{-\frac{d}{2}}\frac{1}{(2\pi\hbar )^d} \int e^{i \langle r^{-1}x-y,\xi\rangle/\hbar}a(\tfrac{r^{-1} x+y}{2},\xi)[U_{r} u](y)dyd\xi\\
&=\frac{1}{(2\pi \hbar)^d} \int e^{i \langle r^{-1}(x-r y),\xi\rangle/\hbar}a( \tfrac{r^{-1} x+y}{2},\xi)u(r y)dyd\xi\\
&=\frac{1}{(2\pi r \hbar )^d} \int e^{i \langle x- w,\xi\rangle/(r\hbar)}a( \tfrac{r^{-1} (x+w)}{2},\xi)u( w)dw d\xi\\
&=[\Op{r \hbar}(\tilde{a}_r)u](x).
\end{align*}
The fact that $\tilde{a}_r\in \Sp^m$ follows easily from the definition of $S^m_{r,\delta}$.
\end{proof}

\referee{
\begin{remark}
Notice that the proof of Lemma~\ref{l:unitaryEquiv} shows that the pseudodifferential calculus can be used in the classes $S_{r,\delta}^m$. In particular, if $a\in S_{r,\delta}^{m_1}$ and $b\in S_{r,\delta}^{m_2}$, then $\Op{\hbar}(a)\Op{\hbar}(b)=\Op{\hbar}(e)$ for some $e\in S^{m_1+m_2}_{r,\delta}$.  
\end{remark}
}

\referee{\st{Since the gauge transform involves commutators, we need to be able to compute the symbol of the commutator of two pseudodifferential operators. } Our next lemma will allow us to understand conjugation of pseuodifferential operators by $e^{iG}$ for $G=\Op{\hbar}(g)$ and $g\in S_{r,\delta}$.} \referee{Denote} 
$$
\operatorname{ad}_AB:=[A,B].
$$
\begin{lemma}
\label{l:conjugate}
Let $0\referee{\leq}\delta<1$, $N_0>0$, $m\in \mathbb{R}$, and suppose that $ \hbar^{N_0}\leq  r(\lp_{n})\leq 1$. Suppose that $a\in \Sp^m$, $g\in r^{-1}\lp_{n}^\delta S_{r,\delta}^{\fcomp}$ is real valued, and $b\in S_{r,\delta}^{\referee{\fcomp}}$ are such that for all $N$
\begin{equation}
\label{e:commuteReferee}
[\Op{\hbar}(g),\Op{\hbar}(a)]=\lp_{n}^\delta\hbar\Op{\hbar}(b) +O(\hbar^{\infty})_{H_{\hbar}^{-N}\to H_{\hbar}^N}.
\end{equation}
Then, with $G:=\Op{\hbar}(g)$,
$$
e^{-i\referee{G}}\Op{\hbar}(a)e^{i\referee{G}}\sim \sum_{j=0}^{\infty} \frac{ \ad_{G}^j \Op{\hbar}(a)}{i^j j!}.
$$
 This asymptotic formula holds in the sense that for all $N>0$ and $\ell\in\mathbb{R}$, there is $\referee{L}_0\in\mathbb{R}$ such that for $\referee{L\geq L_0}$ we have
 $$
 \Big\|e^{-i\referee{G}}\Op{\hbar}(a)e^{i\referee{G}}- \sum_{j=0}^{\referee{L}} \frac{ \ad_{G}^j \Op{\hbar}(a)}{i^j j!}\Big\|_{H_{\hbar}^{-\ell}\to H_\hbar^{\ell}}\leq C_{N,\ell}\hbar^{N}.
 $$
 In addition, for all $\ell$, $e^{i\referee{G}}:H_{\hbar}^{\ell}\to H_{\hbar}^\ell$ is bounded.
\end{lemma}
\referee{\begin{remark}
Although in this paper we only use Lemma~\ref{l:conjugate} when $a\in S_{r,\delta}^{m}$, in which case the proof can be reduced to the standard one by conjugating with $U_r$, we expect the statement above to be useful in other contexts and therefore choose to make a more general formulation.
\end{remark}}
\begin{remark}
\referee{Notice that since $g$ is momentum compact, the left hand side of~\eqref{e:commuteReferee} maps $H_{\hbar}^{-N}$ to $H_{\hbar}^N$ for any $N$ and hence it is natural to assume that $b$ can be taken independent of $N$.}
\end{remark}
\begin{proof}

\referee{We first show that for any $\ell\in \mathbb{R}$ and $t\in[-1,1]$, there is $C_\ell>0$ such that 
\begin{equation}
\label{e:coffeeClaim}
\|e^{itG}\|_{H_{\hbar}^\ell\to H_{\hbar}^\ell}\leq C_\ell. 
\end{equation}
To see this, observe that using Lemma~\ref{l:unitaryEquiv},  we have
$$
U_r^*e^{itG}U_r=e^{itU_r^*GU_r}= e^{it\Op{\hbar r}(\tilde{g})}
$$
for some $\tilde{g}\in \Sp^{\fcomp}$. In particular, this implies $U_r^*e^{itG}U_r =\Op{r\hbar}(b)$ for some $b\in \Sp^0$ and hence that 
\begin{equation}
\label{e:teaClaim}
\|U_r^*e^{itG}U_r\|_{H_{r\hbar}^\ell\to H_{r\hbar}^\ell}\leq C_\ell.
\end{equation}
Now, since 
$$
\|(\hbar r\partial_x)^\alpha U_r^* u\|_{L^2}=\|U_r(\hbar r\partial_x)^\alpha U_r^* u\|_{L^2}=\|(\hbar \partial_x )^{\alpha}u\|_{L^2},
$$
we have 
$$
\|U_r^*u\|_{H_{r\hbar}^\ell}=\|u\|_{H_\hbar^\ell}, 
$$
and hence~\eqref{e:coffeeClaim} follows from~\eqref{e:teaClaim}.

}

\referee{By Taylor's formula,} for $N\geq 1$
\begin{align*}
&e^{-i\referee{G}}\Op{\hbar}(a)e^{i\referee{G}}\\&=\sum_{k=0}^{N-1}\frac{\ad_{\referee{G}}^k\Op{\hbar}(a)}{i^kk!}-\int_{0}^1\frac{(1-s)^{N-1}}{i^N(N-1)!}e^{-is\referee{G}}\ad_G^N\Op{\hbar}(a)e^{is\referee{G}}ds\\
&=\sum_{k=0}^{N-1}\frac{\ad_{\referee{G}}^k\Op{\hbar}(a)}{i^kk!}-\int_{0}^1\frac{\hbar \lp_n^\delta(1-s)^{N-1}}{i^N(N-1)!}e^{-is\referee{G}}U_{r}\ad_{\Op{\hbar r}(\tilde{g})}^{N-1}\Op{\hbar r}(\tilde b)U_{r}^*e^{is\referee{G}} ds\\
&\qquad +O(\hbar^{\infty})_{H_{\hbar}^{-N}\to H_{\hbar}^N},
\end{align*}
where the last equality follows from Lemma~\ref{l:unitaryEquiv} with $\tilde{g}$ and $\tilde{b}$ given by~\eqref{e:ar}.

Now, since $\tilde{g}\in r^{-1}\lp_{n}^\delta \Sp^{\fcomp}$ and $\tilde{b}\in \Sp^{\fcomp}$, 
$$
\ad_{\Op{\hbar r}(\tilde{g})}^{N-1}\Op{\hbar r}(\tilde b)\in \hbar^{(N-1)}\lp_n^{\delta(N-1)} \Psi_r^{\fcomp}.
$$
Using Lemma~\ref{l:unitaryEquiv} again, 
$$
E_{N}:=U_{r}\ad_{\Op{\hbar r}(\tilde{g})}^{N-1}\Op{\hbar r}(\tilde b)U_{r}^*\in \hbar^{N-1}\lp_{n}^{\delta(N-1)} \Psi^{\fcomp}, 
$$ 
and hence
\begin{equation}
\label{e:integratedForm}
\begin{aligned}
&e^{-i\referee{G}}\Op{\hbar}(a)e^{i\referee{G}}\\
&=\sum_{k=0}^{N-1}\frac{\ad_{\referee{G}}^k\Op{\hbar}(a)}{i^kk!}-\int_{0}^1\frac{\hbar \referee{\lp_n^\delta} (1-s)^{N-1}}{i^N(N-1)!}e^{-is\referee{G}}E_N e^{is\referee{G}} ds +O(\hbar^{\infty})_{H_{\hbar}^{-N}\to H_{\hbar}^N}.
\end{aligned}
\end{equation}
Now, using\referee{~\eqref{e:coffeeClaim}\st{ that $e^{is\referee{G}}$ is unitary, }}we obtain that for $N\geq 1$ \referee{and any $\ell\in \mathbb{R}$} we have
\begin{equation}
\label{e:firstConj}
e^{-i\referee{G}}\Op{\hbar}(a)e^{i\referee{G}}=\sum_{k=0}^{N-1}\frac{\ad_{\referee{G}}^k\Op{\hbar}(a)}{i^kk!} +O(\hbar^{N}\lp_{n}^{N\delta})_{\referee{H_\hbar^{-\ell}\to H_{\hbar}^{\ell}}}.
\end{equation}

\end{proof}

\referee{\begin{remark}
In principle, one could work directly on the conjugated side, writing asymptotic formulae for $U_r^*e^{-iG}\Op{\hbar}(a)e^{iG}U_r$ instead of those in Lemma~\ref{l:conjugate} by using Lemma~\ref{l:unitaryEquiv}, but we have chosen not to do this.
\end{remark}
}

We will also need the next lemma which controls how the operator $e^{i\Op{\hbar}(g)}$ moves singularities.
\begin{lemma}
\label{l:waveFront1}
Let $0\leq \delta<1$, and $N_0>0$, $c>0$. Suppose that $\hbar^{N_0}<r=r(\lp_n )\leq 1$ and $g\in r^{-1}\lp_{n}^\delta S_{r,\delta}^{\fcomp}$ is real valued. Then for all $a,b\in S^0$ with $\referee{dist}(\supp a,\supp b)>c>0$, we have
$$
\Op{\hbar r}(b)U_{r}^*e^{i\Op{\hbar}(g)}U_{r}\Op{\hbar r}(a) =O(\hbar^{\infty})_{H_{\hbar r}^{-\ell}\to H_{\hbar r}^\ell}.
$$
\end{lemma}
\begin{proof}
Observe that 
\begin{align*}
\Op{\hbar r}(b)U_{r}^*e^{i\Op{\hbar}(g)}U_{r}\Op{\hbar r}(a)&=U_{r}^*U_{r}\Op{\hbar r}(b)U_{r}^*e^{i\Op{\hbar}(g)}U_{r}\Op{\hbar r}(a)U_{r}^*U_{r}\\
&=U_{r}^* \Op{\hbar}(\tilde{b})e^{i\Op{\hbar}(g)}\Op{\hbar}(\tilde{a})U_{r}\\
&=U_{r}^*e^{i\Op{\hbar}(g)}e^{-i\Op{\hbar}(g) }\Op{\hbar}(\tilde{b})e^{i\Op{\hbar}(g)}\Op{\hbar}(\tilde{a})U_{r},
\end{align*}
where $\tilde{a}$ and $\tilde{b}$ are as in~\eqref{e:ar}. The lemma now follows from Lemma~\ref{l:conjugate}.
\end{proof}

Later, we will need an oscillatory integral formula for $e^{\frac{i}{r\hbar}t\Op{\hbar}(g)}$. This is given in our next lemma.
\begin{lemma}
\label{l:localForm}
Suppose that $N_0>0$, $\hbar^{N_0}\leq r\leq 1$, $\referee{S}>0$ $g\in S_{r,\delta}^{\fcomp}$. Then for $(x_0,\xi_0)\in \mathbb{R}^{2d}$, there is a neighbourhood $U$ of $(x_0,\xi_0)$ and $\varphi \in C^\infty([-\referee{S},\referee{S}]\times U)$ and $\referee{b}\in C^\infty([-\referee{S},\referee{S} ];\Sp^{\fcomp})$ such that for any $u$ with $\WF_{\hbar r}(u)\subset U$, we have
\begin{equation}
\label{e:form}
U_{r}^*e^{\frac{i}{r\hbar}t\Op{\hbar}(g)}U_{r} u(x)=\frac{1}{(2\pi r \hbar)^d}\int e^{\frac{i}{\hbar r}(\varphi(t,x,\eta)-\langle y,\eta\rangle)} \referee{b}(t,x,\eta)u(y)dy d\eta +O(\hbar^{\infty})_{H_{\hbar}^{-N}\to H_{\hbar}^N}.
\end{equation}
Moreover, 
$$
\partial_t \varphi(x,\eta)= g(r^{-1}x,\partial_x\varphi(x,\eta)),\qquad \varphi(0,x,\eta)=\langle x,\eta\rangle,
$$
and
$$
\referee{b}(t,x,\eta)=(\det \partial_{x\eta}\varphi)^{1/2}+O(\hbar r)_{C_c^\infty}.
$$
\end{lemma}
\begin{proof}
The lemma is a direct consequence of Lemma~\ref{l:unitaryEquiv} and~\cite[Theorem 10.4]{Zw:12}.
\end{proof}

Finally, we record the following lemma on changing scales.
\begin{lemma}
\label{l:waveFront2}
Let $N_0>0$ and $r_1,r_2:[1,\infty)\to (0,1]$ be non-increasing \referee{functions} with $\hbar^{N_0}\leq r_1(\mu_{n(\hbar)})\leq r_2(\mu_{n(\hbar)})\leq 1$. Suppose that $u$ is $\hbar r_2$-compactly microlocalized.
Then
$$
\WF_{\hbar r_1}(U_{r_1}^*U_{r_2}u)\subset \bigcap_{0<\hbar_0<1}\overline{\bigcup_{0<\hbar<\hbar_0} \{(r_1r_2^{-1}x,\xi)\mid (x,\xi)\in \WF_{\hbar r_2}(u)\}}
$$
and $U_{r_1}^*U_{r_2}u$ is $\hbar r_1$-compactly microlocalized.
\end{lemma}
\begin{proof}
First, observe that
$$\referee{\mathcal{K}}:=\bigcap_{0<\hbar_0<1}\overline{\bigcup_{0<\hbar<\hbar_0} \{(r_1r_2^{-1}x,\xi)\mid (x,\xi)\in \WF_{\hbar r_2}(u)\}} $$
is closed. Therefore, if $(x_0,\xi_0)\notin \referee{\mathcal{K}}$, there is a neighbourhood, $U$, of $(x_0,\xi_0)$ such that $U\cap \referee{\mathcal{K}}=\emptyset.$  Suppose that $a\in C_c^\infty(U)$, $a(x_0,\xi_0)=1$. 
Then
\begin{align*}
[\Op{\hbar r_1}(a)]U_{r_1}^*U_{r_2}u&=\frac{1}{(2\pi \hbar r_1)^d} r_2^{\frac{d}{2}}r_1^{-\frac{d}{2}}\int e^{\frac{i}{\hbar r_1}\langle x-y,\xi\rangle}a(\tfrac{x+y}{2},\xi)u(r_2r_1^{-1}y)dyd\xi\\
&=\frac{1}{(2\pi  \hbar r_2)^d} r_2^{\frac{d}{2}}r_1^{\frac{-d}{2}}\int e^{\frac{i}{\hbar r_2}\langle r_2r_1^{-1}x-w,\xi\rangle}a(\tfrac{r_1r_2^{-1}(r_2r_1^{-1}x+w)}{2},\xi)u(w)dwd\xi\\
&=U^*_{r_1r_2^{-1}}[\Op{\hbar r_2 }(\tilde{a})u],
\end{align*}
where $\tilde{a}=\tilde{a}_{r_1 r_2^{-1}}\in S^0$ is defined as in~\eqref{e:ar}.  Moreover, by construction, $\supp\tilde{a}\cap \WF_{\hbar r_2}(u)=\emptyset$, and, since $u$ is $\hbar r_2$-compactly microlocalized, 
$$
[\Op{\hbar r_2}(\tilde{a})u]=O(\hbar^{\infty})_{H_{\hbar}^\ell}.
$$

The compact microlocalization of $U_{r_1}^*U_{r_2}u$ follows from the fact that $\pi_{x}\referee{\mathcal{K}}$ is compact, there is $\cutPhase \in C_c^\infty(\mathbb{R}^{2d})$ such that $u=\Op{\hbar r_2}(\cutPhase)u+O(\hbar^{\infty})_{H_{\hbar}^\ell}$, and that if $d(\supp \tilde{a},\supp \cutPhase)>0$, then
\begin{align*}
[\Op{\hbar r_1}(a)]U_{r_1}^*U_{r_2}\Op{\hbar r_2}(\cutPhase)=U^*_{r_1r_2^{-1}}\Op{\hbar r_2}(\tilde{a})\Op{\hbar r_2}(\cutPhase)=O(\hbar^{\infty})_{H_{\hbar}^{-\ell}\to H_{\hbar}^\ell}.
\end{align*}
\end{proof}

\section{The Gauge Transform for USB potentials}
\label{s:gauge}
Let $\mathbf{q}_0\in  S^1(\mathbb{R}^2)$ be real valued \referee{and satisfy 
\begin{equation}
\label{e:supportMe}
\supp \mathbf{q}_0\subset\{a<|\xi|< b\},
\end{equation}
for some $0<a<b$. In Section~\ref{s:fourierMultiplier}, we will show that, for the purposes of computing the spectral function at some energy $\omega\in(a,b)$,  we may assume that~\eqref{e:supportMe} holds.}
We consider the operator
\begin{equation}
\label{e:H0}
\semOp{\mathbf{Q}_0}:=-\hbar^2\Delta +\hbar\Op{\hbar}(\mathbf{q}_0).
\end{equation}
The goal for this section is, given $\referee{N}>0$, to perform a Gauge transform with a unitary operator $U$ such that $U^*\semOp{\mathbf{Q}_0}U=\semOp{\mathbf{Q}_1}\referee{+O(\hbar^\infty)_{\Psi^{-\infty}}}$ \referee{with}
$$
\semOp{\mathbf{Q}_{\referee{1}}}=-\hbar^2\Delta+ \hbar\Op{\hbar}(\referee{\mathbf{q}_1}),
$$
where $\referee{\mathbf{q}_1}\in  S^1$ is real valued and
$$\supp \referee{\hat{\mathbf{q}}_1}(\theta,\xi)\cap 
\{|\theta |\geq\lp_n^{-\referee{N}}\}=\emptyset.$$ 
\referee{(Recall the definition of $\lp_n$ from~\eqref{e:defMoo} and~\eqref{e:muNBound}.)}

\begin{remark}
For the gauge transform we do not need to assume that $\mathbf{q}_0$ is periodic. What is important is that if $\mathbf{q}_0$ is periodic, then so is $\referee{\mathbf{q}_1}$ and $\referee{\mathbf{q}_1}$ has the same period as $\mathbf{q}_0$. Thus, when we apply this gauge transform to a $\mathbf{q}_0$ with period $\ll \lp_n^{\referee{N}}$ we will have that 
$$\supp \referee{\hat{\mathbf{q}}_1}(\theta,\xi)\subset\{\theta=0\},$$
so $\referee{\mathbf{Q}_1}$ is a Fourier.  
We will use this fact to obtain a formula for the spectral function of $\semOp{\referee{\mathbf{Q}_1}}$. 
\end{remark}

We will use an `onion peeling' strategy to perform the Gauge transform. In particular, we will remove the frequencies of $\mathbf{q}_0$ in layers starting from those with frequency larger than 1 and then removing successive layers. These layers will be evenly spaced in a logarithmic with the factor \referee{$\lp_n^{-\delta}$},
i.e. of the form $\referee{\lp_n^{-(k+1)\delta}}\leq |\theta|\leq \referee{\lp_n^{-k\delta}}$, $k=0,1,\dots, \lfloor \referee{N}\delta^{-1}\rfloor$. 

\subsection{Two useful lemmas}
We will need the following two lemmas to perform the gauge transform. These lemmas allow us to find a \referee{symbol} $g$ that solves the equation $[-\hbar^2\Delta,\Op{\hbar}(g)]=\Op{\hbar}(\mathbf{q})$ under certain assumptions on the support of the Fourier transform of $\mathbf{q}$.

The next lemma is, in fact, about functions of the single variable with $\xi$ playing the role of a parameter.
\begin{lemma}
\label{l:integrate}
There is $C>0$ such that for all $\referee{\iota}>0$ and $\mathbf{q}\in S^0$ with $\supp \hat {\mathbf{q}}\subset \{|\theta|\geq \referee{\iota}\}$, setting 
$$
g(x,\xi):=\int_0^x \mathbf{q}(s,\xi)ds,
$$
we have 
$$
\|\partial_x^\alpha \partial_\xi^\beta g(\cdot,\xi)\|_{L^\infty}\leq C\referee{\iota}^{-1}\| \partial_x^\alpha \partial_{\xi}^\beta \mathbf{q}(\cdot,\xi)\|_{L^\infty}\quad \alpha,\beta\in \mathbb{N},\qquad \supp \hat{g}\subset \supp \hat{\mathbf{q}}\cup\{\theta=0\}.
$$
Moreover, if $\mathbf{q}$ is $L$-periodic in $x$, then so is $g$. 
\end{lemma}
\begin{remark}
\referee{To see that $g$ is periodic when $\mathbf{q}$ is, we use crucially that $\{\theta=0\}\notin \supp \hat{\mathbf{q}}$}. 
\end{remark}
\begin{proof}
Let $\mathbf{q}(x,\xi)\in S^1$ with $\supp \hat{\mathbf{q}}\subset \{|\theta|\geq \referee{\iota}\}$. Let $\auxCut \in C_c^\infty(-1,1)$ with $\auxCut \equiv 1$ on $[-\tfrac{1}{2},\tfrac{1}{2}]$, $\check{\auxCut}$ real valued, and define
$$
I(x):=1_{[0,\infty)}(x)-\int_{-\infty}^x \check{\auxCut}(s)ds.
$$
\referee{Observe that  for $x\geq 0$, 
\begin{align*}
|I(x)|&=\Big|1_{[0,\infty)}(x) -\int_{-\infty}^x \referee{\check{f}}(s)ds\Big|\\
&=\Big|1 -\int_{-\infty}^x \referee{\check{f}}(s)ds\Big|=\Big|\int_x^\infty \referee{\check{f}}(s)ds\Big| \leq C_N\int_x^\infty \langle s\rangle^{-N}ds\leq C_{N}\langle x\rangle^{-N+1}.
\end{align*}
Next, for $x<0$,
\begin{align*}
|I(x)|&=\Big|1_{[0,\infty)}(x) -\int_{-\infty}^x \referee{\check{f}}(s)ds\Big|\\
&=\Big|-\int_{-\infty}^x \referee{\check{f}}(s)ds\Big| \leq C_N\int_{-\infty}^x \langle s\rangle^{-N}ds\leq C_{N}\langle x\rangle^{-N+1}.
\end{align*}
Combining these two estimates, we obtain $I\in L^1$.}

Let $I_{\referee{\iota}}(x):=I(\referee{\iota}x)$ and 
$$
\tilde{g}(x,\xi):=[I_{\referee{\iota}}(\cdot)*\mathbf{q}(\cdot,\xi)](x).
$$
We compute 
$$
\partial_x \tilde{g}(x,\xi)= [(\delta_0(\cdot)-\referee{\iota}\referee{\check{f}}(\referee{\iota}\cdot))*\mathbf{q}(\cdot,\xi)](x)=\mathbf{q}(x,\xi),
$$
since $\supp \hat{\mathbf{q}}\cap \{|\theta|< \referee{\iota}\}=\emptyset$ \referee{(here $\delta_0$ is the Dirac delta function)}.

Since $I\in L^1$, we have
\begin{equation}
\label{e:derEst}
\|I_{\referee{\iota}}* \partial_{x}^\alpha\partial_{\xi}^\beta \mathbf{q}\|_{L^\infty}\leq C\referee{\iota}^{-1} \|\partial_{x}^\alpha\partial_{\xi}^\beta \mathbf{q}\|_{L^\infty}.
\end{equation}

Now observe that $g(x,\xi)=\tilde{g}(x,\xi)-\tilde{g}(0,\xi)$. \referee{Then~\eqref{e:derEst}  implies the derivative estimates on $g$. In addition, \referee{since $\tilde{g}(0,\xi)$ does not depend on $x$, we have} $\supp \widehat{\tilde{g}(0,\xi)}\subset \{\theta=0\}$ and hence $\supp \widehat{g(\cdot,\xi)}\subset \supp \widehat{\tilde{g}(\cdot,\xi)}\cup\{\theta=0\}$, which completes the proof.}

The statement about periodicity of $g$ is obvious.

\referee{
\begin{remark}
The reader may wonder why we choose to prove the lemma via $\tilde{g}$ as opposed to simply putting $\hat{g}(\theta,\xi)=\frac{1}{i\theta}\hat{\mathbf{q}}(\theta,\xi)$. To us it seems simpler to check $L^\infty$ bounds on the physical than on the Fourier side.\end{remark}
}

\end{proof}

Lemma~\ref{l:integrate} has the following immediate consequence.
\begin{lemma}
\label{l:commute1}
Let $0\leq \delta<1$, $0<a<b$. There is $C>0$ such that for $r>0$ and $\mathbf{q}(x,\xi)\in \Sp^{\fcomp}$ real valued with 
$$
 \supp \hat{\mathbf{q}}(\theta,\xi)\subset \{a\leq |\xi|\leq b,\,|\theta|>r\},
$$
 there is $g\in \Sp^{\fcomp}$ real valued with 
\begin{gather}
\supp \hat{g}\cap \{ 0<|\theta |\leq r\}=\emptyset,\nonumber\\
\label{e:oneMillion} \|\partial_\xi^\beta\partial_{x}^\alpha g(\cdot,\xi)\|_{L^\infty}\leq C r^{-1}\|\partial_x^\alpha\partial_\xi^\beta \mathbf{q}(\cdot,\xi)\|_{L^\infty}, \quad \alpha,\beta\in \mathbb{N},
\end{gather}
and 
$$
i[\Op{\hbar}(g),-\hbar^{2}\Delta]=\hbar\Op{\hbar}(\mathbf{q}).
$$
In particular,\referee{~\eqref{e:oneMillion} implies that} for any $\referee{r_1=r_1(\mu_n)>0}$, if $\mathbf{q}\in S_{r_1,\delta}^{\fcomp}$, then $g\in r^{-1}S^{\fcomp}_{r_1,\delta}$. 
Moreover, if $\mathbf{q}$ is $L$-periodic in $x$, then so is $g$.
\end{lemma}
\begin{proof}
By Lemma~\ref{l:integrate}, there is a real valued $g\in \Sp^{\fcomp}$ such that $\supp \hat{g}\subset \supp \hat{\mathbf{q}}$ and 
$$
-2\xi\partial_x g (x,\xi)= \mathbf{q}(x,\xi),\qquad \|\partial_x^\alpha \partial_{\xi}^\beta g(\cdot,\xi)\|_{L^\infty} \leq Cr^{-1}\|\partial_x^\alpha \partial_\xi^\beta \mathbf{q}(\cdot,\xi)\|_{L^\infty}.
$$
In particular, 
\begin{equation}
\label{e:niceFormula}
g=-\frac{1}{2\xi}\int_0^x \mathbf{q}(s,\xi)ds.
\end{equation}

Direct computations show that 
$$
i[\Op{\hbar}(g),-\hbar^{2}\Delta]=-2\hbar\Op{\hbar}(\xi \partial_x g)=\hbar\Op{\hbar}(\mathbf{q}),
$$
 that if $\mathbf{q}\in S_{r_1,\delta}^0$, then $g\in r^{-1}S^{\fcomp}_{r_1,\delta}$, and that if $\mathbf{q}$ is $L$-periodic in $x$, then so is $g$.
\end{proof}

\begin{remark}
Observe that~\eqref{e:niceFormula} is essentially the same as~\eqref{LP:solution}, but~\eqref{LP:solution} is not very convenient for obtaining $L^\infty$ type estimates.
\end{remark}

\subsection{The onion peeling argument}

The gauge transform will proceed by a layer peeling type argument. That is, we remove successive layers of the Fourier transform of the perturbation. Each layer will be removed by a parallel gauge transform. We start, in Lemma~\ref{l:layer0}, by removing frequencies larger than 1. Then, in Lemma~\ref{l:nextLayer}, we show that it is possible to remove lower frequencies in layers of the form \referee{$\lp_n^{-\delta} r\referee{(\mu_n)}<|\theta|<r\referee{(\mu_n)}$\st{$\step r\referee{(\mu_n)}<|\theta|<r\referee{(\mu_n)}$}} for any $N>0$ and $r\referee{(\mu_n)}>\hbar^{N}$. These lemmas are combined in Proposition~\ref{p:gauge} to complete our onion peeling argument. For a more detailed heuristic description of this procedure, we refer the reader to Section~\ref{s:onionIntro}. We start by using a parallel gauge transform to remove frequencies larger than 1.
\begin{lemma}
\label{l:layer0}
Let $0<a<b$ and suppose that $\semOp{\mathbf{\mathbf{Q}}_0}$ satisfies~\eqref{e:H0}. Then there is \referee{$G=\Op{\hbar}(g)\in \Psi_0^{\fcomp}$} such that 
$$
e^{-iG}\semOp{\mathbf{\mathbf{Q}}_0}e^{iG}=-\hbar^2\Delta +\hbar \Op{\hbar} (\mathbf{q}_1) +O(\hbar^{\infty})_{H_{\hbar}^{-N}\to H_{\hbar}^{N}},
$$
with real-valued $\mathbf{q}_1\in \Spz^1$ satisfying
\begin{equation}
\label{e:scaresLeonid}
\supp \hat{\mathbf{q}}_1(\theta,\xi) \cap \{|\theta|\geq 1\}=\emptyset,\qquad \referee{\supp \mathbf{q}_1\subset \{a\leq |\xi|\leq b\}}.
\end{equation}
In addition, $\mathbf{q}_1\referee{\in\Spz^1}$ and $\referee{g}\referee{\in S_0^{\fcomp}}$ depend continuously on $\mathbf{q}\referee{\in S^1}$ \referee{ in the corresponding topologies} and if $\mathbf{q}$ is $L$-periodic in $x$, then so are $\mathbf{q}_1$ and $\referee{g}$.
\end{lemma}
\begin{remark}\label{decay1}
As it was discussed at the beginning of Section 1.1.3, if we impose stricter conditions on the potential, for example, $Q$ being the sum of a smooth periodic potential and a potential from the Schwartz class, technicalities simplify. In particular, following the proof of the lemma above or the construction from~\cite[Section 6]{PaSh:16}, one can show that the statement of Lemma~\ref{l:layer0} holds for $|\theta|>0$, and thus, further onion peeling (see Lemma~\ref{l:nextLayer}) is not needed. This also leads to the significant simplification of the concluding arguments about actual asymptotics from Section~\ref{s:gauge} (see Remark~\ref{decay2}).
\end{remark}
\begin{proof}
Let $\cutThetaFreq \in C_c^\infty((-1,1);[0,1])$ with $\cutThetaFreq \equiv 1$ in a neighbourhood of $[-1/2,1/2]$, with $\check \cutThetaFreq$ real valued \referee{\st{and $\cutFreq\in C_c^\infty((0,\infty);[0,1])$ real valued with $\cutFreq\equiv 1$ in a neighbourhood of $[a,b]$,}} and let 
$$
\mathbf{q}_{\mc{H}}(x,\xi):= \mathbf{q}_{0,\mc{H}}(x,\xi),\qquad \mathbf{q}_{0,\mc{H}}:=(1-\cutThetaFreq(D_x))\mathbf{q}_0(\cdot,\xi).
$$
Observe that, since 
$$
\|\partial_x^\alpha \cutThetaFreq(D_x)\mathbf{q}_0(\cdot,\xi)\|_{L^\infty} \leq C_\alpha \|\mathbf{q}_0(\cdot,\xi)\|_{L^\infty},
$$
we have $\mathbf{q}_{\mc{H}}\in \Spz^{\fcomp}$.

By Lemma~\ref{l:commute1}, there is $g_0\in \Spz^{\fcomp}$ real valued with $\supp \hat{g}_0\subset\supp \hat{\mathbf{q}}_{\mc{H}}$, 
$$
\|\partial_x^\alpha \partial_\xi^\beta g_0\|_{L^\infty}\leq C \|\partial_x^\alpha \partial_\xi^\beta \mathbf{q}_{\mc{H}}\|_{L^\infty},
$$
and such that 
$$
i[\Op{\hbar}(g_0),-\hbar^{2}\Delta]=\hbar\Op{\hbar} (\mathbf{q}_{\mc{H}}).
$$
Now, by Lemma~\ref{l:conjugate},
$$
e^{-i\Op{\hbar}(g_0)}\semOp{\mathbf{Q}_0}e^{i\Op{\hbar}(g_0)}= -\hbar^{2}\Delta+\hbar(\Op{\hbar}(\mathbf{q}_0-\Op{\hbar}(\mathbf{q}_{\mc{H}})) +\hbar^{2}\Op{\hbar}(e_0) +O(\hbar^{\infty})_{H_\hbar^{-N}\to H_\hbar^N}
$$
with $e_0\in \Spz^{\fcomp}$ real valued. Now we proceed by induction. Suppose we have found $g_0,g_1,\dots g_N\in \Spz^{\fcomp}$ such that, for $G_N:=\sum_{j=0}^N \hbar^{j}g_j$, we have
$$
e^{-i\Op{\hbar}(G_N)}\semOp{\mathbf{Q}_0}e^{i\Op{\hbar}(G_N)}=  -\hbar^{2}\Delta +\hbar\Op{\hbar}(\mathbf{q}_{1,N}) +\hbar^{N+2}\Op{\hbar}(e_{N})+O(\hbar^{\infty})_{H_\hbar^{-N}\to H_\hbar^N}
$$
where $\mathbf{q}_{1,N}\in \Spz^1$, $e_N\in \Spz^{\fcomp}$ are real valued with $\supp \hat{\mathbf{q}}_{1,N} \cap \{|\theta|\geq 1\}=\emptyset$, \referee{$\supp \mathbf{q}_{1,N}\subset\{a\leq |\xi|\leq b\}$.}

Then, put $e_{N,\mc{H}}= (1-\cutThetaFreq(D_x))e_N$  so that 
$$
\|\partial_{\xi}^\beta \partial_x^\alpha e_{N,\mc{H}}\|\leq C_{\beta}\sum_{j=0}^{|\beta|}\|\partial_\xi^j\partial_x^\alpha e_N\|_{L^\infty}
$$
and let $g_{N+1}\in \Spz^{\fcomp}$ be real valued with $\supp \hat{g}_{N+1}\subset \supp \hat{e}_{N,\mc{H}}$ such that 
$$
\|\partial_x^\alpha \partial_\xi^\beta g_{N+1}\|_{L^\infty}\leq C \|\partial_x^\alpha \partial_\xi^\beta e_{N,\mc{H}}\|_{L^\infty}
$$
and
$$
i[\Op{\hbar}(g_{N+1}),-\hbar^{2}\Delta]=\hbar\Op{\hbar} (e_{N,\mc{H}}).
$$
Then, putting $G_{N+1}:=G_N+\hbar^{N+1}g_N,$ we have by Lemma~\ref{l:conjugate}
$$
e^{-i\Op{\hbar}(G_{N+1})}\semOp{\mathbf{Q}_0}e^{i\Op{\hbar}(G_{N+1})}=-\hbar^{2}\Delta +\hbar\Op{\hbar}(\mathbf{q}_{1,N+1}) +\hbar^{N+3}\Op{\hbar}(e_{N+1})+O(\hbar^{\infty})_{H_\hbar^{-N}\to H_\hbar^N}
$$
with $e_{N+1}\in \Spz^{\comp}$, $\mathbf{q}_{1,N+1}\in \Spz^1$ real valued, $\supp \hat {\mathbf{q}}_{1,N+1}\cap\{|\theta|\geq 1\}=\emptyset$, \referee{and $\supp \mathbf{q}_{1,N+1}\subset \{a\leq |\xi|\leq b\}$. }

In particular, putting $g\sim \sum_{j=0}^\infty \hbar^{j}g_j$ (\referee{see Lemma~\ref{l:borelSum}}) completes the proof of the lemma.
\end{proof}

Next, we show how to peel off layers of the form $\mu_n^{-k\delta}\leq |\theta|\leq \mu_n^{-(k-1)\delta}$ from the Fourier transform of the pseudodifferential potential.

\begin{lemma}[layer peeling lemma]
\label{l:nextLayer}
Let $0<a<b$, $\referee{N_0}>0$, $0<\delta<1$, and $\hbar^{\referee{N_0}}\leq \referee{r=r(\mu_n)}\leq 1$. Suppose \referee{that for any $N>0$},
$$
\sem H:=-\hbar^2 \Delta +\hbar \Op{\hbar}(\mathbf{q})+O(\hbar^{\infty})_{H_{\hbar }^{\referee{-N}}\to H_{\hbar}^{\referee{N}}}
$$
for some real valued $\mathbf{q}\in \Sp^1$ satisfying 
\begin{equation}
\label{e:terrifiesLeonid}
\supp \hat{\mathbf{q}}\cap \{|\theta| \geq r\}=\emptyset,\qquad\referee{\supp \hat{\mathbf{q}}\subset \{a\leq |\xi|\leq b\}}.
\end{equation}
Then there is a real valued $g\in r^{-1}\referee{\lp_n^{\delta}} S_{r,\delta}^{\fcomp}$ supported in $\{a\leq |\xi|\leq b\}$ such that \referee{for any $N$}
$$
e^{-iG}\sem He^{iG}= -\hbar^2\Delta +\hbar \Op{\hbar}(\mathbf{q}_1)+O(\hbar^{\infty})_{H_\hbar^{-N}\to H_{\hbar}^N},\qquad G:=\Op{\hbar}(g)
$$
for some $\mathbf{q}_1\in \Sp^1$ satisfying 
\begin{equation}
\label{e:freightensLeonid}
\supp \hat{\mathbf{q}}_1\cap  \{|\theta| \geq r\referee{\lp_n^{-\delta}}\}=\emptyset,\qquad \referee{\supp\mathbf{q}_1\subset \{a\leq |\xi|\leq b\}}.
\end{equation}
In addition, $\mathbf{q}_1$, $g$ depend continuously on $\mathbf{q}$ and, if $\mathbf{q}$ is $L$-periodic, then so are $\mathbf{q}_1$ and $g$.
\end{lemma}
\begin{proof}
The proof is similar to that of Lemma~\ref{l:layer0} except that we must keep more careful track of derivatives of the various $g$'s. 

We first let $\cutThetaFreq \in C_c^\infty((-1,1);[0,1])$ with $\cutThetaFreq \equiv 1$ in a neighbourhood of $[-1/2,1/2]$ and $\check \cutThetaFreq$ real valued.


Now we find $g_0\in r^{-1}\referee{\lp_n^{\delta}} S^{\fcomp}_{r.\delta}$, $\mathbf{q}_{1,0}\in \Sp^1$, and $e_0\in S_{r,\delta}^{\fcomp}$ real valued such that,  $\supp g_0\subset \referee{\{a\leq |\xi|\leq b\}}$, 
 $\supp \hat{\mathbf{q}}_{1,0}\cap \{|\theta| \geq r\referee{\lp_n^{-\delta}}\}=\emptyset$, \referee{$\supp \mathbf{q}_{1,0}\subset \{a\leq |\xi|\leq b\}$,} and 
 \begin{equation}
 \label{e:inductForm0}
e^{-i\Op{\hbar}(g_0)}\sem He^{i\Op{\hbar}(G_0)}=  -\hbar^{2}\Delta +\hbar\Op{\hbar}(\mathbf{q}_{1,0}) +\hbar^{2}\referee{\lp_n^{\delta}}\Op{\hbar}(e_{0})+O(\hbar^{\infty})_{H_\hbar^{-N}\to H_\hbar^N}.
\end{equation}

Put $\mathbf{q}_{\mc{H}}:=[(1-\cutThetaFreq( r^{-1}\referee{\lp_n^{\delta}}D_x))\mathbf{q}](x,\xi)$ so that $\mathbf{q}_{\mc{H}}$ is real valued and
$$
\|\partial_{\xi}^\beta \partial_x^\alpha \mathbf{q}_{\mc{H}}(\cdot,\xi)\|_{L^\infty}\leq C_{\alpha \beta}\sum_{j=0}^{|\beta|}\sup_{|\xi|\in \supp \cutFreq_{-1}}\|\partial_x^\alpha \partial_\xi^j \mathbf{q}(\cdot,\xi)\|_{L^\infty}\leq C_{\alpha \beta} r^{|\alpha|}.
$$
In particular, $\mathbf{q}_{\mc{H}}\in S_{r,\delta}^{\fcomp}$.
We then use Lemma~\ref{l:commute1} to find $g_0\in r^{-1}\referee{\lp_n^{\delta}}S_{r,\delta}^{\fcomp}$ supported in $\referee{\{a\leq |\xi|\leq b\}}$, real valued, satisfying
$$
\|\partial_\xi^\beta\partial_x^\alpha g_0\|_{L^\infty} \leq Cr^{-1}\referee{\lp_n^{\delta}}\|\partial_{\xi}^\beta \partial_x^\alpha \mathbf{q}_{\mc{H}}\|_{L^\infty}\leq C_{\alpha \beta}r^{-1+|\alpha|}\referee{\lp_n^{\delta}}
$$
and 
$$
i[\Op{\hbar} (g_0),-\hbar^{2}\Delta]=\hbar\Op{\hbar}(\mathbf{q}_{\mc{H}}).
$$
Thus, by Lemma~\ref{l:conjugate},
$$
e^{-i\Op{\hbar}(g_0)}\sem He^{i\Op{\hbar}(G_0)}=  
-\hbar^{2}\Delta +\hbar\Op{\hbar}(\mathbf{q}-\mathbf{q}_{\mc{H}}) +\hbar^{2}\referee{\lp_n^{\delta}}\Op{\hbar}(e_{0})+O(\hbar^{\infty})_{H_\hbar^{-N}\to H_\hbar^N},
$$
with $e_0\in S^{\fcomp}_{r,\delta}$. Since $\supp (\widehat{\mathbf{q}}-\widehat{\mathbf{q}_{\mc{H}}})\cap  \{(\theta,\xi)\,:\, |\theta| \geq r\referee{\lp_n^{-\delta}}\}=\emptyset$, we may put $\mathbf{q}_{1,0}=\mathbf{q}-\mathbf{q}_{\mc{H}}$ to obtain~\eqref{e:inductForm0}.

We again proceed by induction. Let $N\geq 0$ and suppose we have found $g_{0},\dots, g_N\in r^{-1}\referee{\lp_n^{\delta}} S_{r,\delta}^{\fcomp}$ real valued, supported in $\referee{\{a\leq |\xi|\leq  b\}}$ such that with $G_N=\sum_{j=0}^N \hbar^{j}g_j$, we have
$$
e^{-i\Op{\hbar}(G_N)}\sem He^{i\Op{\hbar}(G_N)}=  -\hbar^{2}\Delta +\hbar\Op{\hbar}(\mathbf{q}_{1,N}) +\hbar^{1+(N+1)}\referee{\lp_n^{\delta(N+1)}}\Op{\hbar}(e_{N})+O(\hbar^{\infty})_{H_\hbar^{-N}\to H_\hbar^N},
$$
where $\mathbf{q}_{1,N}\in \Sp^1$, $e_N\in S_{r,\delta}^{\fcomp}$ are real valued, $\supp \hat{\mathbf{q}}_{1,N}\cap \{(\theta,\xi)\,:\,|\theta| \geq r\referee{\lp_n^{-\delta}}\}=\emptyset$, \referee{and $\supp \mathbf{q}_{1,N}\subset \{a\leq |\xi|\leq b\}$.}

Then, put $e_{N,\mc{H}}= [(1-\cutThetaFreq( r^{-1}\referee{\lp_n^{\delta}}D_x))e_N(\cdot,\xi)](x)$ so that $b_{N,\mc{H}}$ is real valued and
$$
\|\partial_{\xi}^\beta \partial_x^\alpha e_{N,\mc{H}}\|\leq C_{\alpha \beta}\sum_{j=0}^{|\beta|}\|\partial_x^\alpha \partial_\xi^j e_N\|_{L^\infty}\leq C_{\alpha \beta} r^{|\alpha|}.
$$
In particular, $e_{N,\mc{H}}\in S_{r,\delta}^{\fcomp}$.

Now, by Lemma~\ref{l:integrate}, there is $g_{N+1}\in r^{-1}\referee{\lp_n^{\delta}}S_{r,\delta}^{\fcomp}$ real valued with $\supp \hat{g}_{N+1}\subset \supp \hat{b}_{N,\mc{H}}\cup\{\theta=0\}$ such that 
$$
\|\partial_x^\alpha \partial_\xi^\beta g_{N+1}\|_{L^\infty}\leq C_\alpha r^{-1}\referee{\lp_n^{\delta}} \|\partial_x^\alpha\partial_\xi^\beta  e_{N,\mc{H}}\|_{L^\infty}\leq C_{\alpha\beta}  r^{-1+|\alpha|}\referee{\lp_n^{\delta}}$$
and such that 
$$
i[\Op{\hbar}(g_{N+1}),-\hbar^{2}\Delta]=\hbar\Op{\hbar} (e_{N,\mc{H}}).
$$
Then, by Lemma~\ref{l:conjugate}, with $G_{N+1}=G_N+\hbar^{N+1}\referee{\lp_n^{\delta(N+1)}}g_{N+1}$, we have
\begin{align*}
&e^{-i\Op{\hbar}(G_{N+1})}\sem H e^{i\Op{\hbar}(G_{N+1})}\\
&=  -\hbar^{2}\Delta +\hbar\Op{\hbar}(\mathbf{q}_{1,N}) +\hbar^{1+(N+1)}\referee{\lp_n^{\delta(N+1)}}\Op{\hbar}(e_{N}-e_{N,\mc{H}})\\
&\qquad+\hbar^{1+(N+2)}\referee{\lp_n^{\delta(N+2)}}\Op{\hbar}(\tilde{e}_{N+1})+
O(\hbar^{\infty})_{H_\hbar^{-N}\to H_\hbar^N},
\end{align*}
where $\tilde{e}_N\in S_{r,\delta}^{\fcomp}$.
Since $\supp( \widehat{e_N}-\widehat{e_{N,\mc{H}}})\cap \{ (\theta,\xi)\,:\, |\theta |\geq r\referee{\lp_n^{-\delta}}\}=\emptyset,$ we may define 
$$
\mathbf{q}_{1,N+1}=\mathbf{q}_{1,N}+\hbar^{N+1}\referee{\lp_n^{\delta(N+1)}}(e_N-e_{N,\mc{H}})
$$
in order to have the required properties for $\mathbf{q}_{1,N+1}$.

We can now put $g\sim \sum_j \hbar^{j}\referee{\lp_n^{j\delta}}g_j$ \referee{(see Lemma~\ref{l:borelSum})} to finish the proof of the lemma.

\end{proof}

\begin{remark}
The proofs of Lemmas~\ref{l:layer0} and~\ref{l:nextLayer} may look as though they require performing infinitely many parallel gauge transform steps; something that experts in the gauge transform could be concerned about. However, the proofs actually rely on being able to make a finite but arbitrarily large number of such steps. Morally, we do not make the sets on left hand sides of~\eqref{e:scaresLeonid},~\eqref{e:terrifiesLeonid}, and~\eqref{e:freightensLeonid} empty, instead making the corresponding part of $q$ smaller than $\hbar^N$ for \referee{some} arbitrarily large $N$. \referee{We then apply the Borel summation lemma (Lemma~\ref{l:borelSum}).} \referee{Recall also that $n=n(\hbar)$ and satisfies~\eqref{e:muNBound} and hence, since the remainders are controlled in $\hbar$, they are controlled in $n$.}
\end{remark}

The final proposition of this section shows that, using a serial gauge transform, one can remove frequencies which are larger than any fixed power of $\hbar$ from the potential. 
\begin{proposition}
\label{p:gauge}
Suppose that
$$
\semOp{\mathbf{Q}_0}=-\hbar^2\Delta +\hbar\Op{\hbar}(\mathbf{q}_0)
$$
for some \referee{$\mathbf{q}_0\in \Spz^1$} real valued. Let  $0<a<b$ \referee{such that $\supp \mathbf{q}_0\subset\{a\leq |\xi|\leq b\}$}. $0< \delta <1$, $M>0$ \referee{\st{and suppose $K^{-1}\lp_{n}^{-\delta}\leq \step\leq K\lp_{n}^{-\delta}$}}. Put $r_{-1}=1$, $r_j=\referee{\lp_n^{-\delta j}}$, $j=0,1,\dots$. Then there are $g_{-1}\in \Spz^{\fcomp}$, and $g_{j}\in r_j^{-1}\referee{\lp_n^{\delta}}S^{\fcomp}_{r_j,\delta}$, $j=0,1,\dots, M$ real valued such that such that  for all $\referee{N}$,
\begin{gather*}
U_{M}^*\dots U_0^*U_{-1}^*\semOp{\mathbf{Q}_0}U_{-1} U_0\dots U_{M}= -\hbar^2\Delta +\hbar \Op{\hbar}(\mathbf{q}_1)+O(\hbar^{\infty})_{H_\hbar^{-\referee{N}}\to H_{\hbar}^{\referee{N}}},\\
U_j :=e^{i\Op{\hbar} (g_j)},
\end{gather*}
and 
$$
\supp \hat{\mathbf{q}}_1\cap \{ |\theta|\geq \referee{\lp_n^{-M\delta}}\}=\emptyset,\qquad \referee{\supp \mathbf{q}_1\subset \{a\leq |\xi|\leq b\}.}
$$
Moreover, $g_j$ and $\mathbf{q}_1$ depend continuously on \referee{$\mathbf{q}_0$ and, if $\mathbf{q}_0$} is $L\mathbb{Z}$-periodic, then so are $\mathbf{q}_1$ and $g_j$.
\end{proposition}
\begin{remark}
\referee{In order to remove all of the frequencies of $\sem{q}_0$ larger than $\lp_n^{-N}$, we take $M=\lceil N\delta^{-1}\rceil$.}
\end{remark}
\begin{remark}
\referee{We prove Proposition~\ref{p:gauge} using a serial sequence of parallel gauge transforms (see Remark~\ref{r:serialParallel}). Indeed, notice that the proof of Lemma~\ref{l:nextLayer} involves a parallel gauge transform which we apply a large, independent of $\hbar$ number of times.}
\end{remark}
\begin{proof}
By Lemma~\ref{l:layer0}, there is $g_{-1}\in \Spz^{\fcomp}$ such that 
$$
\semOp{\mathbf{Q}_{\referee{[1]}}}:=U_{-1}^*\semOp{\mathbf{Q}_0}U_{-1}=-\hbar^2\Delta+\hbar\Op{\hbar}(\mathbf{q}_{\referee{[1]}})+O(\hbar^{\infty})_{H_{\hbar}^{-\referee{N}}\to H_{\hbar}^{\referee{N}}}
$$
with $\mathbf{q}_{\referee{[1]}}\in \Spz^1$ real valued and satisfying
$$
\supp \hat{\mathbf{q}}_{\referee{[1]}}(\theta,\xi)\cap \{|\theta| \geq 1\}=\emptyset,\qquad \referee{\supp \mathbf{q}_{\referee{[1]}}\subset \{ a\leq |\xi|\leq b\}.}
$$
Setting $\mathbf{q}$ to $\mathbf{q}_{\referee{[1]}}$ in Lemma~\ref{l:nextLayer}, we find $g_0\in \referee{\lp_n^{\delta}} \Sp^{\fcomp}$ such that 
$$
\sem H_{\referee{[2]}}:=U_{1}^*\semOp{\mathbf{Q}_{\referee{[1]}}}U_1=-\hbar^2\Delta+\hbar\Op{\hbar}(\mathbf{q}_{\referee{[2]}})+O(\hbar^{\infty})_{H_{\hbar}^{-\referee{N}}\to H_{\hbar}^{\referee{N}}}, 
$$
with $\mathbf{q}_{\referee{[2]}}\in \Sp^1$ real valued and satisfying
$$
\supp \hat{\mathbf{q}}_{\referee{[2]}}(\theta,\xi)\cap \{ |\theta| \geq r_1=\referee{\referee{\lp_n^{-\delta}}}\}=\emptyset,\qquad \referee{\supp \mathbf{q}_{\referee{[2]}}\subset\{a\leq |\xi|\leq b\}.}
$$
Iterating this process $M$ times completes the proof of the proposition.
\end{proof}

\section{Computing the local density of states}
\label{s:compute}
Before we apply the gauge transform procedure from the previous section, it will be crucial to replace $\semOp{\mathbf{Q}_0}$ by a periodic operator.  \referee{Let }$\referee{N}\in\mathbb{R}$. We aim to compute the local density of states modulo errors of size $\referee{O}(\hbar^{\referee{N}})$ \referee{(or, equivalently, $\referee{O(}\lp_n^{-N})$)}. 

As stated in Section~\ref{s:muN}, we fix a discrete sequence $\{\lp_n\referee{=2^n}\}_{n=1}^\infty$ and work with 
$$
\hbar\in [2^{-10}\lp^{-1}_n,2^{10}\lp^{-1}_{n}].
$$ 
In order to compute the local density of states,  we start by replacing $\semOp{\mathbf{Q}_0}$ by a periodic operator, $\semOp{\referee{{}^P\!\mathbf{Q}_{0}}}$ with period at scale $\lp_n^\referee{N}$. The local densities of states for the two operators are close by Proposition~\ref{p:distantPerturb}. We then study the local density of states for $\semOp{\referee{{}^P\!\mathbf{Q}_{0}}}$ by applying the gauge transform from Proposition~\ref{p:gauge}. This will result in an operator which acts as a Fourier multiplier for semiclassical energies $\omega\in [a,b]$. As we will see \referee{in Corollary~\ref{c:transformedProjector}}, computing the local density of states for such operators is relatively straightforward. Finally, in order to complete the proof of the main theorem, we will need to understand how the unitary operator found in Proposition~\ref{p:gauge} acts on delta functions. In some sense, this corresponds to  `unpeeling' (or rebuilding) the onion peeled by the gauge transform.

\subsection{Periodising the \referee{perturbation}}
\label{s:periodisePotential}

We now periodise the \referee{perturbation} in a way that will \referee{have a negligible} effect on the local density of states. Let $\cutMan \in C_c^\infty((-\tfrac{1}{2},\tfrac{1}{2});[0,1])$ with $\cutMan \equiv 1$ on $[-\tfrac{1}{4},\tfrac{1}{4}]$, and put $\cutMan_n(x):=\cutMan(\lp_n^{-\referee{N}}x)$. 
Suppose that 
$$
\mathbf{Q}_0:=\referee{\mathbf{V}}^{1}(x)\hbar D_x+\hbar D_x \referee{\mathbf{V}}^1(x)+ \referee{\mathbf{V}}^0(x)\in \sem{Diff}^1.
$$

Then, put 
$$
\referee{\mathbf{V}}^j_{n}(x):=\sum_{k\in \mathbb{Z}} \cutMan_n(x-k\lp_n^\referee{N})\referee{\mathbf{V}}^j(x-k\lp_n^\referee{N}),\qquad j=0,1
$$
so that $\referee{\mathbf{V}}^j_{n}(x)=\referee{\mathbf{V}^j}$ on $|x|\leq \tfrac{1}{4}\lp_n^\referee{N}$, and
$$
|\partial_{x}^\alpha \referee{\mathbf{V}^{j}_n}(x)|\leq C_\alpha,\qquad x\in \mathbb{R}.
$$
Define
$$
\referee{{}^P\!\mathbf{Q}_{0}}:=\referee{\mathbf{V}}^1_n(x)\hbar D_x+\hbar D_x \referee{\mathbf{V}}^1_n(x)+ \referee{\mathbf{V}}_n^0.
$$
\referee{Here, we use the notation ${}^P$ to remind the reader that ${}^P\!\mathbf{Q}_{0}$ is the periodised version of $\mathbf{Q}_0$ (see also Example~\ref{example:1} part (3)).}

\referee{We claim that  for $\omega\in [a-\e,b+\e]$ and $\lambda\in [-\e,\e]$, $\semOp{\referee{{}^P\!\mathbf{Q}_{0}}}$ satisfies
\begin{equation}
\label{e:periodisedIsLipschitz}
|\sem E(\semOp{\referee{{}^P\!\mathbf{Q}_{0}}})(x,y,\omega)-\sem E(\semOp{\referee{{}^P\!\mathbf{Q}_{0}}})(x,y,\omega +\lambda)|\leq C\lp_n^{-\referee{N}}\langle \lp_n^{\referee{N}+1}\lambda\rangle.
\end{equation}
Once we prove~\eqref{e:periodisedIsLipschitz}, Proposition~\ref{p:distantPerturb} \referee{applied} with $T=\tfrac{1}{8}\lp_n^{\referee{N}}$ will show that 
\begin{equation}
\label{e:comparison}
|\sem E(\semOp{\mathbf{Q}_0})(x,y,\omega)-\sem E(\semOp{\referee{{}^P\!\mathbf{Q}_{0}}})(x,y,\omega)|\leq C\lp_n^{-\referee{N}},\qquad \omega\in[a,b].
\end{equation}}

It therefore remains only to compute $\sem E(\semOp{\referee{{}^P\!\mathbf{Q}_{0}}})(x,y,\omega)$ \referee{and prove~\eqref{e:periodisedIsLipschitz}}.


\subsection{Analysis of $\sem E(\semOp{\referee{{}^P\!\mathbf{Q}_{0}}})$: reduction to a Fourier multiplier}
\label{s:fourierMultiplier}

We first fix $\referee{\delta'\in(0,1]}$ and work either on diagonal or \referee{assume} $|x-y|\geq ch^{1-\delta'}$. Then, let  $a,b,\delta \in \mathbb{R}$ such that $0<a<b$, $0<\delta<\min(\frac{1}{2},\frac{\referee{\delta'}
}{2})$.
The goal of this section is to show that~\eqref{e:periodisedIsLipschitz} holds \referee{with constants depending on all the parameters introduced above but uniformly over $\hbar\in (0,1]$} and to compute an asymptotic formula for $\sem{E}(\semOp{\referee{{}^P\!\mathbf{Q}_{0}}})$.

\referee{\begin{remark}
\label{r:deltaIsSmall}
Observe that it is only necessary to work with $\delta\ll1$ in order to obtain asymptotics very close to, but not on, the diagonal. Indeed, the requirement $\delta<\frac{\delta'}{2}$ is the only reason we cannot simply fix $\delta$ from the outset. Consequently, the reader only interested in on-diagonal asymptotics may work with $\delta=\frac{1}{4}$ for example. (See also Remark~\ref{r:deltaNonzero}.)
\end{remark}
}

\referee{We now reduce to the case where $\referee{{}^P\!\mathbf{Q}_{0}}$ is supported in $a< |\xi|<b$. We use Lemma~\ref{l:abstralSpectralNew} to prove the following lemma.} \referee{Since} we expect this lemma to be useful in future work, we prove it in arbitrary dimension. 

\referee{
\begin{lemma}
\label{l:transformedProjector}
Let $0<a<b$. Suppose that $\mathbf{q}_1,\mathbf{q}_2\in S^1(T^*\mathbb{R}^d)$ are real valued and for all $a<|\xi|<b$, $x\in \mathbb{R}^d$ we have
$$
\mathbf{q}_1(x,\xi)=\mathbf{q}_2(x,\xi).
$$ 
Put $\mathbf{Q}_j:=\Op{\hbar}(\mathbf{q}_j):H_h^1(\mathbb{R}^d)\to L^2(\mathbb{R}^d)$. For all $N>0$ and $\e>0$ there are $C>0$, $L>0$ such that if for all $x\in \mathscr{K}$ and all $\omega\in[a,b]$ we have
 $$
\sum_{|\alpha|\leq N}\|\sem E(\semOp{\mathbf{Q}_2};[\omega^2-\hbar^L,\omega^2+\hbar^L])\partial^\alpha\delta_x\|_{L^2}\leq C^{-1}\hbar^N,
$$
then for all $\omega \in (a+\e,b-\e)$ we have
\begin{equation}
\label{e:forLeonid}
\Big\|\sem E(\semOp{\mathbf{Q}_1})(\cdot,\cdot,\omega)-\sem E(\semOp{\mathbf{Q}_2})(\cdot,\cdot,\omega)\Big\|_{C^N(\mathscr{K}\times \mathscr{K})}\leq C\hbar^N.
\end{equation}
\end{lemma}}

\referee{\begin{proof}

We will apply Lemma~\ref{l:abstralSpectralNew} with $J=((a+\e)^2,(b-\e)^2)$ and
$$
H_1:=\semOp{\mathbf{Q}_1},\qquad H_2:=\semOp{\mathbf{Q}_2}.
$$
Then, let $\chi \in C_c^\infty(a^2,b^2)$ with $\chi \equiv 1$ on $((a+\e)^2,(b-\e)^2)$ so that 
$$
(H_1-H_2)E(H_2;J)=(H_1-H_2)\chi(H_2)E(H_2;J)=O(\hbar^\infty)_{\Psi^{-\infty}}.
$$
Similarly,
$$
E(H_1;J)(H_1-H_2)=E(H_1;J)\chi(H_1)(H_1-H_2)=O(\hbar^\infty)_{\Psi^{-\infty}}.
$$
In particular, the hypotheses of Lemma~\ref{l:abstralSpectralNew} hold with $\e_1=\e_2=\e_3=O(\hbar^\infty)$ for any $s\in \mathbb{R}$.

\referee{In order to apply \referee{Lemma~\ref{l:abstralSpectralNew}}, we estimate
$$
\|\sem E(\semOp{\mathbf{Q}_2})(\omega)\partial_x^\alpha\delta_x\|_{L^2}.
$$
To do this, observe that for any $s\in \mathbb{R}$
$$
\|\sem E(\semOp{\mathbf{Q}_2})(\omega)(\semOp{\mathbf{Q}_2}+1)^s\|_{L^2\to L^2}\leq C(\omega^2+1)^s
$$
and the principal symbol $\sigma(\semOp{\mathbf{Q}_2}+1)=|\xi|^2+1$ is non-vanishing. In particular, $(\semOp{\mathbf{Q}_2}+1)^{-s}\in \Psi^{-2s}$ exists. Therefore, 
\begin{equation}
\label{e:smoothness}
\begin{aligned}
\|\sem E(\semOp{\mathbf{Q}_2})(\omega)\partial_x^\alpha\delta_x\|_{L^2}&\leq
\|\sem E(\semOp{\mathbf{Q}_2})(\omega)(\semOp{\mathbf{Q}_2}+1)^{s}\|_{L^2\to L^2} \|(\semOp{\mathbf{Q}_2}+1)^{-s}\partial_x^\alpha\delta_x\|_{L^2}\\
&\leq C(\omega+1)^s\|\partial_x^\alpha\delta_x\|_{H_{\hbar}^{-2s}}\leq C(\omega+1)^s\hbar^{-|\alpha|-\frac{d}{2}}
\end{aligned}
\end{equation}
for any $s>\frac{d}{4}+|\alpha|$. 
}
Thus, by \referee{Lemma~\ref{l:abstralSpectralNew}}

\begin{equation}
\label{e:comparisonNow}
\sem E(\semOp{\mathbf{Q}_1})(\omega;x,y)=\sem E(\semOp{\mathbf{Q}_2})(\omega;x,y)+O(\hbar^{\infty})_{C^\infty}.
\end{equation}

\begin{remark}
Above, we apply the statement in \referee{Lemma~\ref{l:abstralSpectralNew}} for each $\hbar$ to obtain~\eqref{e:comparisonNow}.
\end{remark}

\end{proof}}

\referee{Lemma~\ref{l:transformedProjector} has the following useful corollary.}
\begin{corollary}
\label{c:transformedProjector}
Suppose that $\mathbf{q}\in S^1(T^*\mathbb{R}^d)$ is real valued and for all $a<|\xi|<b$, $x\in \mathbb{R}^d$ we have
$$
\mathbf{q}(x,\xi)=\tilde{\mathbf{q}}(\xi).
$$ 
Put $\mathbf{Q}:=\Op{\hbar}(\mathbf{q}):H_h^1(\mathbb{R}^d)\to L^2(\mathbb{R}^d)$. Then, for all \referee{$\referee{\mathcal{K}}\subset \mathbb{R}^d\times \mathbb{R}^d$ compact}, $N>0$ and $\e>0$ there is $C_N>0$ such that for all $\omega \in (a+\e,b-\e)$ we have
\begin{equation}
\label{e:forLeonid}
\Big\|\sem E(\semOp{\mathbf{Q}})(x,y,\omega)-\frac{1}{(2\pi \hbar)^{d}}\int_{G_\hbar(\omega)} e^{i \langle x-y,\xi\rangle/ \hbar}d\xi\Big\|_{C^N\referee{(\referee{\mathcal{K}})}}\leq C_N\hbar^N,
\end{equation}
where 
$$
G_\hbar(\omega):=\{ \xi \mid |\xi|^2+\hbar \tilde{\mathbf{q}}(\xi)\leq \omega^2\}.
$$
\end{corollary}
\begin{remark}
The fact that the right-hand size of~\eqref{e:forLeonid} is non-zero, albeit small, is due to our use of the Weyl quantisation rather than the left quantisation. See \referee{e.g.~\cite[(7.17)]{PaSh:16} and~\cite[(6.12)]{PaSh:12}} for the equality in the left quantisation. For the on the diagonal case, similar lemmas also appear in~\cite{Sh:79}.
\end{remark}

\referee{By Lemma~\ref{l:transformedProjector}, for $a<\omega<b$, we have 
$$
\mathbf{E}(\semOp{\referee{{}^P\!\mathbf{Q}_{0}}})(\omega;x,y)=\mathbf{E}(\semOp{\referee{{}^P\!\tilde{\mathbf{Q}}_{0}}})(\omega;x,y)+O(h^\infty)_{C^\infty}, 
$$
where 
$$
\referee{{}^P\!\tilde{\mathbf{q}}_{0}}(x,\xi)=\chi(|\xi|)\referee{{}^P\!\mathbf{q}_{0}}(x,\xi),
$$
with $\chi \in C_c^\infty(\mathbb{R}_+)$ and $\chi \equiv 1$ on $[a,b]$.  }

\referee{By Proposition~\ref{p:gauge}, with $M=\lceil N\delta^{-1}\rceil$, there is a unitary operator, $U=U_n$, and $\mathbf{q}_1\in S^1$ real valued such that 
$$
\mathbf{H}_1:=U^*\semOp{\referee{{}^P\!\tilde{\mathbf{Q}}_{0}}}U=-\hbar^2 \Delta+\hbar \Op{\hbar}(\mathbf{q}_1)\referee{+O(\hbar^\infty)_{\Psi^{-\infty}}},
$$
where $\mathbf{q}_1\in S^1$ is $\lp_n^\referee{N}$-periodic and
\begin{equation}
\label{e:fourierMe}
\supp\hat{\mathbf{q}}_1\cap \{|\theta|\geq \lp_n^{-\referee{N}}\}=\emptyset.
\end{equation}

Now, since $\mathbf{q}_1$ is $\lp_n^\referee{N}$-periodic, 
$$
\supp\hat{\mathbf{q}}_1\cap \{|\theta|\leq \lp_n^{-\referee{N}}\}\subset \{ \theta=0\}.
$$
In particular,
\begin{equation}
\label{e:Q2}
\mathbf{q}_2(\xi):=\mathbf{q}_1(x,\xi)\in \Sp^{\fcomp}
\end{equation}
is independent of $x$.

Put
$$
\semOp{\mathbf{Q}_1}:=-\hbar^2 \Delta+\hbar \Op{\hbar}(\mathbf{q}_1),\qquad \tilde{\sem H}_1:= U \semOp{\mathbf{Q}_1} U^*.
$$
Then, \referee{Lemma~\ref{l:abstralSpectral}} implies
\begin{equation}
\label{e:voodoo}
\partial_x^\beta\partial_y^\alpha(\sem E(\tilde{\sem{H}}_1)(x,y,\omega)-\sem E(\semOp{\referee{{}^P\!\mathbf{Q}_{0}}})(x,y,\omega))=O(\hbar^{\infty}).
\end{equation}
}
\begin{remark}
Note that we apply \referee{Lemma~\ref{l:abstralSpectral}} to the derivatives of the delta function and \referee{use} the fact that for $\cutEnergy\in C_c^\infty$,
$$
\|\cutEnergy(\semOp{\mathbf{\mathbf{Q}}})\partial^\alpha_x\delta\|_{L^2}\leq C_\alpha\hbar^{-|\alpha|-\frac{\referee{1}}{2}}.
$$
\end{remark}

We now focus on computing 
\begin{multline*}
(\partial_x^\beta\partial_y^\alpha)\sem E(\tilde{\sem H}_1)(x_0,y_0,\omega)=(\partial_x^\beta\partial_y^\alpha)\sem E(U\semOp{\mathbf{Q}_1}U^*)(x_0,y_0,\omega)\\=\langle U\sem E(\semOp{\mathbf{Q}_1})(\omega)U^*(-\partial_y)^\alpha \delta_{y_0},(-\partial_x)^\beta\delta_{x_0}\rangle.
\end{multline*}
This will be a priori simpler than computing $\sem{E}(\semOp{\referee{{}^P\!\mathbf{Q}_{0}}})$ since $\mathbf{Q}_1$ \referee{is} a Fourier multiplier \referee{and hence we have an exact formular for $\sem{E}(\semOp{\mathbf{Q}_1})$.}

\begin{remark}
We have replaced $(x,y)$ in the statement of our theorems by $(x_0,y_0)$ to avoid notational clashes in the next section.
\end{remark}

\subsection{Asymptotics of the spectral function: `unpeeling' the onion}
\label{s:unpeel}

Before we can understand the asymptotics of the spectral function, we need a lemma which gives the kernel of the spectral projector for $\semOp{\mathbf{Q}_2}$.

\begin{remark}\label{decay2}
In the case when the potential is the sum of a smooth periodic function and a function from the Schwartz class, the onion peeling is not needed (see Remark~\ref{decay1}). In particular, the gauge transform is made by a single operator $U=e^{i\Op{\hbar}(g_{-1})}$ with $g_{-1}\in \Spz^{\fcomp}$ which allows us to proceed immediately to the conclusion of Lemma~\ref{l:asymptotics1} below, and thus to complete the proof of the main result. In the general setting though, one has to deal with $U$ described by \eqref{subseqform} and additional technical arguments due to onion peeling and specifics of the corresponding classes $\referee{\lp_n^{\delta}}r_j^{-1}S^{\fcomp}_{r_j,\delta}$, $r_j=\referee{\lp_n^{-j\delta}}$.
\end{remark}

Now that we have computed the kernel of $\sem E(\semOp{\mathbf{Q}_1})$, we need to handle the action of $U$ and $U^*$ on $\sem E(\semOp{\mathbf{Q}_1})$. To do this, we first describe how  $\sem E(\semOp{\mathbf{Q}_1})$ moves wavefront sets.
\begin{lemma}
\label{l:frequencyPreserved}
Let $b>0$ $C>0$ and $\hbar^{C}\leq r\leq 1$. Then for all $\chi\in C_c^\infty(\mathbb{R})$, \referee{and all $\hbar$-tempered $u$} we have
\begin{equation}
\label{e:panda}
\WF_{\hbar r}(U_{r} \sem E(\semOp{\mathbf{Q}_1})(\omega)U_{r}^*\chi u)\subset \{ (x,\xi)\mid \xi\in \pi_{\xi}(\WF_{\hbar r}(u)),\,|\xi|\leq \omega\},
\end{equation}
where $\pi_{\xi}(x,\xi)=\xi$ is the natural projection. Moreover, for $\cutFreq\in C_c^\infty(\mathbb{R})$ with $\cutFreq\equiv 1$ on $[-b,b]$, and all $\omega\in (-b,b)$ we have
$$
U_{r} \sem E(\semOp{\mathbf{Q}_1})(\omega)U_{r}^*\Op{\hbar r}(1-\cutFreq(\xi))=0.
$$
\end{lemma}
\begin{proof}
First, recall that $U_{r} \sem E(\semOp{\mathbf{Q}_1})(\omega)U_{r}^*$ is given by
$$
U_{r} \sem E(\semOp{\mathbf{Q}_1})(\omega)U_{r}^*(x,y)=\frac{1}{2\pi \hbar r}\int _{G_{\hbar}(\omega)}e^{\frac{i}{ \hbar r}(x-y)\xi}d\xi.
$$

Let $\cutFreq \in C_c^\infty$ with $\cutFreq \equiv 1$ on $[-b,b]$, then
$$
U_{r} \sem E(\semOp{\mathbf{Q}_1})(\omega)U_{r}^*\Op{\hbar r}(1-\cutFreq(\xi))=\frac{1}{2\pi \hbar r}\int \int _{G_{\hbar}(\omega)}e^{\frac{i}{ \hbar r}[(x-z)\xi+(z-y)\eta]}(1-\cutFreq(|\eta|))  d\xi dz d\eta=0.
$$
Therefore, we may replace $\chi u$ by $\Op{\hbar r}(\cutFreq(\xi))\chi u$ \referee{in~\eqref{e:panda}} and hence assume $u$ is compactly microlocalized.

Suppose that $\xi_0\notin \pi_\xi(\WF_{\hbar r}(u))$. Then, since $u$ is compactly microlocalized, there is $\tilde{\cutPhase} \in C_c^\infty(\mathbb{R}^2)$ such that 
$$
(1-\tilde{\cutPhase})u=O(\hbar^{\infty})_{H_{\hbar}^\ell}. 
$$
In particular, $\WF_{\hbar r}(u)=\WF_{\hbar r}(\tilde{\cutPhase} u)$ is compact and there is $U$, a neighbourhood of $\xi_0$, such that $\overline{U}\cap \pi_{\xi}(\WF_{\hbar r}(u))=\emptyset$. Thus, there is $b\in C_c^\infty(\mathbb{R}^2)$ such that 
$$
\Op{\hbar r}(1-b)u= O(\hbar^{\infty})_{H_{\hbar}^{\ell}}
$$
and $\pi_\xi(\supp b)\cap U=\emptyset.$

Let $x_0\in \mathbb{R}$ and suppose $a\in C_c^\infty(\mathbb{R}^2)$ with $a(x_0,\xi_0)=1$  and $\pi_\xi \supp a\subset U$.  Then
\begin{align*}
&\Op{\hbar r}(a)U_{r} \sem E(\semOp{\mathbf{Q}_1})(\omega)U_{r}^*u=\Op{\hbar r}(a)U_{r} \sem E_{\sem H_2}(\omega)U_{r}^*\Op{\hbar r}(b)u +O(\hbar^{\infty})_{H_{\hbar}^\ell}\\
&\;=\frac{1}{(2\pi \hbar r)^3}\int\int _{G_{\hbar}(\omega)} e^{\frac{i}{\hbar r}[(w-z)\xi+(x-w)\eta+(z-y)\zeta]}a(\tfrac{x+w}{2},\eta)b(\tfrac{z+y}{2},\zeta)u(y)d\xi dyd\zeta dz dw d\eta +O(\hbar^{\infty})_{H_{\hbar}^\ell}.
\end{align*}
Since $|\zeta-\eta|>c>0$ on the support of the integrand, integration by parts in $(z,w)$ shows that $(x_0,\xi_0)\notin \WF_{r \hbar}(U_{r} \sem E(\semOp{\mathbf{Q}_1})(\omega)U_{r}^*u)$.

Next, \referee{let} $(x_0,\xi_0)\in \mathbb{R}^2$ such that $|\xi_0|>\omega$. Then there is a neighbourhood $U$ of $(x_0,\xi_0)$ such that $\overline{U}\cap \{|\xi|\leq \omega\}=\emptyset$. \referee{As above,} let $a\in C_c^\infty(\mathbb{R}^2)$ with $\pi_\xi \supp a\subset U$ and $a(x_0,\xi_0)=1$. Then
\begin{align*}
&\Op{\hbar r}(a)U_{r} \sem E(\semOp{\mathbf{Q}_1})(\omega)U_{r}^*u=\Op{\hbar r}(a)U_{r} \sem E_{\sem H_2}(\omega)U_{r}^*\Op{\hbar r}(\tilde{\cutPhase})u +O(\hbar^{\infty})_{H_{\hbar}^\ell}\\
&\;=\frac{1}{(2\pi \hbar r)^3}\int\int _{G_{\hbar}(\omega)} e^{\frac{i}{\hbar r}[(w-z)\xi+(x-w)\eta+(z-y)\zeta]}a(\tfrac{x+w}{2},\eta)\tilde{\cutPhase}(\tfrac{z+y}{2},\zeta)u(y)d\xi dyd\zeta dz  dw d\eta +O(\hbar^{\infty})_{H_{\hbar}^\ell}.
\end{align*}   
Now, since $\pi_\xi( \supp a)\subset U$ and $|\xi|\leq \omega +C\hbar$ in $G_\hbar(\omega)$, we have $|\xi-\eta|>c>0$ on the integrand, and integration by parts in $w$ then shows that $(x_0,\xi_0)\notin \WF_{r \hbar}(U_{r} \sem E(\semOp{\mathbf{Q}_1})(\omega)U_{r}^*u)$ as claimed.
\end{proof}

The final piece of the proof involves rebuilding the layers of our potential. That is, we compute asymptotics for a series of oscillatory integrals, coming from $U$ and $U^*$, which oscillate at different scales. In particular, the unitary operator, $U$, used to gauge transform from $\semOp{\referee{{}^P\!\mathbf{Q}_{0}}}$ to $\semOp{\mathbf{Q}_1}$ is of the form
\begin{equation}
\label{subseqform}
U= e^{i\Op{\hbar}(g_{-1})}\dots e^{i\Op{\hbar}(g_{\referee{N}_\delta-1})}e^{i\Op{\hbar}(g_{\referee{N}_\delta})}
\end{equation}
with $g_{-1}\in \Spz^{\fcomp}$, $g_j\in \referee{\lp_n^{\delta}}r_j^{-1}S^{\fcomp}_{r_j,\delta}$, and $r_j=\referee{\lp_n^{-j\delta}}$.

We start by showing that $U$ and $U^*$ do not appreciably move the momentum variables ($\xi$'s).
\begin{lemma}
\label{l:onlyLowFreq}
Let $\referee{\gamma} \in S^0$ be compactly supported in $x$ such that $(\supp \referee{\gamma}) \cap \{\xi\in [-b,b]\}=\emptyset$. Then there is $\e>0$ such that for all $\omega \in [-b-\e,b+\e]$ we have 
$$
\sem E(\semOp{\mathbf{Q}_1})(\omega)U^*\Op{\hbar}(\referee{\gamma}) =O(\hbar^{\infty})_{H_{\hbar}^{-\ell}\to H_{\hbar}^{\ell}},\qquad  \Op{\hbar}(\referee{\gamma})U\sem E(\semOp{\mathbf{Q}_1})(\omega)  =O(\hbar^{\infty})_{H_{\hbar}^{-\ell}\to H_{\hbar}^{\ell}}.
$$
\end{lemma}
\begin{proof}
First, observe that 
\begin{multline*}U^*=U_{r_{\referee{N}_\delta}} U_{r_{\referee{N}_\delta}}^*e^{-i\Op{\hbar}(g_{\referee{N}_\delta})}U_{r_{\referee{N}_\delta}}U_{r_{\referee{N}_\delta}}^*U_{r_{\referee{N}_\delta-1}} \circ \\
U_{r_{\referee{N}_\delta-1}}^* e^{-i\Op{\hbar}(g_{\referee{N}_\delta-1})}U_{r_{\referee{N}_\delta-1}}\dots U_{r_{1}}^* e^{-i\Op{\hbar}(g_{0})} e^{-i\Op{\hbar}(g_{-1})}
\end{multline*}
and hence, by Lemmas~\ref{l:waveFront1} and~\ref{l:waveFront2}, for any $\tilde{\cutPhase}\in S^0$ with $\tilde{\cutPhase}\equiv 1$ on $\supp e$ we have
$$
U^*\Op{\hbar}(\referee{\gamma})= U_{r_{\referee{N}_\delta}}\Op{\hbar r_{\referee{N}_\delta}}(\tilde{\cutPhase})U_{r_{\referee{N}_\delta}}^*U^*\Op{\hbar}(\referee{\gamma})+O(\hbar^{\infty})_{H_{\hbar}^{-\ell}\to H_{\hbar}^\ell}.
$$
In particular, letting $\tilde{\cutPhase}\in S^0$ with $\supp \tilde{\cutPhase}\cap \{\xi\in [-b,b]\}=\emptyset$, $\tilde{\cutPhase}\equiv 1$ on $\supp \referee{\gamma}$, and $\cutFreq\in C_c^\infty$ such that $\cutFreq\equiv 1$ in a neighbourhood of $[-b,b]$, and $\supp \cutFreq(\xi)\cap \supp \tilde{\cutPhase}=\emptyset$,  we have by Lemma~\ref{l:frequencyPreserved}
\begin{align*}
 &\sem E(\semOp{\mathbf{Q}_1})(\omega)U^*\Op{\hbar}(\referee{\gamma})\\
 &= U_{r_{\referee{N}_\delta}}U^*_{r_{\referee{N}_\delta}} \sem{E}_{\sem{H}_2}(\omega)U_{r_{\referee{N}_\delta}}\Op{\hbar r_{\referee{N}_\delta}}(\tilde{\cutPhase})U_{r_{\referee{N}_\delta}}^*U^*\Op{\hbar}(\referee{\gamma})+O(\hbar^{\infty})_{H_{\hbar}^{-\ell}\to H_{\hbar}^\ell}\\
 &= U_{r_{\referee{N}_\delta}}U^*_{r_{\referee{N}_\delta}} \sem E(\semOp{\mathbf{Q}_1})(\omega)U_{r_{\referee{N}_\delta}}\Op{\hbar r_{\referee{N}_\delta}}(1-\cutFreq(|\xi|))\Op{\hbar r_{\referee{N}_\delta}}(\tilde{\cutPhase})U_{r_{\referee{N}_\delta}}^*U^*\Op{\hbar}(\referee{\gamma})+O(\hbar^{\infty})_{H_{\hbar}^{-\ell}\to H_{\hbar}^\ell}\\
 &=O(\hbar^{\infty})_{H_{\hbar}^{-\ell}\to H_{\hbar}^\ell}.\end{align*}

\end{proof}

The final preparatory lemma before we proceed to the proof of our main theorem gives asymptotics for the spectral function $\sem{E}_{\semOp{\referee{{}^P\!\mathbf{Q}_{0}}}}$ in terms of the discrete parameter $\referee{\lp_n}$. Since the number of unitary operators from which $U$ is built depends on the value of $\referee{N}$ in the error from~\eqref{e:comparison}, the number of oscillatory integrals needed to describe $\sem{E}_{\semOp{\referee{{}^P\!\mathbf{Q}_{0}}}}$ ($N_0$ in the lemma below) will also depend on $\referee{N}$. \referee{In the proof of the next lemma, we will need Lemma~\ref{l:localForm} which gives an oscillatory integral approximation to $e^{-i\Op{\hbar}(g)}$ when $g\in r^{-1}S^{\fcomp}_{r,\delta}.$}
\begin{lemma}
\label{l:asymptotics1}
There is $\referee{\Upsilon}>0$ and $\{\Psi_j\}_{j=1}^{\referee{\Upsilon}}\in \Sp^1$ such that for all $\alpha,
\beta\in \mathbb{N}$, there are $e_{j\alpha\beta}\in \Sp$, $j=1,\dots, \referee{\Upsilon}$, such that for $\omega\in [a,b]$
\begin{multline*}
\langle \sem E(\semOp{\mathbf{Q}_1})(\omega)U^* (-\partial_{\referee{y}})^\alpha\delta_{y_0},U^*(-\partial_{\referee{x}})^\beta\delta_{x_0}\rangle\\=\sum_{j=1}^{\referee{\Upsilon}}\hbar^{-1-\alpha-\beta}\int_{G_\hbar(\omega)} e^{\frac{i}{\hbar}(x_0-y_0)\Psi_j(x_0,y_0,\eta)}e_{j\alpha\beta}(x_0,y_0,\eta) d\eta,
\end{multline*}
and 
\begin{equation}
\label{e:psiForm}
\Psi_j \sim \eta+\sum_{\referee{l\geq1}} \hbar^{\referee{l}}\referee{\lp_n^{l\delta}}\Psi_{j,\referee{l}},\qquad \referee{\Psi_{j,\referee{l}}\in S^1.}
\end{equation}
\referee{Here, $U$ is the unitary operator in~\eqref{subseqform}.}
\end{lemma}
\begin{proof}
By Lemma~\ref{l:onlyLowFreq}, we need only to compute
$$
\langle \sem E(\semOp{\mathbf{Q}_1})(\omega)U^*\Op{\hbar}(\cutFreq(|\xi|))(-\partial_{\referee{y}})^\alpha\delta_{y_0}, U^*\Op{\hbar}(\cutFreq(|\xi|))(-\partial_{\referee{x}})^{\beta}\delta_{x_0}\rangle
$$
for a given $\cutFreq\in C_c^\infty(\mathbb{R})$. 
We use again that
\begin{multline*}U^*=U_{r_{\referee{N}_\delta}} U_{r_{\referee{N}_\delta}}^*e^{-i\Op{\hbar}(g_{\referee{N}_\delta})}U_{r_{\referee{N}_\delta}}U_{r_{\referee{N}_\delta-1}}^*U_{r_{\referee{N}_\delta-1}}\circ \\
U_{r_{\referee{N}_\delta-1}}^* e^{-i\Op{\hbar}(g_{\referee{N}_\delta-1})}U_{r_{\referee{N}_\delta-1}}\dots \referee{U_{r_1}^*}e^{-i\Op{\hbar}(g_{0})} e^{-i\Op{\hbar}(g_{-1})}.
\end{multline*}
\begin{remark}
\label{r:deltaNonzero}
\referee{Recall that the number of products here is large, but independent of $\hbar$. If, in our onion peeling argument, we peeled away layers of the form $2^{-j-1}\leq |\theta|\leq 2^{-j}$ rather than $\lp_n^{-(j+1)\delta}\leq |\theta|\leq \lp_n^{-j\delta}$, then we would require $\sim|\log \hbar|$ steps to obtain a constant coefficient operator. Not only would this require \emph{much} finer control in each step of the gauge transform, but also unpeeling the onion would become substantially more complicated. In particular, this is why we cannot take $\delta=0$.}
\end{remark}

By Lemma~\ref{l:localForm}, for all $k=1,\dots,\referee{N}_\delta$, there are $\{\referee{W}_{j,k}\}_{j=1}^{\referee{\Upsilon}_k}\subset \mathbb{R}_{(x,\xi)}^2$ open such that 
$$
\{(\referee{x_0},\xi)\in \referee{T^*\mathbb{R}}\mid |\xi|\leq b\}\subset \bigcup_{j=1}^{\referee{\Upsilon}_k}\referee{W}_{j,k}
$$ 
and for all $u$ with $\WF_{r_j \hbar}(u)\subset \referee{W}_{j,k}$, $U_{r_j}^*e^{-i\Op{\hbar}(g_j)}U_{r_j}u$ takes the form~\eqref{e:form} with $t=\hbar \referee{\lp_n^{\delta}}$ and $g=r_j\referee{\lp_n^{-\delta}} g_j$. Furthermore, there are $\{\referee{W}_{j,0}\}_{j=1}^{\referee{\Upsilon}}$ such that
$$
\{(x_0,\xi)\mid |\xi|\leq b\}\subset \bigcup_{j=1}^{\referee{\Upsilon}}\referee{W}_{j,0}
$$
and for all $k=1,\dots, \referee{N}_\delta$ and $j=1,\dots , \referee{\Upsilon}$ there is $i_{k,j}$ such that 
\begin{equation}
\label{e:stepDown}
\{(\referee{x_0},\xi)\mid \xi\in \referee{W}_{j,0}\}\subset \referee{W}_{i_{k,j},k}.
\end{equation}
In addition, for all $j=1,\dots, \referee{\Upsilon}$, and $u$ with $\WF_{\hbar}(u)\subset \referee{W}_{j,0}$, $e^{-i\Op{\hbar}(g_0)}u$ takes the form~\eqref{e:form} with $t=\hbar \referee{\lp_n^{\delta}}$ and $g=\referee{\lp_n^{-\delta}} g_0$.

Let $\{\cutPhase_j\}_{j=1}^N$ be a partition of unity near $\{(x_0,\xi)\mid |\xi|\leq b\}$ subordinate to $\{\referee{W}_{j,0}\}_{j=1}^{\referee{\Upsilon_0}}$. Then,
$$
\Op{\hbar}(\cutFreq(|\xi|))(-\partial_{\referee{x}})^\beta\delta_{x_0}=\sum_{j=1}^N \Op{\hbar}(\cutFreq(|\xi|)\cutPhase_j)(-\partial_{\referee{x}})^\beta\delta_{x_0}.
$$
Let $\tilde{\cutPhase}_j\in C_c^\infty(U_{j,0})$ with $\tilde{\cutPhase}_j\equiv 1$ on $\supp \cutPhase_j$. Then, Lemmas~\ref{l:waveFront1} and~\ref{l:waveFront2} imply

$$
\WF_{\hbar r_{\referee{N}_\delta}}(U_{r_{\referee{N}_\delta}}^*U^*\Op{\hbar}(\cutFreq(|\xi|)\cutPhase_j)(-\partial_{\referee{x}})^\beta\delta_{x_0})\subset \{(x,\xi)\,|\, x=0,\xi\in \pi_{\xi}(\supp \cutFreq(|\xi|)\cutPhase_j)\}
$$
and 
$$
\WF_{\hbar r_{\referee{N}_\delta}}(U_{r_{\referee{N}_\delta}}^*U^*\Op{\hbar}(\cutFreq(|\xi|)(1-\tilde{\cutPhase}_j)(-\partial_{\referee{x}})^\beta\delta_{x_0})\subset \{(x,\xi)\,|\, x=0,\,\xi\in \pi_{\xi}(\supp \cutFreq(|\xi|)(1-\tilde{\cutPhase}_j))\}.
$$

In particular, Lemma~\ref{l:onlyLowFreq} implies 
$$
\langle \sem E(\semOp{\mathbf{Q}_1})(\omega)U^*\Op{\hbar}(\cutFreq(|\xi|)\cutPhase_j)\delta_{x_0}, U^*\Op{\hbar}(\cutFreq(|\xi|)(1-\tilde{\cutPhase}_j)(-\partial_{\referee{x}})^\beta\delta_{x_0}\rangle =O(\hbar^{\infty}).
$$

We now analyze
$$
U^*\Op{\hbar}(\cutFreq(|\xi|)\cutPhase_j)(-\partial_{\referee{x}})^\beta\delta_{x_0}.
$$
To ease notation, we put
$$
\begin{gathered}
v_{k,j}:=U_{r_{k}}^* e^{-i\Op{\hbar}(g_{k})}U_{r_{k}}\dots \referee{U_{r_1}^*}e^{-i\Op{\hbar}(g_{0})} e^{-i\Op{\hbar}(g_{-1})}\Op{\hbar}(\cutFreq(|\xi|)\cutPhase_j)(-\partial_{\referee{x}})^\beta\delta_{x_0},\\
\tilde{v}_{k,j}:=U_{r_{k}}^* e^{-i\Op{\hbar}(g_{k})}U_{r_{k}}\dots \referee{U_{r_1}^*}e^{-i\Op{\hbar}(g_{0})} e^{-i\Op{\hbar}(g_{-1})}\Op{\hbar}(\cutFreq(|\xi|)\tilde{\cutPhase}_j)(-\partial_{\referee{x}})^\beta\delta_{x_0}.
\end{gathered}
$$
Since $g_{-1}\in S^{\fcomp}$, we have $e^{-i\Op{\hbar}(g_{-1})}\in \Psi^{0}$ and 
$$\WF_{\hbar}(e^{-i\Op{\hbar}(g_{-1})}\Op{\hbar}(\cutFreq(|\xi|)\cutPhase_j)(-\partial_{\referee{x}})^\beta\delta_{x_0})\subset U_{j,0}.$$
 Moreover, by Lemmas~\ref{l:waveFront1} and~\ref{l:waveFront2} together with~\eqref{e:stepDown}, we may assume that $U_{r_j}^*e^{-i\Op{\hbar}(g_j)}U_{r_j}$ takes the form~\eqref{e:form} as described above.

By Lemma~\ref{l:localForm}, since $g_{-1}\in \Spz^{\fcomp}$ and $g_0\in \referee{\lp_n^{\delta}}\Spz^{\fcomp}$, we have
$$
v_{0,j}(x)=\frac{\hbar^{-\beta}}{(2\pi \hbar)^2}\int e^{\frac{i}{\hbar} (\varphi_{0,\referee{j}}(\hbar\referee{\lp_n^{\delta}},x,\eta)-y\eta+(y-x_0)\xi)}a_j(x,y,x_0,\eta,\xi)dyd\xi d\eta
$$
with $a_j\in \Sp^{\fcomp}$. Now, observe that
$$
\varphi_{0,\referee{j}}(\hbar \referee{\lp_n^{\delta}},x,\eta)\sim \langle x,\eta\rangle+ \sum_{\referee{l}=1}^\infty \frac{\hbar^{\referee{l}}\referee{\lp_n^{l\delta}}}{\referee{l}!}\partial_t^{\referee{l}}\varphi_0(0,x,\eta),\qquad
\partial_t\varphi_0=g_0(x,\partial_x\varphi_0).
$$
In particular, $\varphi_{0,\referee{j}}\in \Sp^1$.

Applying stationary phase in $(y,\eta)$, we obtain
$$
v_{0,j}(x)=\frac{\hbar^{-\beta}}{2\pi \hbar}\int e^{\frac{i}{\hbar} (\varphi_{0,\referee{j}}(\hbar \referee{\lp_n^{\delta}},x,\eta)-x_0\eta)}\tilde{a}_{0,j}(x,\eta,x_0)d\eta 
$$
for some $\tilde{a}_{0,j}\in \Sp^{\fcomp}$.
We claim that
\begin{equation}
\label{e:claimThis}
v_{k,j}(x)=\frac{\hbar^{-\beta} }{(2\pi \hbar )r_k^{1/2}}\int e^{\frac{i}{\hbar r_k}(\tilde{\varphi}_{k,\referee{j}}(x,\eta,x_0)-r_kx_0\eta)}\tilde{a}_{k,j}(x,\eta,x_0)d\eta
\end{equation}
with $\tilde{\varphi}_{k,\referee{j}}\in \Sp^1$, $\tilde{a}_{k,j}\in \Sp^{\fcomp}$, 
\begin{equation}
\label{e:varphiForm}
\tilde{\varphi}_{k,\referee{j}}(x,\eta,x_0)\sim x\eta +\sum_{l=1}^\infty \hbar^{l}\referee{\lp_n^{l\delta}}\tilde{\varphi}_{k,\referee{j},l}(x,\eta,x_0),\qquad |\partial^\alpha_{x_0}\partial_x^{\beta_1}\partial_\eta^{\beta_2}\tilde{\varphi}_{k,\referee{j},l}|\leq C_{k\alpha\beta_1\beta_2}r_k,\quad \alpha\geq 1
\end{equation}
and $\tilde{\varphi}_{k,\referee{j},l}\in C^\infty$ having bounded derivatives. Indeed, we have checked this for $k=0$.

\begin{remark}
\referee{Observe that we claim in~\eqref{e:claimThis} that the integral kernel of $U^*_{r_{N_\delta}}U$ takes the form given by~\eqref{e:claimThis} with $k=N_\delta$. Indeed, $U^*_{r_{N_\delta}}U$ is `nearly' a semiclassical Fourier integral operator with small parameter $\hbar r_{N_\delta}$. The formal issue with this statement is that the phase function is not independent of $\hbar$. 
}
\end{remark}

Suppose~\eqref{e:claimThis} holds for some $k=1,\dots, N-1$. Then we compute 
$$v_{N,j}=U_{r_N}^*e^{-i\Op{\hbar}(g_N)}U_{r_N} U_{r_N}^*U_{r_{N-1}}v_{N-1,j}.$$
Observe that there is $\varphi_{N,\referee{j}}\in \Sp^1$ with $\varphi_{N,\referee{j}}(\hbar s^{-1},x,\eta)=\langle x,\eta\rangle +O(\hbar s^{-1})_{C^\infty}$ and
\begin{align*}
v_{N,j}(x)&=\frac{\hbar^{-\beta}}{(2\pi\hbar)^2r_{N-1}}r_N^{-\frac{1}{2}}\int e^{\frac{i}{\hbar r_{N}} (\varphi_{N,\referee{j}}(\hbar \referee{\lp_n^{\delta}},x,\eta)+\frac{r_{N}}{r_{N-1}}(-y\eta+ \tilde{\varphi}_{N-1,\referee{j}}(y,\xi,r_{N-1}x_0)-r_{N-1}x_0\xi))}\\
&\qquad\qquad\qquad\qquad\qquad\qquad\qquad\qquad \qquad\qquad\qquad\qquad \times a_{N,\referee{j}}(x,\eta)\tilde{a}_{N-1,j}(y,x_0,\xi)dyd\eta d\xi \\
&=\frac{\hbar^{-\beta}}{(2\pi\hbar)r_{N}^{1/2}}\int e^{\frac{i}{\hbar r_{N}} (\varphi_{N,\referee{j}}(\hbar \referee{\lp_n^{\delta}},x,\eta)+\frac{r_{N}}{r_{N-1}}(-y_c\eta+ \tilde{\varphi}_{N-1,\referee{j}}(y_c,\xi_c,r_{N-1}x_0)-r_{N-1}x_0\xi_c))}\tilde{a}_{N,j}(\eta,x_0)d\eta .
\end{align*}
In the last line we apply stationary phase in the $(y,\xi)$ variables \referee{to obtain $\tilde{a}_{N,j}$.} We then find, using the asymptotics~\eqref{e:varphiForm}, that the critical point $(y_c,\xi_c)$ solving
$$
\partial_\xi(-y\eta+ \tilde{\varphi}_{N-1,\referee{j}}(y,\xi,r_{N-1}x_0)-r_{N-1}x_0\xi)|_{\substack{y=y_c\\\xi=\xi_c}}=\partial_y(-y\eta+ \tilde{\varphi}_{N-1,\referee{j}}(y,\xi,r_{N-1}x_0)-r_{N-1}x_0\xi)|_{\substack{y=y_c\\\xi=\xi_c}}=0
$$
satisfies
\begin{equation*}
y_c(x_0,\eta)\sim r_{N-1} x_0+\sum_{l=1}^\infty \hbar^{l}\referee{\lp_n^{l\delta}} y_{c,l}(x_0,\eta),\qquad \xi_c(x_0,\eta)=\eta+\sum_{l=1}^\infty \hbar^{l}\referee{\lp_n^{l\delta}}\eta_{c,l}(x_0,\eta)
\end{equation*}
with $y_{c,l}, \eta_{c,l}\in C^\infty$ and
\begin{equation}\label{e:critical}
|\partial_{x_0}^\alpha \partial_\eta^\beta y_{c,l}|+|\partial_{x_0}^\alpha\partial_\eta^\beta\xi_{c,l}|\leq C_{\alpha \beta l}r_{N-1},\qquad \alpha\geq 1.
\end{equation}
To see~\eqref{e:critical}, observe that 
$$
\begin{pmatrix}\partial_{x_0}y_c\\\partial _{x_0}\eta_c\end{pmatrix}=\begin{pmatrix}\partial^2_y\tilde{\varphi}_{N-1,\referee{j}} (y_c,\xi_c,x_0)&\partial^2_{y\xi}\tilde{\varphi}_{N-1,\referee{j}} (y_c,\xi_c,x_0)\\
\partial^2_{\xi y}\tilde{\varphi}_{N-1,\referee{j}} (y_c,\xi_c,x_0)&\partial^2_{\xi\xi}\tilde{\varphi}_{N-1,\referee{j}} (y_c,\xi_c,x_0)\end{pmatrix}^{-1}\begin{pmatrix} -\partial_{x_0 y}^2\tilde{\varphi}_{N-1,\referee{j}}(y_c,\xi_c,x_0)\\r_{N-1}-\partial_{x_0\xi}^2\tilde{\varphi}_{N-1,\referee{j}}(y_c,\xi_c,x_0)\end{pmatrix},
$$
and hence~\eqref{e:varphiForm} implies~\eqref{e:critical}. Plugging the expression into the formula for $v_{N,j}$ then completes the proof of the inductive step.

An identical analysis shows that 
$$
\tilde{v}_{k,j}(x)=\frac{\hbar^{-\beta} }{(2\pi \hbar)r_k^{1/2}}\int e^{\frac{i}{\hbar r_k}(\tilde{\varphi}_{k,\referee{j}}(x,\eta,x_0)-r_kx_0\eta)}\tilde{b}_{k,j}(x,\eta,x_0)d\eta,
$$
\referee{for some $\tilde{b}_{k,j}\in \Sp^{\fcomp}$.} Here, crucially, the same phase function $\tilde{\varphi}_{k,\referee{j}}$ appears as in $v_{k,j}$.

Now, putting $k_\referee{N}=\referee{N}_\delta$, we obtain
\begin{align*}
&\langle \sem E(\semOp{\mathbf{Q}_1})(\omega)U^*(-\partial_{\referee{y}})^\alpha\delta_{y_0},U^*(-\partial_{\referee{x}})^\beta\delta_{x_0}\rangle\\
&=\sum_{j=1}^{\referee{\Upsilon}}\langle  U_{r_{k_\referee{N}}}\sem E(\semOp{\mathbf{Q}_1})(\omega)U^*_{r_{k_\referee{N}}}v_{k_\referee{N},j},\tilde{v}_{k_\referee{N},j}\rangle +O(\hbar^{\infty})\\
&=\sum_{j=1}^{\referee{\Upsilon}}\frac{\hbar^{-\alpha-\beta}}{(2\pi\hbar)^3 r_{k_\referee{N}}^{2}}\int_{G_\hbar(\omega)}\int e^{\frac{i}{\hbar r_{k_\referee{N}}}( (x-y)\eta+\tilde{\varphi}_{k_\referee{N},\referee{j}}(y,\xi,y_0)-r_{k_\referee{N}}y_0\xi-\tilde{\varphi}_{k_\referee{N},\referee{j}}(x,\zeta,x_0)+r_{k_\referee{N}}x_0\zeta)}\\
&\qquad\qquad\qquad\qquad\qquad\tilde{a}_{k_\referee{N},j} \overline{\tilde{b}_{k_\referee{N},j}}dy d\xi dx d\zeta d\eta +O(\hbar^{\infty}).
\end{align*}
Finally, performing stationary phase in $(y,\xi)$, $(x,\zeta)$, the critical points are given by $y=y_c(\eta,y_0)$, $\xi=\xi_c(\eta,y_0)$, $x=y_c(\eta,x_0)$, and $\zeta=\xi_c(\eta,x_0)$. In particular, when $x_0=y_0$, the phase vanishes. Moreover, we have $y_c\in\Sp^0,\xi_c\in \Sp^1$ and
$$ y_c\sim r_{k_\referee{N}}x_0+\sum_{l\geq 1}\hbar^{l}\referee{\lp_n^{l\delta}} y_{c,l},\quad \xi_c=\eta+\sum_{l\geq 1}\hbar^{l}\referee{\lp_n^{l\delta}} \xi_{c,l},\quad
|\partial^\alpha_{x_0}\partial_\eta^\beta y_{c,l}|+|\partial^\alpha_{x_0}\partial_\eta^\beta\xi_{c,l}|\leq C_{l\alpha\beta} r_{k_\referee{N}},\quad \alpha\geq 1.$$

Therefore, writing $\Phi_{\referee{j}}(x_0,y_0,\eta)$ for the phase at the critical point, we have
$$
\Phi_{\referee{j}}(x_0,y_0,\eta)= r_{k_\referee{N}}(x_0-y_0)\Psi_{\referee{j}}(x_0,y_0,\eta)
$$
for some $\Psi_{\referee{j}} \in \Sp^1$ with $\Psi_{\referee{j}}=\eta+O(\hbar\referee{\lp_n^{\delta}})$.

In particular, this implies
$$
\langle \sem E(\semOp{\mathbf{Q}_1})(\omega)U^*(-\partial_{\referee{y}})^\alpha\delta_{y_0},U^*(-\partial_{\referee{x}})^{\beta}\delta_{x_0}\rangle=\sum_{j=1}^{\referee{\Upsilon}}\hbar^{-1-\alpha-\beta}\int_{G_\hbar(\omega)} e^{\frac{i}{\hbar}(x_0-y_0)\Psi_{\referee{j}}(x_0,y_0,\eta)}e_j(x_0,y_0,\eta) d\eta,
$$
with $\Psi_{\referee{j}}$ and $e_j$ as claimed.
\end{proof}

\referee{Now that we have obtained an asymptotic expansion for the spectral projector of $\semOp{\mathbf{Q}_1}$, we pass to $\semOp{\mathbf{Q}}$}. First, observe that Lemma~\ref{l:asymptotics1} implies that the assumptions~\eqref{e:lip2} hold for all $\alpha,\beta, N$ and $\omega \in [a+\e,b-\e]$, with $T(\hbar)=O(\hbar^{-N})$. Therefore, by Proposition~\ref{p:distantPerturb},  
$$
\partial_x^\alpha\partial_y^\beta \sem E(\semOp{\mathbf{Q}_0})(x,y,\omega)-\partial_x^\alpha\partial_y^\beta \sem E(\semOp{\mathbf{Q}_1})(x,y,\omega)=O(\lp_n^{-\referee{N}+|\alpha|+|\beta|}).
$$
\referee{Here, the implicit constant depends on $\alpha,\beta,\referee{N},a,$ and $b$ but not on $\hbar,n,x,y,$ and $\omega$.} 
Now, using Lemma~\ref{l:asymptotics1}, we have for $\hbar\in [\lp^{-1}_{n+1},\lp^{-1}_{n-1}]$,
\begin{equation}
\label{e:asymptotic1}
\begin{gathered}
\partial_x^\alpha\partial_y^\beta \sem E(\semOp{\mathbf{Q}_0})(x_0,x_0,1)= C_{0,\alpha,\beta,n}(x_0)\hbar^{-1-\alpha-\beta} \sum_{k=1}^K c_{k,n,\alpha,\beta}(x_0)\hbar^{k}+O(\lp_n^{-\referee{N}}),\\| c_{k,n,\alpha,\beta}|\leq C_{k\alpha\beta} \referee{\lp_n^{k\delta}}.\end{gathered}
\end{equation}
By~\cite[Lemma 3.6]{PaSh:16} Theorem~\ref{t:USBAsymptotics} holds for $x$ in any bounded set.

Next, we prove Theorem~\ref{t:USBAsymptoticsOff}. When $|x_0-y_0|>0$, \referee{using~\eqref{e:psiForm}}, we have $|\partial_\eta \Psi (x_0-y_0)|>c|x_0-y_0|>0$. Therefore, we can integrate by parts using $L=\hbar \frac{  D_\eta}{(x_0-y_0)\partial_\eta \Psi}$ and setting
$$
G_{\hbar,\pm}(\omega):=\pm \sup\{  \eta\in \mathbb{R}\,:\, \pm\eta \in G_\hbar(\omega)\}\sim \pm\omega+\sum_{j\geq 1} g_{\pm,j}(\omega)\hbar^{j}\referee{\lp_n^{j\delta}},
$$
 to obtain
\begin{multline*}
\hbar^{-1} \int_{G_\hbar(\omega)} e^{\frac{i}{\hbar}(x_0-y_0)\Psi(x_0,y_0,\eta_0)}e_j(x_0,y_0,\eta)\\=\frac{1}{x_0-y_0}\Big(\Big[e^{\frac{i}{\hbar} (x_0-y_0)\Psi(x_0,y_0,\eta)}\frac{e_j(x_0,y_0,\eta)}{\partial_\eta \Psi}\Big]^{\eta=G_{\hbar,+}(\omega)}_{\eta=G_{\hbar,-}(\omega)}\\
-\int_{ G_{\hbar}(\omega)}e^{\frac{i}{\hbar}(x_0-y_0)\Psi(x_0,y_0,\eta)}D_\eta\frac{e_j(x_0,y_0,\eta)}{\partial_\eta \Psi(x_0,y_0,\eta)}d\eta\Big).
\end{multline*}
Repeating this process and using that $|\partial_\eta \Psi|>c>0$, we obtain for $|x_0-y_0|\gg \hbar$, that there are $c_{k, n,\alpha,\beta\pm}(x_0,y_0)$ satisfying
$$
| c_{k,n,\alpha,\beta, \pm}|\leq C_{k\alpha \beta}|x_0-y_0|^{-\alpha-\beta} \referee{\lp_n^{k\delta}}
$$
such that
 \begin{equation}
\label{e:asymptotic1B}
\begin{aligned}
\partial_x^\alpha\partial_y^\beta& \sem E(\semOp{\mathbf{Q}_0})(x_0,y_0,\omega)= e^{\frac{i}{\hbar}\tilde{\Psi}_{+,n}(x_0,y_0)}\hbar^{-\alpha-\beta}(\sum_{k=0}^K c_{k,\alpha,\beta,n,+}(x_0,y_0)(x_0-y_0)^{-k-1}\hbar^{k})\\
&+e^{\frac{i}{\hbar}\tilde{\Psi}_{-,n}(x_0,y_0)}\hbar^{-\alpha-\beta}(\sum_{k= 0}^K c_{k,\alpha,\beta ,n,-}(x_0,y_0)(x_0-y_0)^{-k-1}\hbar^{k})+O(\lp_n^{-\referee{N}+\alpha+\beta}),
\end{aligned}
\end{equation}
where $\tilde{\Psi}_{\pm}=(x_0-y_0)\Psi(x_0,y_0,G_{\pm,\hbar}(\omega))\in \Sp^0$ satisfies
$$
\tilde{\Psi}_{\pm,n}\sim \pm(x_0-y_0)\omega+ (x_0-y_0)\hbar \referee{\lp_n^{\delta}}\tilde{\Psi}_{1,\pm,n}+ (x_0-y_0)\hbar^{2}\referee{\lp_n^{2\delta}}\sum_{j=0}^\infty \hbar^j\referee{\lp_n^{j\delta}} \tilde{\Psi}_{j+2,\pm,n}.
$$

Now, 
$$
e^{\frac{i}{\hbar} \tilde{\Psi}_{\pm,n}} =e^{\pm \frac{i}{\hbar} (x_0-y_0)\omega}e^{i (x_0-y_0)\referee{\lp_n^{\delta}}\tilde{\Psi}_{1,\pm,n}}\sum_{j=0}^{J-1} (\hbar (x_0-y_0)\referee{\lp_n^{2\delta}})^j\tilde{\Psi}_{j,\pm,n}'+O( (x_0-y_0)^J\hbar^{J}\referee{\lp_n^{2J\delta}}),
$$
where $\Psi_{j,\pm,n}'$ can be calculated from $\Psi_{j,\pm,n}$.

Therefore, since $\delta<\frac{1}{2}$, we may take $J$ large enough so that $\hbar^{J}\referee{\lp_n^{2J\delta}}=O(\lp_n^{-\referee{N}})$, and hence we have 
\begin{equation}
\label{e:asymptotic1C}
\begin{aligned}
\partial_x^\alpha\partial_y^\beta& \sem E(\semOp{\mathbf{Q}_0})(x_0,y_0,\omega)= e^{\frac{i}{\hbar}(x_0-y_0)\omega}\hbar^{-\alpha-\beta}\sum_{k= 0}^K \tilde{c}_{k,\alpha,\beta,n,+}(x_0,y_0)(x_0-y_0)^{-k-1}\hbar^{k}\\
&+e^{-\frac{i}{\hbar}(x_0-y_0)\omega}\hbar^{-\alpha-\beta}\sum^K_{\ell= 0} \tilde{c}_{\ell,\alpha,\beta,n,-}(x_0,y_0)(x_0-y_0)^{-\ell-1}\hbar^{\ell}+O(\lp_n^{-\referee{N}+\alpha+\beta})
\end{aligned}
\end{equation}
with 
$$
| \tilde{c}_{k,n,\alpha,\beta,\pm}|\leq C_{k\alpha\beta}|x_0-y_0|^{-|\alpha|-|\beta|} \referee{\lp_n^{2k\delta}}.
$$
Now, since $2\delta<\delta'$ and $|x_0-y_0|\geq\hbar^{1-\delta'}$, we may apply \referee{Lemma~\ref{l:glue}} together with Lemma~\ref{l:firstTerm} to complete the proof of Theorem~\ref{t:USBAsymptoticsOff} for $x$ and $y$ in a bounded set with $|x-y|\geq \hbar^{1-\delta'}$.


\referee{
\begin{remark}
\label{r:noGlue}
Although we have Lemma~\ref{l:asymptotics1} uniformly for $(x,y)$ in any compact subset of $\mathbb{R}^2$, since we do not know that the integrand there is close for $n$ and $n+1$, we are not able to glue our asymptotics using an analogue of Lemma~\ref{l:glue} to obtain a single integral formula for all $(x,y,\hbar)$. 
\end{remark}
}

\subsection{Uniformity in $x$}

It is easy to check that for any $N>0$  there is $K>0$ such that all the constants in the $O(\hbar^N)$ remainders above depend only on $\|\mathbf{Q}_0\|_{\sem{Diff}^1_{K}}$.

Now, let $T_su(x)=u(x+s)$ so that $T^*_su(x)=u(x-s)$. Then, $T_s$ is unitary and, with $\mathbf{Q}_s:=T_s \mathbf{Q}_0 T_s^*$,
$$
\|\mathbf{Q}_s\|_{\sem{Diff}^1_{K}}=\|\mathbf{Q}_0\|_{\sem{Diff}^1_{K}}.
$$
Note that
\begin{align*}
\sem{E}(\semOp{\mathbf{Q}_0})(s,y+s,\omega)&= \langle 1_{(-\infty,\omega^2]}(\semOp{\mathbf{Q}_0})\delta_{s},\delta_{y+s}\rangle \\
&= \langle 1_{(-\infty,\omega^2]}(\semOp{\mathbf{Q}_s})\delta_{0},\delta_{y}\rangle.
\end{align*}
Thus, since $\mathbf{Q}_s$ is bounded in $\sem{Diff}^1_{K}$, Theorems~\ref{t:USBAsymptotics} and~\ref{t:USBAsymptoticsOff} hold uniformly for all $x\in \mathbb{R}$ and $y\in B(x,R)$. 

\subsection{Derivatives in $\omega$}

\begin{lemma}
\label{l:derivativesOmega}
For all $\alpha,\beta\in \mathbb{N}$, there is $f_{\alpha,\beta}$ such that 
$$
\partial_x^\alpha\partial_y^\beta \sem{E}(\semOp{\mathbf{Q}_0})(x,y,\omega) = f_{\alpha,\beta}(x,y,\omega)+O(\hbar^\infty)
$$
and 
$$
|\partial_\omega^\ell  f_{\alpha,\beta}(x,y,\omega)|\leq C_{\alpha\beta \ell}\hbar^{-\alpha-\beta-\ell}\referee{|x-y|^\ell},\qquad \ell\geq 1.
$$
\end{lemma}
\begin{remark}
A more careful analysis of the gluing argument used to obtain our main theorems~\cite[Lemma 3.6]{PaSh:16} shows that in fact $f_{\alpha\beta}$ has a full asymptotic expansion in powers of $\hbar$ and this expansion can be differentiated in $\omega$.
\end{remark}
\begin{proof}
It is easy to see from Lemma~\ref{l:asymptotics1} that 
$$
\partial_x^\alpha\partial_y^\beta\partial^\ell_\omega \sem{E}(\semOp{\tilde{\mathbf{Q}}_2})(x,y,\omega)=\sum_{\pm}\sum_{j=1}^{\referee{\Upsilon}}\hbar^{-1-\alpha-\beta}\partial_{\omega}^{\ell}\Big(e^{\frac{i}{h}(x_0-y_0)\Psi_j(x_0,y_0,\eta_{\pm}(\omega))}e_{j\alpha\beta}(x_0,y_0,\eta_{\pm}(\omega))\Big),
$$
where $\eta_{\pm}(\omega)$ are the two smooth solutions of $|\eta_{\pm}(\omega)|^2+\hbar \tilde{\mathbf{q}}_2(\eta_{\pm}(\omega))=\omega^2$, with $\tilde{\mathbf{q}}_2$ as in~\eqref{e:Q2}. In particular, 
$$
\eta_{\pm}\sim \pm\omega +\sum_j \hbar^j\referee{\lp_n^{-j\delta}}\eta_{\pm,j}(\omega).
$$

Since 
$$
\partial_x^\alpha\partial_y^\beta \sem{E}(\semOp{\tilde{\mathbf{Q}}_2})(x,y,\omega)=\partial_x^\alpha\partial_y^\beta \sem{E}(\semOp{\mathbf{Q}_0})(x,y,\omega) +O(\hbar^\infty),
$$
this implies that 
$$
\partial_x^\alpha\partial_y^\beta \sem{E}(\semOp{\mathbf{Q}_0})(x,y,\omega)= f_{\alpha,\beta}(x,y,\omega)+O(\hbar^\infty),
$$
where 
$$
|\partial_\omega^\ell  f_{\alpha,\beta}(x,y,\omega)|\leq C_{\alpha\beta \ell}\hbar^{-\alpha-\beta-\ell}\referee{|x-y|^\ell},\qquad \ell\geq 1.
$$
\end{proof}

\section{Consequences of the main theorem}
\label{s:consequences}

In this section, we discuss a few consequences of our main theorem. Our first corollary is a direct consequence of Theorem~\ref{t:USBAsymptotics}. 
\begin{corollary}
Let $\mathbf{Q}_0\in\sem{Diff}^1$ and let $\{u_{\lambda_\alpha,\hbar }\}_{\alpha\in \mc{A}(\hbar)}$ be \referee{an orthonormal system of} $L^2(\mathbb{R})$-normalized eigenfunctions of $\semOp{\mathbf{Q}_0}$ with eigenvalues $\lambda^2_\alpha=\lambda^2_\alpha(\hbar)$; i.e. 
$$
(\semOp{\mathbf{Q}_0}-\lambda^2_\alpha)u_{\lambda_\alpha,\hbar}=0,\qquad \referee{\langle u_{\lambda_{\alpha,\hbar}},u_{\lambda_{\beta,\hbar}}\rangle= \delta_{\alpha,\beta}}.
$$ 
Then, for any $a\in (0,\infty)$ and $N>0$, there is $C_N$ such that 
$$
\sup_{x\in \mathbb{R}}\sum_{\lambda_{\alpha}\in [a,a+\referee{\zeta}]}|u_{\lambda_{\alpha},\hbar}(x)|^2\leq C_N\hbar^{-1+N}\langle \referee{\zeta}\hbar^{-N}\rangle.
$$
\end{corollary}
\begin{proof}
\referee{Let 
$$
\Lambda([a,a+\zeta]):=\overline{\operatorname{Span}\{ u_{\lambda_{\alpha,\hbar}}\,:\, \lambda_{\alpha,\hbar}\in[a,a+\zeta]\}},
$$
and $\Pi_{\Lambda}:L^2(\mathbb{R})\to \Lambda([a,a+\zeta])$ denote the orthogonal projector onto $\Lambda([a,a+\zeta])$. Then
\begin{align*}
\Pi_\Lambda&=\big(\sem E(\semOp{\mathbf{Q}_0})(a+\zeta)- \sem E(\semOp{\mathbf{Q}_0})(a\referee{-\zeta})\big)\Pi_{\Lambda}\\
&=\Pi_\Lambda\big(\sem E(\semOp{\mathbf{Q}_0})(a+\zeta)- \sem E(\semOp{\mathbf{Q}_0})(a\referee{-\zeta})\big).
\end{align*} 
In particular, 
$$
\sem E(\semOp{\mathbf{Q}_0})(a+\zeta)- \sem E(\semOp{\mathbf{Q}_0})(a\referee{-\zeta})- \Pi_\Lambda
$$
is an orthogonal projector and thus a positive operator. Letting $\Pi_\Lambda(x,y)$ denote the integral kernel of $\Pi_\Lambda$, we then have
$$
\sum_{\lambda_{\alpha}\in [a,a+\zeta]}|u_{\lambda_{\alpha}}(x)|^2=\Pi_\Lambda(x,x)\leq \sem E(\semOp{\mathbf{Q}_0})(x,x,a+\zeta)- \sem E(\semOp{\mathbf{Q}_0})(x,x,a\referee{-\zeta}).
$$}
Next, by Lemma~\ref{l:derivativesOmega} with $\alpha=\beta=0$,
$$
 \sem{E}(\semOp{\mathbf{Q}_0})(x,y,\omega) = f(x,y,\omega)+O(\hbar^\infty)
$$
with 
$$
|\partial_\omega^\ell  f(x,y,\omega)|\leq C_{\ell}\hbar^{-\ell},\qquad \ell\geq 1.
$$
In particular, by the mean value theorem, for all $N>0$, there is $C_N>0$ such that 
\begin{align*}
|\sem E(\semOp{\mathbf{Q}_0})(x,x,a+\referee{\zeta})- \sem E(\semOp{\mathbf{Q}_0})(x,x,a\referee{-\zeta})|&\leq  \referee{2}\sup_{\omega\in [a\referee{-\zeta},a+\referee{\zeta}]}|\partial_\omega f_{0,0}(x,x,\omega)||\referee{\zeta}|+C_N\hbar^{N-1}\\
&\leq C\hbar^{-1}(|\referee{\zeta}|+C_N\hbar^N)\leq C_N\hbar^{-1+N}\langle \referee{\zeta}\hbar^{-N}\rangle. 
\end{align*}
\end{proof}

Our next corollary concerns the growth of solutions to 
\begin{equation}
\label{e:basicEq0}
(\semOp{\mathbf{Q}_0}-\omega^2)u_{\omega,\hbar} =0,\qquad \mathbf{Q}_0=\referee{\mathbf{V}^1}\hbar D_x+\hbar D_x\referee{\mathbf{V}^1}+\referee{\mathbf{V}^0}
\end{equation}
that may or may not lie in $L^2$. We first define the energy density at $x$ of $u_\hbar$ by
$$
ED(u_{\omega,\hbar})(x):=|u_{\omega,\hbar}(x)|^2+\hbar^{2}\omega^{-2}|\partial_x u_{\omega,\hbar(x)}|^2.
$$
From now on, we write $u_{\omega,\hbar}=u_{\omega}$, leaving the dependence on $\hbar$ implicit.

We start by considering the case where $\mathbf{Q}\in \sem{Diff}^0$, studying solutions to 
\begin{equation}
\label{e:basicEq}
(\hbar^2D^2+\hbar \referee{\mathbf{V}^0}-\omega^2)u_\omega=0.
\end{equation}
Our first estimate gives a basic understanding of how fast the energy of a solution may change from one point to another
\begin{lemma}
\label{l:basicGrow}
Suppose that $\referee{\mathbf{V}^0}\in L^\infty(\mathbb{R};\mathbb{R})$. Then for any $u_\omega$ solving~\eqref{e:basicEq} and $a,b\in \mathbb{R}$, we have
$$
ED(u_\omega)(b)\leq e^{\|\referee{\mathbf{V}^0}\|_{L^\infty} |a-b|/\omega}ED(u_\omega)(a).
$$
\end{lemma}
\begin{proof}
Observe that 
$$
\omega^{-1}\hbar D_x \begin{pmatrix} u_\omega\\ \omega^{-1}\hbar D_x u_\omega\end{pmatrix}=\begin{pmatrix} 0& 1\\1-\omega^{-2}\hbar \referee{\mathbf{V}^0}(x)&0\end{pmatrix}\begin{pmatrix} u_\omega\\ \omega^{-1}\hbar D_x u_\omega\end{pmatrix}=:A(x)\begin{pmatrix} u_\omega\\ \omega^{-1}\hbar D_x u_\omega\end{pmatrix}.
$$

Therefore,
\begin{align*}
\hbar\partial_x ED(u_\omega)(x)&= \hbar\partial_x\Big\langle \begin{pmatrix} u_\omega\\ \omega^{-1}\hbar D_x u_\omega\end{pmatrix},\begin{pmatrix} u_\omega\\ \omega^{-1}\hbar D_x u_\omega\end{pmatrix}\Big\rangle\\
&=2\omega\Big\langle \Im A(x)\begin{pmatrix} u_\omega\\ \omega^{-1}\hbar D_x u_\omega\end{pmatrix},\begin{pmatrix} u_\omega\\ \omega^{-1}\hbar D_x u_\omega\end{pmatrix}\Big\rangle.
\end{align*}
Now, 
$$
\|\Im A(x)\|=\Big\|\frac{\hbar\omega^{-2}}{2i}\begin{pmatrix} 0&\referee{\mathbf{V}^0}(x)\\-\referee{\mathbf{V}^0}(x)&0\end{pmatrix}\Big\|\leq \frac{\hbar}{2\omega^2}\|\referee{\mathbf{V}^0}\|_{L^\infty}.
$$
In particular, 
$$
\partial_x ED(u_\omega)(x)\leq \omega^{-1} \|\referee{\mathbf{V}^0}\|_{L^\infty}ED(u_\omega)(x),
$$
and hence, by Gr\"onwall's inequality,
$$
ED(u_\omega)(b)\leq e^{\|\referee{\mathbf{V}^0}\|_{L^\infty}|a-b|/\omega}ED(u_\omega)(a).
$$
\end{proof}

Our next lemma allows us to glue solutions of~\eqref{e:basicEq} together.
\begin{lemma}
\label{l:glueSol}
Suppose that $\referee{\mathbf{V}^0_L},\referee{\mathbf{V}^0_R}\in L^\infty$, $\supp \referee{\mathbf{V}^0_L}\subset (-\infty,0)$, and $\supp \referee{\mathbf{V}^0_R}\subset (0,\infty)$, and that $u^{L}_{\omega}$, $u_{\omega}^R$ are real valued and solve~\eqref{e:basicEq}  with $\referee{\mathbf{V}^0}=\referee{\mathbf{V}^0_L}$ or $\referee{\mathbf{V}^0}=\referee{\mathbf{V}^0_R}$ respectively. Then there is $0\leq s<2\pi\hbar/\omega$ such that, putting
$$
\referee{\mathbf{V}^0}(x)=\begin{cases} \referee{\mathbf{V}^0_L}(x)&x\leq 0,\\0&0< x< s,\\\referee{\mathbf{V}^0_R}(x-s)&s\leq x,\end{cases}
$$
there is a solution, $v_\omega:\mathbb{R}\to \mathbb{R}$, to 
$$
(\hbar^2 D^2+ \hbar^2 \referee{\mathbf{V}^0}-\omega^2)v_\omega=0
$$
with
$$
 v_\omega(x)=\begin{cases} \sqrt{ED(u_\omega^R)(0)}u_\omega^L(x)&x\leq 0\\  \sqrt{ED(u_\omega^L)(0)}u_\omega^R(x-s)&s< x.\end{cases}
$$
\end{lemma}
\begin{proof}
Since $\referee{\mathbf{V}^0_L}\equiv 0$ on $[0,\infty)$, we have that 
$$
u_\omega^L(x)=((A^L+iB^L)e^{ix\omega /\hbar}+(A^L-iB^L)e^{-ix\omega/\hbar}),\qquad x\geq 0,
$$
with  $A^L,B^L\in \mathbb{R}$,
$
|A^L|^2+|B^L|^2=ED(u_\omega^L)(0).
$
Similarly,
$$
u_{\omega}^R(x)=((A^R+iB^R)e^{ix\omega/\hbar}+(A^R-iB^R)e^{-ix\omega/\hbar}),\qquad x\leq 0
$$
with $|A^R|^2+|B^R|^2=ED(u_\omega^R)(0).$

To complete the proof, we need only find $0\leq s<2\pi \hbar/\omega$ such that 
\begin{gather}
\sqrt{(|A^L|^2+|B^L|^2)}(A^R+iB^R)e^{-is\omega/\hbar}=\sqrt{(|A^R|^2+|B^R|^2)}(A^L+iB^L),\label{e:match1}\\
\sqrt{(|A^L|^2+|B^L|^2)}(A^R-iB^R)e^{is\omega/\hbar}=\sqrt{(|A^R|^2+|B^R|^2)}(A^L-iB^L)\label{e:match2}.
\end{gather}
Since the absolute values of the left and right hand sides above agree, it is easy to see that there is $s\in [0,2\pi\hbar/\omega)$ such that~\eqref{e:match1} holds. But then~\eqref{e:match2} also holds by taking the conjugate of both sides.
\end{proof}

\begin{lemma}
\label{l:realEnergyChange}
For all $N$ there is $c_N>0$ and $\hbar_0>0$ such that for all $0<\hbar<\hbar_0$, all $0<|a-b|<c_N\hbar^{-N}$,  and all solutions, $u_\omega$, to~\eqref{e:basicEq} we have
\begin{equation*}
e^{-\omega^{-1}\|\referee{\mathbf{V}^0}\|_{L^\infty}-\hbar^N}ED(u_\omega)(b)\leq ED(u_\omega)(a)\leq e^{\omega^{-1}\|\referee{\mathbf{V}^0}\|_{L^\infty}+\hbar^N}ED(u_\omega)(b).
\end{equation*}
\end{lemma}
\begin{proof}
The proof is trivial if $ED(u_{\omega})(b)=0$ since then $u_{\omega}\equiv 0$. Therefore, we may assume $ED(u_\omega)(b)\neq 0$. 

Suppose that 
\begin{equation}
\label{e:energyChange}
\frac{ED(u_\omega)(a)}{ED(u_\omega)(b)}> (1+\frac{1}{\beta})e^{2\|\referee{\mathbf{V}^0}\|_{L^\infty}/\omega},\qquad \beta := \frac{1}{e^{\omega^{-1}\|\referee{\mathbf{V}^0}\|_{L^\infty}+\hbar^N}-1}.
\end{equation}
Let $\cutMan \in C_c^\infty(\mathbb{R})$ with $\cutMan \equiv 1$ on $[a,b]$, $\supp \cutMan \subset (a-1,b+1)$, and put $\tilde{V}_0(x):=\cutMan(x)\referee{\mathbf{V}^0}(x).$ Let $f_{\omega}$ be real valued and solve
$$
(-\hbar^2 D^2+\hbar \tilde{V}_0-\omega^2)f_{\omega}=0, \qquad f_{\omega}(a)=u_\omega(a),\quad\partial_xf_{\omega}(a)=\partial_xu_\omega(a).
$$
Then, by Lemma~\ref{l:basicGrow} together with~\eqref{e:energyChange},
$$
\frac{ED(f_{\omega})(a-1)}{ED(f_{\omega})(b+1)}\geq 1+\frac{1}{\beta}=e^{\omega^{-1}\|\referee{\mathbf{V}^0}\|_{L^\infty}+\hbar^N}.
$$
Next,  by Lemma~\ref{l:glueSol}, there is $s\in [0,2\pi \hbar/\omega)$ such that 
$$
\sqrt{ED(f_{\omega})(a-1)}\begin{pmatrix} f_{\omega}(b+1+s)\\\partial_xf_{\omega}(b+1+s)\end{pmatrix}=\sqrt{ED(f_{\omega})(b+1)}\begin{pmatrix} f_{\omega}(a-1)\\\partial_xf_{\omega}(a-1)\end{pmatrix}.
$$

Therefore, putting
$$
F_+(x)=f_{\omega}(x)1_{(-\infty, b+1+s)}(x)+\sum_{k\geq 1} \Big[\frac{ED(f_{\omega})(b+1)}{ED(f_{\omega})(a-1)}\Big]^{\frac{k}{2}}(f_{\omega}1_{[a-1,b+1+s)})(x-k(b-a+2+s)),
$$
we have
$$
(\hbar^2D^2+\hbar^2V_+-\omega^2)F_+(x)=0, \qquad V_+(x)= \sum_{k\geq 0} \referee{\mathbf{V}^0}(x-k(b-a+2+s)).
$$
Notice that $F_+$ is a linear combination of shifted pieces of $f_\omega$.

Similarly, letting $F_-(x)=F_+(-x)$, $F_-$ solves
$$
(\hbar^2D^2+\hbar^2V_+(-x)-\omega^2)F_-=0,
$$
and hence, we may find $s_-\in[0,2\pi\hbar/\omega)$ such that 

$$
\sqrt{ED(F_+)(a-1)}\begin{pmatrix} F_-(-a +1+s_-)\\\partial_xF_-(-a+1+s_-)\end{pmatrix}=\sqrt{ED(F_-)(-a+1+s_-)}\begin{pmatrix} F_+(a-1)\\\partial_x F_+(a-1)\end{pmatrix}.
$$

In particular, letting 
\begin{gather*}
V=V_+(x)1_{[(a-1,\infty)} +V_+(-x+s_-)1_{(-\infty,a-1]}(x),\\
 F=F_+1_{[a-1,\infty)}+F_-(-x+s_-)1_{(-\infty,a-1]}(x),
 \end{gather*}
we have that 
$$
(\hbar^2D^2+\hbar^2V-\omega^2)F=0
$$
and
$$
\|F\|^2_{L^2}=\|F_+\|^2_{L^2(a-1,\infty)}+\|F_-\|^2_{L^2(-\infty,-a+1+s_-)}.
$$
Now, 
$$
\|F_+\|^2_{L^2(a-1,\infty)}=\sum_{k=0}^\infty \|f_{\omega}\|_{L^2(a-1,b+1+s)}^2\Big[\frac{ED(f_{\omega})(b+1)}{ED(f_{\omega})(a-1)}\Big]^k\leq \beta \|f_{\omega}\|_{L^2(a-1,b+1+s)}^2
$$
and 
\begin{align*}
\|F_-\|_{L^2(-\infty,-a+1+s_-)}^2&=\|F_-\|_{L^2(-\infty,-a+1)}^2+\|F_-\|^2_{L^2(-a+1,-a+1+s_-)}\\
&=\|F_+\|_{L^2(a-1,\infty)}^2+\|F_-\|_{L^2(-a+1,-a+1+s_-)}^2\\
&\leq  \beta \|f_{\omega}\|_{L^2(a-1,b+1+s)}^2 +\tfrac{2\pi \hbar}{\omega} ED(f_{\omega})(a-1)\\
&\leq \beta\|f_{\omega}\|_{L^2(a-1,b+1+s)}^2+\tfrac{2\pi \hbar}{\omega}\int_{a}^{a+1} ED(f_{\omega})(s)e^{\omega^{-1}\|\referee{\mathbf{V}^0}\|_{L^\infty}|s-a+1|}ds \\
&\leq \beta \|f_{\omega}\|_{L^2(a-1,b+1+s)}^2 +\hbar C_\omega e^{\omega^{-1}\|\referee{\mathbf{V}^0}\|_{L^\infty}}\|f_{\omega}\|_{L^2(a-1,b+1+s)}^2  .
\end{align*}
Therefore, 
$$
(b+2\pi \hbar/\omega+2-a)^{-1}\|f_{\omega}\|^2_{L^2(a-1,b+1+s)}\leq \|F\|^2_{L^\infty}\leq C_N\hbar^{3N}(2\beta+C_\omega he^{\omega^{-1}\|\referee{\mathbf{V}^0}\|_{L^\infty}})\|f_{\omega}\|^2_{L^2(a-1,b+1+s)}.
$$
This is a contradiction unless $|b-a+3|\geq c_N\hbar^{-2N}$. The argument for $\frac{ED(u_{\omega})(a)}{ED(u_{\omega})(b)}<(1-\frac{1}{1+\beta})e^{-2\omega^{-1}\|V\|_{L^\infty}}$ is identical.
\end{proof}

We now extend the results of Lemma~\ref{l:realEnergyChange} to the case of non-zero $\referee{\mathbf{V}^1}$ i.e. perturbations which have a first order term.
\begin{theorem}
For all $N>0$, there are $c_N>0$, such that for $0<\hbar<1$,  $\|\referee{\mathbf{V}^1}\|_{L^\infty}\leq \frac{1}{4} \omega \hbar^{-1}$, $u_\omega$ any solution to~\eqref{e:basicEq}, and $|a-b|<c_N\hbar^{-N}$, we have
\begin{multline*}
ED(u_\omega)(b)e^{-\omega^{-1}(\|\referee{\mathbf{V}^0}\|_{L^\infty}+4\hbar  \|\referee{\mathbf{V}^1}\|_{L^\infty}+\hbar\|\referee{\mathbf{V}^1}\|_{L^\infty}^2)-\hbar^N}\leq ED(u_\omega)(a)\\\leq e^{\omega^{-1}(\|\referee{\mathbf{V}^0}\|_{L^\infty}+4\hbar\|\referee{\mathbf{V}^1}\|_{L^\infty}+\hbar\|\referee{\mathbf{V}^1}\|_{L^\infty}^2)+\hbar^N}ED(u_\omega)(b).
\end{multline*}
\end{theorem}
\begin{proof}
Suppose that $u_\omega$ solves~\eqref{e:basicEq0}. Then $v=e^{-i\int_0^x \referee{\mathbf{V}^1}(s)ds}u_\omega$ solves~\eqref{e:basicEq} with $\referee{\mathbf{V}^0}$ replaced by $\referee{\mathbf{V}^0}-\hbar (\referee{\mathbf{V}^1})^2$. 
Since $\referee{\mathbf{V}^0}$ and $\referee{\mathbf{V}^1}$ are real valued, $\Re v$ and $\Im v$ solve the same equation as $v$. In particular, by Lemma~\ref{l:realEnergyChange}, 
\begin{equation*}
\begin{gathered}
e^{-\omega^{-1}(\|\referee{\mathbf{V}^0}\|_{L^\infty}+\hbar \|\referee{\mathbf{V}^1}\|_{L^\infty}^2)-\hbar^N}ED(\Re v)(b)\leq  ED(\Re v)(a)\leq e^{\omega^{-1}(\|\referee{\mathbf{V}^0}\|_{L^\infty}+\hbar \|\referee{\mathbf{V}^1}\|_{L^\infty}^2)+\hbar^N}ED(\Re v)(b),\\
e^{-\omega^{-1}(\|\referee{\mathbf{V}^0}\|_{L^\infty}+\hbar \|\referee{\mathbf{V}^1}\|_{L^\infty}^2)-\hbar^N}ED(\Im v)(b)\leq ED(\Im v)(a)\leq e^{\omega^{-1}(\|\referee{\mathbf{V}^0}\|_{L^\infty}+\hbar \|\referee{\mathbf{V}^1}\|_{L^\infty}^2)+\hbar^N}ED(\Im v)(b).
\end{gathered}
\end{equation*}

From this, it easily follows that
\begin{equation*}
e^{-\omega^{-1}(\|\referee{\mathbf{V}^0}\|_{L^\infty}+\hbar \|\referee{\mathbf{V}^1}\|_{L^\infty}^2)-\hbar^N}ED( v)(b)\leq ED( v)(a)\leq e^{\omega^{-1}(\|\referee{\mathbf{V}^0}\|_{L^\infty}+\hbar \|\referee{\mathbf{V}^1}\|_{L^\infty}^2)+\hbar^N}ED( v)(b).
\end{equation*}

Next, observe that since $\hbar\omega^{-1}\|\referee{\mathbf{V}^1}\|_{L^\infty}\leq \e\referee{\leq\frac{1}{4}}$, we have
\begin{align*}
 (1-\tfrac{5}{4}\hbar \omega^{-1}\|\referee{\mathbf{V}^1}\|_{L^\infty})ED(v)(x)&\leq  (1-\hbar^2\omega^{-2}\|\referee{\mathbf{V}^1}\|_{L^\infty}^2-\hbar \omega^{-1}\|\referee{\mathbf{V}^1}\|_{L^\infty})ED(v)(x)\\
 &\leq ED(u)(x)\\
 &\leq  (1+\hbar^2\omega^{-2}\|\referee{\mathbf{V}^1}\|_{L^\infty}^2+\hbar \omega^{-1}\|\referee{\mathbf{V}^1}\|_{L^\infty})ED(v)(x)\\
 &\leq  (1+\tfrac{5}{4}\hbar \omega^{-1}\|\referee{\mathbf{V}^1}\|_{L^\infty})ED(v)(x).
\end{align*}
Finally, the fact that
 $$
 \referee{\frac{1+s}{1-s}\leq e^{2s},\, 0<s<1/2, \qquad \big(s:=\frac54 h\omega^{-1}\|V^1\|_\infty\big)}
 $$ 
 complete the proof.

\end{proof}

\appendix

\section{Apriori computation of the first asymptotic terms}
\label{a:firstTerm}

\begin{lemma}
\label{l:firstTerm}
For $j=0,1$ let $V_j\in \sem{Diff}^0$ and define $\mathbf{Q}:=\referee{\mathbf{V}^1}\hbar D_x+\hbar D_x \referee{\mathbf{V}^1} +\referee{\mathbf{V}^0}\in \sem{Diff}^1$. Let $0<a<b$. Then, for all $\referee{\mathcal{K}}\Subset \mathbb{R}$, $\e>0$, there is $\hbar_\e>0$ such that for all $x,y \in \referee{\mathcal{K}}$ with $|x-y|\geq \hbar^{1-\e}$, $0<\hbar<\hbar_\e$ and $\omega\in[a,b]$,  
$$
\Big|\sem{E}(\semOp{\mathbf{Q}})(\omega\,;\,x,x)-\frac{\omega}{\pi\hbar}\Big|\leq C\hbar^{-3/4},\qquad \Big|\sem{E}(\semOp{\mathbf{Q}})(\omega;x,y )-\frac{\sin(\omega|x-y|/\hbar)}{\pi|x-y|}\Big|\leq \e.
$$
\end{lemma}
\begin{remark}
When $\referee{\mathbf{V}^0}=\hbar \tilde{V}_0$ with $\tilde{V}_0\in C_c^\infty$ and $\referee{\mathbf{V}^1}\in C_c^\infty$, Lemma~\ref{l:firstTerm} is well known and can be recovered e.g. from~\cite{PoSh:83,Va:83,Va:84,Va:85}. The semiclassical version, when $V_j$ have compact support, can be obtained from well known formulae in scattering theory (see e.g.~\cite[Lemma 3.6]{DyZw:19}). However, when $V_j$ are only $\USB$, we are not aware of an appropriate reference for these formulae. Here, we use Proposition~\ref{p:distantPerturb} to obtain these formulae from those for compactly supported perturbations. 
\end{remark}

\begin{proof}
\referee{Let} $\e>0$, and $\cutMan \in C_c^\infty((-2,2))$ with $\cutMan\equiv 1$ on $[-1,1]$. Then, define $\cutMan_\e(x)=\cutMan(\e x)$ and 
$$
\sem{H}_\e:=\semOp{\mathbf{Q}_\e},\qquad \mathbf{Q}_\e:=\cutMan_\e \referee{\mathbf{V}^1}\hbar D_x+\hbar D_x(\cutMan_\e \referee{\mathbf{V}^1})+\cutMan_\e \referee{\mathbf{V}^0}.
$$
Then, by e.g.~\cite{Ga:20, Va:84, PoSh:83},  (see~\cite[(6.5)]{Ga:20}, and note that the gauge transform procedure in that paper is unnecessary, since the perturbation is compactly supported)
\begin{align*}
\sem E(\sem H_\e)(x,y,\omega)&=\Big\langle \frac{1}{2\pi\hbar}\int_{-\infty}^{\omega^2}\int\hat{\cutEnergyFourier}(t)\cutMan_1e^{ \frac{i}{\hbar}t(\mu-\sem{H}_\e)}\cutFreq(\hbar D)\cutMan_1dtd\mu  \delta_y,\delta_x\Big\rangle,
\end{align*}
where $\cutMan_1\in C_c^\infty(\mathbb{R})$ with $\cutMan_1 \equiv 1$ on $\pi_L(\referee{\mathcal{K}})\cup \pi_{R}(\referee{\mathcal{K}})$, with $\pi_{L/R}:\mathbb{R}^2\to \mathbb{R}$ the natural projections and $\cutFreq\in C_c^\infty(\mathbb{R})$ with $\cutFreq \equiv 1$ on $[-3b,3b]$, and $\hat{\cutEnergyFourier}\in C_c^\infty$ with $\hat{\cutEnergyFourier}\equiv 1$ on $|t|\leq 2\diam (\pi_L(\referee{\mathcal{K}})\cup\pi_R(\referee{\mathcal{K}}))$. For any $T>0$ and $|t|\leq T$ we have that the kernel of $\cutMan_1e^{- \frac{i}{\hbar}t\sem{H}_\e}\cutFreq(\hbar D)\cutMan_1$ is given by
$$
\frac{1}{2\pi \hbar} \int e^{\frac{i}{\hbar}((x-y)\xi-t|\xi|^2)}a_\e(t,x,y,\xi)d\xi,
$$
where $a_\e\in \Spz^{\comp}$, $a_\e(0,x,y,\xi)=\cutMan_1(x)\cutMan_1(y)\cutFreq(\xi)$.

We start by computing $\sem{E}(\semOp{\mathbf{Q}_\e})(x,x,\omega)$ for any $x\in \referee{\mathcal{K}}$. Let $\auxCut\in C_c^\infty((-2,2))$ with $\auxCut\equiv 1$ near $[-1,1]$. Then, for $\hbar^{1/2}<\omega <3b$, we have
\begin{align*}
\partial_\omega \sem E(\semOp{\mathbf{Q}}_\e)(x,x,\omega)&= \frac{2 \omega}{(2\pi \hbar)^2}\int \hat{\cutEnergyFourier}(t)e^{\frac{i}{\hbar}t(\omega^2 -|\xi|^2)}a_\e(t,x,x,\xi)d\xi dt\\
&=\frac{2 \omega^2}{(2\pi \hbar)^2}\int \hat{\cutEnergyFourier}(t)e^{\frac{i\omega^2}{\hbar}t(1 -|\eta|^2)}a_\e(t,x,x,\omega \eta)d\eta dt\\
&=\frac{2 \omega^2}{(2\pi \hbar)^2}\int \hat{\cutEnergyFourier}(t)e^{\frac{i\omega^2}{\hbar}t(1 -|\eta|^2)}a_\e(t,x,x,\omega \eta)\auxCut(\eta)d\eta dt+O(\hbar^{-1}(\hbar\omega^{-2})^\infty).
\end{align*}
Performing stationary phase in $(t,\eta)$, we obtain 
\begin{equation}
\label{e:derivativeHigh}
\partial_\omega \sem E(\sem H_\e)(x,x,\omega)\sim\frac{1}{\pi \hbar}+ 
\sum_{j\geq 0}c_{\e,j}(x,\omega)\hbar^{j}\omega^{-2(j+1)}+O(\hbar^{-1}(\hbar\omega^{-2})^\infty),\qquad \hbar^{1/2}\leq \omega \leq 3b.
\end{equation}

Next, we estimate $\sem E(\semOp{\mathbf{Q}}_\e)(x,x,M\hbar^{1/2})$. We have
$$
\sem E(\semOp{\mathbf{Q}_\e})(x,x,M\hbar^{1/2})=\frac{1}{(2\pi \hbar)^2}\int_{-\infty}^{M\hbar^{1/2}} \int \hat{\cutEnergyFourier}(t)e^{\frac{i}{\hbar}t(\mu-|\xi|^2)}a_\e(t,x,x,\xi)d\xi dtd\mu.
$$
Integration by parts in $t$ then shows that 
\begin{align*}
|\sem E(\semOp{\mathbf{Q}_\e})(x,x,M\hbar^{1/2})|&\leq C_N \hbar^{-2}\int_{-\infty}^{M\hbar^{1/2}} \int_{-3b}^{3b}\int \langle t\rangle^{-N} \langle \hbar^{-1}(\mu-|\xi|^2)\rangle^{-N}dt d\xi d\mu\\
&\leq C_N\hbar ^{-2}\int_{-\infty}^{M\hbar^{1/2}} \int_{-3b}^{3b} \langle \hbar^{-1}(\mu-|\xi|^2)\rangle^{-N}d\xi d\mu \leq C_M\hbar^{-3/4}.
\end{align*}

Therefore, by~\eqref{e:derivativeHigh}, we have
\begin{equation}
\label{e:onDiagonalAppend}
\begin{aligned}
\sem{E}(\semOp{\mathbf{Q}_\e})(x,x,\omega)&=\int_{M\hbar^{1/2}}^\omega \partial_\omega \sem E(\semOp{\mathbf{Q}}_\e)(x,x,s)ds+O_M(\hbar^{-3/4})\\
&=\frac{\omega}{\pi \hbar} +O_M(\hbar^{-3/4}).
\end{aligned}
\end{equation}

Using~\eqref{e:derivativeHigh} again, we have, for $\omega\in [a/2,2b]$ and $|s|\leq \frac{a}{4}$,
\begin{equation}
\label{e:onDiagLip}
|\sem E(\sem H_\e)(x,x,\omega)-\sem E(\sem H_\e)(x,x,\omega-s)|\leq (\frac{1}{\pi \hbar}+O_\e(1))|s|.
\end{equation}

Therefore, using Proposition~\ref{p:distantPerturb}, with $\delta=0$, $T=\e^{-1}$,  $C_1>0$ bounded uniformly in $0<\hbar<\hbar_\e$, we have
$$
\Big|\sem E (\semOp{\mathbf{Q}})(x,x,\omega)-\sem E(\semOp{\mathbf{Q}_\e})(x,x,\omega)\Big|\leq \e \quad \text{ for }\omega\in[a,b].
$$
In particular, by~\eqref{e:onDiagonalAppend}, we have
$$
\Big|\sem E (\semOp{\mathbf{Q}})(x,x,\omega)-\frac{\omega }{\pi \hbar}\Big|\leq C\hbar^{-\frac{3}{4}}.
$$
This completes the proof of the first part of the lemma.

We now proceed to the off-diagonal part. For $\delta>0$, to be \referee{chosen} later, let $\auxCut\in C^\infty(-2\delta,\infty)$ with $\auxCut\equiv 1$ on $[-\delta,\infty)$,
\begin{align*}
&\sem E(\semOp{\mathbf{Q}_\e} )(x,y,\omega)\\
&= \frac{1}{(2\pi \hbar)^2}\int_{-\infty}^{\omega^2}\int_{-\infty}^{\infty}\int_{-\infty}^{\infty}  e^{\frac{i}{\hbar} (t(\mu-|\xi|^2)+(x-y)\xi)}\hat{\cutEnergyFourier}(t)a_\e(t,x,y,\xi)d\xi dt d\mu\\
&=\frac{1}{(2\pi \hbar)^2}\int_{-\infty}^{\omega^2}\int_{-\infty}^{\infty}\int_{-\infty}^{\infty} e^{\frac{i}{\hbar} (t(\mu-|\xi|^2)+(x-y)\xi)}\hat{\cutEnergyFourier}(t)a_\e(t,x,y,\xi) \auxCut(|x-y|\mu)d\xi dt d\mu +O_\e((\hbar|x-y|^{-1})^\infty)\\
&=\frac{|x-y|}{(2\pi\hbar)^2}\int_{-\infty}^{\omega^2}\int_{-\infty}^{\infty}\int_{-\infty}^{\infty}  e^{\frac{i}{\hbar}|x-y| (s(\mu-|\xi|^2)+\frac{x-y}{|x-y|}\xi)}\hat{\cutEnergyFourier}(s|x-y|)a_\e(s|x-y|,x,y,\xi) \auxCut(\mu|x-y|)d\xi ds d\mu \\
&\qquad+O_\e((\hbar|x-y|^{-1})^{\infty}).
\end{align*}
To obtain the second line, we integrate by parts in $t$ and use that $\mu<-\delta|x-y|$ implies $|\mu-|\xi|^2|\geq c_\delta|x-y|\langle |\mu|+|\xi|^2\rangle$. In the third line, we changed variables $t=s|x-y|$.

Now, performing stationary phase in the $(s,\xi)$ variables, we see that stationary points do not exist for $|\mu| \leq \delta |x-y|$ (provided we have chosen $\delta$ small enough). We then obtain
\begin{multline*}
\sem E(\sem H_\e)(x,y,\omega)=\frac{1}{2\pi \hbar}\int_{-\infty}^{\omega^2}\sum_{\pm}\frac{1}{2\sqrt{|\mu|}}e^{\pm \frac{i}{\hbar}(x-y)\sqrt{\mu}}\big(
\hat{\cutEnergyFourier}(\pm\frac{x-y}{2\sqrt{\mu}}) \auxCut(\mu|x-y|) \\+\auxCut(\mu|x-y|)\hat{\cutEnergyFourier}_1(\pm \frac{x-y}{2\sqrt{\mu}})O_\e(\hbar|x-y|^{-1})\big) d\mu +O_\e((\hbar|x-y|^{-1})^{\infty}),
\end{multline*}
where $\hat{\cutEnergyFourier}_1\in C_c^\infty$ with $\hat{\cutEnergyFourier}_1\equiv 1$ on $\supp \cutEnergyFourier$. Choosing $\delta>0$ small enough, we have $\hat{\cutEnergyFourier}_1(t)\equiv 0$ on $|t|\leq \delta$, and hence, the integrand is supported in $\mu\geq \delta|x-y|$. Finally,  changing variables to $s=\sqrt{\mu}$ and integrating by parts once implies
$$
\sem E(\sem H_\e)(x,y,\omega) =\sum_{\pm}\pm\frac{1}{2\pi i(x-y)}e^{\pm \frac{i}{\hbar}(x-y)\omega}+O_\e(\hbar|x-y|^{-1})=\frac{\sin (\omega |x-y|/\hbar)}{\pi |x-y|}+O_\e(\hbar |x-y|^{-1})
$$
and has a full asymptotic expansion in powers of $\hbar|x-y|^{-1}$.

In particular, using this together with~\eqref{e:onDiagLip}, we see that there is $\hbar_\e>0$ and $M_\e>0$ such that the hypotheses of Proposition~\ref{p:distantPerturb} hold for $0<\hbar<\hbar_\e$, and $|x-y|>M_\e\hbar$ with $\sem{H}_1=\semOp{\mathbf{Q}_\e}$, $T(\hbar)=\e^{-1}/2$, and $R_0=\diam(\referee{\mathcal{K}})$, and $\delta(\hbar)=0$. Therefore, 
$$
\Big|\sem E(\semOp{\mathbf{Q}})(x,y,\omega) -\sem E(\semOp{\mathbf{Q}_\e} )(x,y,\omega)\Big|\leq C\e,\qquad 0<\hbar<\hbar_\e,\,\omega\in[a,b].
$$
In particular, this implies
$$
\Big|\sem{E}(\semOp{\mathbf{Q}})(\omega\,;\,x,y)-\frac{\sin(\omega|x-y|/\hbar)}{\pi|x-y|}\Big|\leq \e.
$$
for all $0<\hbar<\hbar_\e$, $|x-y|>M_\e \hbar$, and $\omega\in[a,b]$.  This completes the proof of the lemma.
\end{proof}

\section{Proof of Lemma~\ref{l:glue}}
\label{s:glue}
\referee{
First, factoring out $e^{\frac{i}{\hbar}\frac{\xi_1+\xi_2}{2}}$ and introducing 
$$
\tilde a_{j,n(\hbar)}:=i(a_{j,n(\hbar)}-b_{j,n(\hbar)}),\ \ \ \tilde b_{j,n(\hbar)}:=a_{j,n(\hbar)}+b_{j,n(\hbar)},
$$
we can rewrtite \eqref{preasymp} as
\begin{equation}\label{preasymp1}
g(\hbar):=e^{-\frac{i}{\hbar}\frac{\xi_1+\xi_2}{2}}f(\hbar)= \sin(\frac{\xi_1-\xi_2}{2\hbar})\sum_{j=0}^{N} \tilde a_{j,n(\hbar)} \hbar^{jp} + \cos(\frac{\xi_1-\xi_2}{2\hbar})\sum_{j=0}^{N} \tilde b_{j,n(\hbar)}  \hbar^{jp} +O(\lp_{n(\hbar)}^{-M}),
\end{equation}
for 
$$
-10\leq n(\hbar)+\log_2\hbar\leq 10,
$$
with $\tilde a_{j,n},\tilde b_{j,n}\in \mathbb{C}$, $j=0,1,\dots,$ satisfying
\begin{equation}\label{asympest1}
|\tilde a_{j,n}|+|\tilde b_{j,n}|\leq C_{j}\lp_n^{jp(1-\iota)}.
\end{equation}
For $-9\leq n(\hbar)+\log_2\hbar\leq 9$ we can use \eqref{preasymp1} with $n$ and $n-1$. Subtracting one from another we get
$$
\sin(\frac{\xi_1-\xi_2}{2\hbar})\sum_{j=0}^{N} \hat t_{j,n(\hbar)} \hbar^{jp} + \cos(\frac{\xi_1-\xi_2}{2\hbar})\sum_{j=0}^{N} \check t_{j,n(\hbar)}  \hbar^{jp} +O(\lp_{n(\hbar)}^{-M}),
$$
where $\hat t_{j,n}=\tilde a_{j,n}-\tilde a_{j,n-1}$ and $\check t_{j,n}=\tilde b_{j,n}-\tilde b_{j,n-1}$.
\begin{proposition}
For each $j=0,1,\dots,N$, we have:
\begin{equation}\label{asympest2}
\hat t_{j,n}=O(\lp_{n(\hbar)}^{jp-M}),\ \ \ \check t_{j,n}=O(\lp_{n(\hbar)}^{jp-M}).
\end{equation}
\end{proposition}
\begin{proof}
Put
$$
s:=(\hbar\lp_{n(\hbar)})^{-1},\ \ \ \hat \tau_{j,n}:=\hat t_{j,n}\lp_{n(\hbar)}^{M-jp},\ \ \ \check \tau_{j,n}:=\check t_{j,n}\lp_{n(\hbar)}^{M-jp}.
$$
Then
\begin{equation}\label{asymp2}
\sin(\frac{\xi_1-\xi_2}{2}\lp_{n(\hbar)}s)\sum_{j=0}^{N} \hat \tau_{j,n(\hbar)} s^{-jp} + \cos(\frac{\xi_1-\xi_2}{2}\lp_{n(\hbar)}s)\sum_{j=0}^{N} \check \tau_{j,n(\hbar)}  s^{-jp} =O(1),
\end{equation}
whenever $2^{-9}<s<2^9$. Now, we choose $2N+2$ points in the following way. Assume for definiteness that $\xi_1>\xi_2$. We put
$$
s_l:=\frac{4\pi}{(\xi_1-\xi_2)\lp_{n(\hbar)}}\left(\left[\frac{(\xi_1-\xi_2)\lp_{n(\hbar)}}{4\pi}\right]+l\left[\frac{(\xi_1-\xi_2)\lp_{n(\hbar)}}{4\pi\cdot 2N}\right]\right),\ \ \ l=0,\dots,N,
$$
so that $\sin(\frac{\xi_1-\xi_2}{2}\lp_{n(\hbar)}s)=0$ and $\cos(\frac{\xi_1-\xi_2}{2}\lp_{n(\hbar)}s)=1$, and 
$$
s'_l:=\frac{4\pi}{(\xi_1-\xi_2)\lp_{n(\hbar)}}\left(\left[\frac{(\xi_1-\xi_2)\lp_{n(\hbar)}}{4\pi}\right]+l\left[\frac{(\xi_1-\xi_2)\lp_{n(\hbar)}}{4\pi\cdot 2N}\right]\right)+\frac{\pi}{(\xi_1-\xi_2)\lp_{n(\hbar)}},\ \ \ l=0,\dots,N,
$$
so that $\sin(\frac{\xi_1-\xi_2}{2}\lp_{n(\hbar)}s)=1$ and $\cos(\frac{\xi_1-\xi_2}{2}\lp_{n(\hbar)}s)=0$. We also notice that, assuming $\lp_{n(\hbar)}$ is sufficiently large, we have 
$s^{-p}_{l+1}-s^{-p}_l\sim N^{-1}$ and $(s')^{-p}_{l+1}-(s')^{-p}_l\sim N^{-1}$ uniformly in $n$ and $|\xi_1-\xi_2|\sim 1$.

Now, substituting the points $\{s_l,s'_l\}$ into \eqref{asymp2} and using the Cramer's Rule we find that $\hat \tau_{j,n(\hbar)}$ and $\check \tau_{j,n(\hbar)}$ are fractions with the bounded numerator and uniform non-zero denominator (the denominator is a Vandermonde determinant in $s^{-p}$). This proves the proposition.
\end{proof}

Thus, for $j<[Mp^{-1}]$, the series $\sum_{n=n_0}^\infty \hat t_{j,n(\hbar)}$ is absolutely convergent; moreover, for such $j$ we have:
$$
\tilde a_{j,n}=\tilde a_{j,n_0}+\sum_{n=n_0+1}^n \hat t_{j,n}=\tilde a_{j,n_0}+\sum_{n=n_0+1}^\infty \hat t_{j,n}+O(\lp_{n}^{-M+jp})=:\tilde a'_j+O(\lp_{n}^{-M+jp}),
$$
where we have denoted $\tilde a'_j:=\tilde a_{j,n_0}+\sum_{n=n_0+1}^\infty \hat t_{j,n}$. Similarly,
$$
\tilde b_{j,n}=\tilde b_{j,n_0}+\sum_{n=n_0+1}^n \check t_{j,n}=\tilde b_{j,n_0}+\sum_{n=n_0+1}^\infty \check t_{j,n}+O(\lp_{n}^{-M+jp})=:\tilde b'_j+O(\lp_{n}^{-M+jp}),
$$
where we have denoted $\tilde b'_j:=\tilde b_{j,n_0}+\sum_{n=n_0+1}^\infty \check t_{j,n}$.

By \eqref{asympest1} we also have 
$$
\sum_{j=[Mp^{-1}]}^N (|\tilde a_{j,n}|+|\tilde b_{j,n}|)\lp_{n}^{-jp}=O(\lp_{n}^{-\iota M}).
$$
Thus, for $-10\leq n(\hbar)+\log_2\hbar\leq 10$ we have 
$$
g(\hbar)=\sin(\frac{\xi_1-\xi_2}{2\hbar})\sum_{j=0}^{[Mp^{-1}]-1} \tilde a'_{j} \hbar^{jp} + \cos(\frac{\xi_1-\xi_2}{2\hbar})\sum_{j=0}^{[Mp^{-1}]-1} \tilde b'_{j}  \hbar^{jp} +O(\lp_{n(\hbar)}^{-\iota M}).
$$
Since constants in $O$ do not depend on $n$, for all $\hbar<\hbar_0$ we have the expansion
\begin{equation}\label{asymp1}
g(\hbar):=e^{-\frac{i}{\hbar}\frac{\xi_1+\xi_2}{2}}f(\hbar)= \sin(\frac{\xi_1-\xi_2}{2\hbar})\sum_{j=0}^{[Mp^{-1}\iota^{-1}]-1} \tilde a'_{j} \hbar^{jp} + \cos(\frac{\xi_1-\xi_2}{2\hbar})\sum_{j=0}^{[Mp^{-1}\iota^{-1}]-1} \tilde b'_{j}  \hbar^{jp} +O(\hbar^M),
\end{equation}
with some $\tilde a'_{j},\tilde b'_{j}\in \mathbb{C}$, $j=0,1,\dots$. Defining
$$
a'_j:=\frac{\tilde a'_j+i\tilde b'_j}{2i},\ \ \ b'_j:=\frac{-\tilde a'_j+i\tilde b'_j}{2i}
$$
we obtain \eqref{asymp}.}

\bibliographystyle{amsalpha}
\bibliography{biblio}

\end{document}